\documentclass[12pt]{amsart}
\usepackage{amsmath}
\usepackage{amssymb}
\usepackage{graphicx}
\usepackage{pinlabel}
\usepackage[all]{xy}

\usepackage{amsmath,amssymb,amsfonts,amsthm,graphicx}

\usepackage{pinlabel} 

\usepackage[usenames,dvipsnames]{color}

\usepackage[usenames,dvipsnames]{color}

\usepackage{hyperref}

\numberwithin{equation}{section}

\newtheorem{dummy}{dummy}[section]
\newtheorem{lemma}[dummy]{Lemma}
\newtheorem{theorem}[dummy]{Theorem}

\newtheorem{corollary}[dummy]{Corollary}
\newtheorem{proposition}[dummy]{Proposition}
\newtheorem{Q}[dummy]{Question}

\theoremstyle{definition}
\newtheorem{definition}[dummy]{Definition}

\newtheorem{example}[dummy]{Example}
\newtheorem{remark}[dummy]{Remark}

\newcommand{\R}{\mathbb {R}}

\newcommand{\C}{\mathbb {C}}
\newcommand{\Z}{\mathbb {Z}}

\newcommand{\M}{{\mathcal{M}}}

\newcommand{\T}{{\mathcal{T}}}
\newcommand{\e}{\epsilon}
\newcommand{\dd}{\partial}

\newcommand{\alg}{\mathcal{A}}

\newcommand{\fraka}{\mathfrak{a}}
\newcommand{\wt}{\widetilde}
\newcommand{\wh}{\widehat}
\newcommand{\ol}{\overline}

\newcommand{\symp}{\mathit{Symp}}

\newcommand{\dga}{\mathfrak{DGA}}
\newcommand{\leg}{\mathfrak{Leg}}

\newcommand{\sgn}{\mbox{sgn}}

\def\wt#1{\widetilde{#1}}
\def\ol#1{\overline{#1}}

\begin{document}

\title[Functorial LCH for immersed Lagrangian cobordisms]{Functorial LCH for immersed Lagrangian cobordisms}

\author{Yu Pan}

\author{Dan Rutherford}

\begin{abstract}

For $1$-dimensional Legendrian submanifolds of $1$-jet spaces, we extend the functorality of the Legendrian contact homology DG-algebra (DGA) from embedded exact Lagrangian cobordisms, as in \cite{EHK}, to a class of  immersed exact Lagrangian cobordisms by considering their Legendrian lifts as conical Legendrian cobordisms.    
To a conical Legendrian cobordism $\Sigma$ from $\Lambda_-$ to $\Lambda_+$, we associate an immersed DGA map, which is a diagram  
$$\alg(\Lambda_+) \stackrel{f}{\rightarrow} \alg(\Sigma) \stackrel{i}{\hookleftarrow} \alg(\Lambda_-),
$$
where $f$  is  a DGA map and $i$ is an inclusion map.
This construction gives a functor between suitably defined categories of Legendrians with immersed Lagrangian cobordisms and DGAs with immersed DGA maps.  In an algebraic preliminary, we consider an analog of the mapping cylinder construction in the setting of DG-algebras and establish several of its properties.  As an application we give examples of augmentations of Legendrian twist knots that can be induced by an immersed filling with a single double point but cannot be induced by any orientable embedded filling.

\end{abstract}

\maketitle

\section{Introduction}
For a Legendrian submanifold $\Lambda$ in a $1$-jet space $(J^1M, \ker \alpha)$, where $\alpha=dz -ydx$ and $x$ is the coordinate of the 
manifold $M$,
there is a useful invariant, the Legendrian contact homology (LCH) differential graded algebra (DGA) \cite{Che, Eli} $(\alg(\Lambda), \dd)$, associated to $\Lambda$.
  Using $\Z/2$ coefficients, the underlying algebra is non-commutative, generated by the Reeb chords of $\Lambda$ and can be equipped with a $\Z/n$-grading whenever $n$ is a divisor of the Maslov number of $\Lambda$ (which is $|2 \mathit{rot}(\Lambda)|$ when $\dim \Lambda = 1$).   
 The differential $\dd$ is defined by counting rigid holomorphic disks in the symplectization $\mathit{Symp}(J^1M)= (\R_t\times J^1M, d(e^t \alpha))$ with boundary on $\R_t \times \Lambda$.
 According to \cite{Ekh}, an alternative way to  define the differential is to count rigid gradient flow trees (GFTs).
The DGA $(\alg(\Lambda),\dd)$ is invariant up to stable tame isomorphism under Legendrian isotopy.
 
Following the principles of symplectic field theory \cite{EGH}, the construction of the Legendrian contact DGA is functorial with respect to embedded exact Lagrangian cobordisms.  See \cite{E2, EK, EHK}.
For  two Legendrian submanifolds $\Lambda_-$ and $\Lambda_+$ in $J^1M$,
an {\bf embedded exact Lagrangian cobordism} $L$ from $\Lambda_-$ to $\Lambda_+$ is an exact Lagrangian surface in $\mathit{Symp}(J^1M)$ that is cylindrical over $\Lambda_+$ and $\Lambda_-$ near the positive end and the  negative end, respectively.  
An embedded exact Lagrangian cobordism $L$ from $\Lambda_-$ to $\Lambda_+$ induces a DGA map from  $\alg(\Lambda_+)$ to  $\alg(\Lambda_-)$.
Moreover, if the cobordism changes by an exact Lagrangian isotopy, the induced DGA map changes by a  DGA homotopy.
The induced DGA maps are compatible with composition in the sense that the induced DGA map of a concatenation of two cobordisms is DGA homotopic to the composition of the DGA maps induced by each.
In summary, we have two categories:  a category $\leg$ whose objects are Legendrian submanifolds in $J^1M$ and morphisms are embedded exact Lagrangian cobordisms in $\mathit{Symp}(J^1M)$ up to exact Lagrangian isotopy; and a category $\dga$ whose objects are finitely generated, triangular DGAs (see definition in Section \ref{sec:triDGA}) 
and morphisms are DGA maps up to DGA homotopy.
 There is a contravariant functor from $\leg$ to $\dga$  that sends a Legendrian submanifold to its DGA and sends an embedded exact Lagrangian cobordism to the induced DGA map; see \cite{E2, EK, EHK}. 

\begin{remark}\label{Maslov}  Strictly speaking, to have a well defined functor, the objects of $\mathfrak{Leg}$ need to be equipped with some auxiliary data, eg. a choice of regular almost complex structure  or metric depending on whether holomorphic disks or GFTs are used for the differential on $\mathcal{A}(\Lambda)$, as well as a choice of Maslov potential to define the grading.   
Moreover, in order for the induced maps to preserve grading,  a cobordism $L$ should also be equipped with a Maslov potential extending those of $\Lambda_-$ and $\Lambda_+$. 
For a Legendrian submanifold $\Lambda$, a $\Z/n$-grading on the Legendrian contact DGA arises from a choice of $\Z/n$-valued Maslov potential for $\Lambda$, and such a Maslov potential exists only when $n$ divides the Maslov number of each component of $\Lambda$  (see Section \ref{sec:LCH}).
In order to simplify statements, in the remainder of the introduction we consider only Legendrians and Lagrangians of Maslov number $0$, so that all DGAs have $\Z$-gradings that are preserved by induced maps.  The main body of the article addresses the general $\Z/n$-graded case.    
\end{remark}

In this article, we extend the  functoriality of the LCH DGA to a wide class of {\it immersed} exact Lagrangian cobordisms.
We say that an immersed exact Lagrangian cobordism $L$ is a {\bf good Lagrangian cobordism}  if the Legendrian lift  of $L$ in $\mathit{Symp}(J^1M)\times \R$ is an embedded Legendrian surface (see definition in Section \ref{sec:good}).
Equivalently, one can carry $L$ through a symplectomorphism $\Phi: \mathit{Symp}(J^1M) \to T^*(\R_{>0}\times M)$ and lift the image to an embedded Legendrian surface $\Sigma$ in $J^1(\R_{>0}\times M)$.
This Legendrian surface has conical ends over $\Lambda_-$ and $\Lambda_+$ at the negative end and the positive end, respectively,  and is called a {\bf conical Legendrian cobordism} from $\Lambda_-$ to $\Lambda_+$ (see Definition \ref{def:conical}).

We extend the functorality of  the DGA from embedded exact Lagrangian cobordisms to conical Legendrian cobordisms in the following way.

\begin{theorem}\label{thm:main}  Let $\Lambda_-, \Lambda_+ \subset J^1M$ be closed, $1$-dimensional Legendrian submanifolds, and let $\Sigma$ be a conical Legendrian cobordism from $\Lambda_-$ to $\Lambda_+$ equipped with a suitable metric $g$ (see Section \ref{sec:functorI}).
Then, there is a DGA $\alg(\Sigma)$ associated to $\Sigma$ that is invariant up to stable tame isomorphism under conical Legendrian isotopy of $(\Sigma,g)$.
Moreover, 
 the cobordism $\Sigma$ induces two DGA maps $f, i$
 \begin{equation} \label{eq:immDGA}
 \alg(\Lambda_+) \stackrel{f}{\rightarrow} \alg(\Sigma) \stackrel{i}{\hookleftarrow} \alg(\Lambda_-),
 \end{equation}
 where $i$ is induced by an inclusion map on generators.
  The diagram \eqref{eq:immDGA} is called an {\bf immersed DGA map} from   $\alg(\Lambda_+)$ to  $\alg(\Lambda_-)$  (see  Definition \ref{def:immDGAmap}).

Moreover, suppose $\Sigma$ and $\Sigma'$ are two conical Legendrian cobordisms from $\Lambda_-$ to $\Lambda_+$ that are related by a conical Legendrian isotopy. Then, their induced immersed DGA maps are {\bf immersed homotopic} (see Definition \ref{def:homotopy}),
i.e. there is a stable tame isomorphism
$\varphi: \alg(\Sigma)*S\to \alg(\Sigma')*S'$ 
\[
\xymatrix{ & \alg(\Sigma)*S   \ar[dd]^\varphi   & \\ \alg(\Lambda_+) \ar[ru]^f \ar[rd]_{f'} & & \alg(\Lambda_-) \ar[lu]_i \ar[ld]^{i'} \\ & \alg(\Sigma')*S' } 
\]
such that 
\begin{itemize}
\item $\varphi\circ i=i'$; and
\item $\varphi\circ f \simeq f'$ (DGA homotopy).
\end{itemize}
\end{theorem}
\noindent Theorem \ref{thm:main} is restated in several  parts (Theorems \ref{prop:SigmaDGA}, \ref{DGAmap}, \ref{inva} and Corollary \ref{cor:DGAinva}) in Section \ref{sec:functorI} which are then proven in Section \ref{sec:proof}.  
We remark that the notion of an immersed DGA map is similar to the bordered DGAs introduced by Sivek in \cite{Sivek}.  Moreover, Sivek's work can be viewed as studying the case of $1$-dimensional conical Legendrian cobordisms between $0$-dimensional Legendrian submanifolds.  

 The construction from Theorem \ref{thm:main} may be summarized as follows.  The  
 DGA $\alg(\Sigma)$ is 
generated by Reeb chords of $\Sigma$ and Reeb chords of $\Lambda_-$,
with the map $i: \alg(\Lambda_-) \hookrightarrow \alg(\Sigma)$ induced by the inclusion of generators.
 For the differential $\dd_{\Sigma}$ and the map $f$, we consider a ``Morse cobordism'' $\wt{\Sigma}$, as in \cite{EHK}, obtained from replacing the conical ends of $\Sigma$ with Morse ends of a standard form.  
The Reeb chords of $\wt{\Sigma}$ are the union of Reeb chords of $\Sigma$, $\Lambda_+$ and $\Lambda_-$. 
 Both the differential $\dd_{\Sigma}$ and the map $f$ are 
  defined
  by counting rigid gradient flow trees (GFTs) of $\wt{\Sigma}$.
 To prove that  $f$ is a DGA map and to establish the invariance statement, we 
consider a further Legendrian surface $\ol{\Sigma}$ and show that its DGA is related to the map $f: \alg(\Lambda_+) \rightarrow \alg(\Sigma)$ via a mapping cylinder construction in the setting of DGAs.  (See Section \ref{sec:mapc} for a development of the relevant algebra.)   
Using the algebraic results of Section \ref{sec:mapc}, we see that $\partial_{\ol{\Sigma}}^2=0$  implies that $f$ is a DGA map, and that the   Legendrian isotopy invariance of $\mathcal{A}(\ol\Sigma)$ can be used to establish that the immersed DGA map $\alg(\Lambda_+) \stackrel{f}{\rightarrow} \alg(\Sigma) \stackrel{i}{\hookleftarrow} \alg(\Lambda_-)$, considered up to immersed homotopy, is an invariant of $\Sigma$.

The 
induced immersed DGA map from Theorem \ref{thm:main}  for the concatenation of two cobordisms is immersed homotopic to the composition (suitably defined, see Definition \ref{def:Comp}) of the immersed DGA maps induced by each cobordism.  This is established in Theorem \ref{thm:con}.
To formulate the construction of Theorem \ref{thm:main} as a functor, we 
consider immersed versions of the Legendrian and DGA categories, $\leg_{im}$ and $\dga_{im}$.
The objects in $\leg_{im}$ are 1-dimensional Legendrian submanifolds $\Lambda$ in $J^1M$ (equipped 
with choices of regular metrics and $\Z$-valued Maslov potentials).
A morphism between two objects is a conical Legendrian cobordism, $\Sigma$, (with Maslov potential) considered up to conical Legendrian isotopy. On the other side, the  DGA category $\dga_{im}$ has  $\Z$-graded, finitely generated, triangular DGAs over $\Z/2$ as objects and morphisms are 
immersed DGA maps up to immersed homotopy.  
Thus we have a map $F: \leg_{im}\to \dga_{im}$ that sends a Legendrian to its DGA 
 and a conical Legendrian cobordism to the induced immersed DGA map.
 This construction extends the functor from \cite{EHK} in the following sense.

\begin{theorem} 
The map $F: \leg_{im}\to \dga_{im}$ is a functor.  Moreover, 
 when a good Lagrangian cobordism $L$ is embedded, the DGA $\alg(\Sigma)$ is $\alg(\Lambda_-)$ and the $f$ map is DGA homotopic to the induced  DGA map  introduced in \cite{EHK}. 
\end{theorem}
\noindent The functor $F$ (in the more general $\Z/n$-graded case for $n\geq 0$) is established in Section \ref{sec:functor}, and Section \ref{sec:embed} addresses the case of embedded cobordisms and makes the connection with the constructions from \cite{EHK}.

In extending functoriality to the setting of good Lagrangian cobordisms, we obtain an interesting corollary about the induced maps $f$ in the embedded case.  

\begin{corollary} \label{cor:L}
If $L$ and $L'$ are two embedded exact Lagrangian cobordisms from $\Lambda_-$ to $\Lambda_+$ whose Legendrian lifts are Legendrian isotopic, then the induced DGA maps $f$ and $f'$ are DGA homotopic.
\end{corollary}

According to \cite{EHK}, if two exact Lagrangian cobordisms are exact Lagrangian isotopic, i.e. isotopic through exact Lagrangians,  then the induced DGA maps are DGA homotopic. 
On the face of it, the condition in Corollary \ref{cor:L} appears to be much weaker since a Legendrian isotopy between the Legendrian lifts of $L$ and $L'$ may pass through Legendrians whose Lagrangian projections are non-embedded.

\begin{Q}\label{Q:strong}
Is there an example of two embedded exact Lagrangian cobordisms which are not exact Lagrangian isotopic but have Legendrian lifts that are Legendrian isotopic?
\end{Q}

In practice, it is often the induced DGA maps $f$ that are used to differentiate cobordisms up to exact Lagrangian isotopy; see \cite{EHK, Pan2}.  However, Corollary \ref{cor:L} shows that obtaining a positive answer to Question \ref{Q:strong} would require a stronger invariant.

\subsection{Immersed cobordisms and augmentations}  
An {\bf augmentation}, $\e$, of $\alg(\Lambda)$ to $\Z/2$ is a DGA map  $\e: (\alg(\Lambda), \dd)\to (\Z/2, 0)$, where the latter chain complex has $\Z/2$ in degree $0$ and is $0$ in the other degrees.  For embedded cobordisms there is a well known connection with augmentations of the Legendrian contact DGA:  by functoriality, an exact Lagrangian filling, i.e. a cobordism with $\Lambda_-= \emptyset$, induces an augmentation $\e=f:\alg(\Lambda_+) \rightarrow \alg(\Lambda_-) = \Z/2$.
However, not all augmentations come from embedded exact Lagrangian fillings since, eg.,  
the Seidel isomorphism \cite{EkhSFT, DR} gives a very restrictive condition on the linearized contact homology, which is an invariant of augmentations of Legendrian submanifolds.  
One motivation for studying immersed Lagrangians is to be able to extend this construction 
of augmentations via fillings to a wider class of augmentations.   In order to use a good (immersed) Lagrangian filling $L$ of $\Lambda$ to induce an augmentation, an extra ingredient is required: a choice of augmentation for the conical Legendrian lift of $L$.  
According to the functoriality, a good Lagrangian filling $L$ of $\Lambda$ with conical Legendrian lift $\Sigma$ induces a DGA map  $f:\alg(\Lambda) \to \alg(\Sigma)$.
\begin{theorem}[See Proposition \ref{prop:aug}]\label{thm:induceaug}
Let $L$ be a good Lagrangian filling of $\Lambda$, and suppose $\alpha:\mathcal{A}(\Sigma) \rightarrow \Z/2$ is an augmentation of the conical Legendrian lift of $L$.  
Then, the pair $(\Sigma, \alpha)$ induces an augmentation $\e_{(\Sigma,\alpha)}$ of $\Lambda$ through 
$$\e_{(\Sigma, \alpha)}:= \alpha\circ f.$$
Moreover, the set 
\[
I_\Sigma = \{[ \e_{(\Sigma, \alpha)}]\,|\, \alpha \mbox{ is an augmentation of $\Sigma$}\}
\]
 of DGA homotopy classes of augmentations of $\Lambda$ induced by $\Sigma$ is invariant under good Lagrangian isotopy of $L$.   
\end{theorem}
\noindent 
In Section \ref{sec:aug}, for a family of Legendrian twist knots we give examples of augmentations that can be induced in this manner by an immersed filling with $1$-double point, but cannot be induced by any oriented embedded filling.

In a follow up paper, \cite{PanRu2}, we undertake a more thorough study of augmentations induced by immersed fillings applying tools from \cite{Henry, RuSu1, RuSu3}  involving the cellular DGA and the correspondence between Morse complex families and augmentations.  In particular, we are able to obtain a flexibility result showing that {\it any} (graded) augmentation to $\Z/2$ can be induced by a good Lagrangian filling.
In a complementary direction, Pezzimenti and Traynor have recently used generating family methods to establish rigidity results for immersed Lagrangian fillings such as bounds relating the genus and number of double points of a filling with the ranks of generating family homology groups.

\subsection{Comparison with other approaches}
  
 In the embedded case, \cite{EHK} uses symplectic field theory style moduli spaces of holomorphic disks in defining the induced maps $f$ associated with a Lagrangian cobordism, i.e. they consider holomorphic disks in $\mathit{Symp}(J^1M)$ with boundary punctures that are asymptotic to Reeb chords at $\pm \infty$.  In sections 4 and 5 of \cite{EHK}, it is then shown that GFTs can be used as a computational tool for evaluating these induced maps.  In contrast, we use counts of rigid GFTs in our  definition.
 The reason that we do not use holomorphic disks directly is that the moduli spaces relevant to the setting of immersed cobordisms have a mix of aspects of the moduli spaces used in the original construction of Legendrian contact homology \cite{EESR2n+1}  and those from relative symplectic field theory \cite{Abbas, BEHWZ}.  That is, an SFT-style definition of the induced map $f$ for a good immersed cobordism, $L$, would involve counts of disks in $\mathit{Symp}(J^1M)$ with some punctures at $\pm\infty$ and some at double points of $L$, and such hybrid disks are not typically considered explicitly in the existing literature on analysis of holomorphic curves.    
In Section \ref{sec:SFTpers}, we sketch an alternate definition of induced maps in the immersed case that is natural from the SFT point of view, and observe using results from \cite{EHK} that under a strong assumption on the almost complex structure the two definitions agree.

With the count of rigid GFTs used as the definition, the construction of the functor $F$ in Sections 2-7, while relying heavily on \cite{Ekh} and \cite{EES,EESR2n+1}, is independent from \cite{E2, EHK} and may have some novelty even in the embedded case. 
In general, our approach employs more (DG)-algebraic results and avoids some of the analytic methods used in \cite{E2, EHK}.  We highlight here a few points in which the two approaches differ:
\begin{itemize}
\item The abstract perturbations from \cite{EK} that are used to establish DGA homotopy invariance of induced maps in \cite{EHK} are replaced with algebraic results about mapping cylinder DGAs.  (In particular, see Propositions \ref{prop:unique} and \ref{prop:DGAhmtp}.)
\item In establishing that concatenation of cobordisms corresponds to composition of (immersed) DGA maps, we do not require stretching and gluing results for holomorphic disks in symplectizations.  Instead, an easy evaluation of GFTs for concatenated {\it Morse} cobordisms is combined with Legendrian invariance of the construction and an algebraic result (Proposition \ref{lem:CompositionAlt}) about composition of immersed DGA maps.
\end{itemize}

However, using a GFT based approach does have a few drawbacks.  So far, the class of contact/symplectic manifolds where GFTs methods are available is restricted mainly to $1$-jet spaces/cotangent bundles.  In addition, for Legendrians of dimension $\geq 3$, the correspondence between GFTs and holomorphic disks only holds for Legendrians with simple front singularities,  so in higher dimensions (i.e. for cobordisms of dimension $\geq 3$) our construction only applies to a restricted class of cobordisms.  This is the reason that our considerations are restricted to $1$-dimensional Legendrians with $2$-dimensional cobordisms.

\subsection{Organization of paper}
We start by addressing the algebraic side of the project in the Sections 2-3, and then turn to the geometric side in later sections. 
 In Section \ref{sec:mapc},
we introduce a class of DGAs that we call mapping cylinder DGAs and establish results that allow us to translate between properties of DGA maps and properties of their associated mapping cylinders. 
 Section \ref{sec:DGAcat}  contains the construction of the 
category $\mathfrak{DGA}_{im}$ of DGAs with immersed homotopy classes of immersed DGA maps.
In Section \ref{sec:defn}, we shift to geometry by introducing good Lagrangian cobordisms and conical Legendrian cobordisms.
After reviewing necessary background about LCH in Section \ref{sec:LCH}, the definition of the immersed DGA maps associated to conical Legendrian cobordisms appears in Section \ref{sec:functorI} where properties of the construction are discussed with most proofs deferred to the following Section \ref{sec:proof}.   
Section \ref{sec:aug} discusses induced augmentations, and provides a family of examples of augmentations that may be induced by immersed fillings but not by   
any oriented embedded fillings.
Finally in Section \ref{sec:SFTpers}, we translate the story to $\mathit{Symp}(J^1M)$ and show that it matches with the SFT framework.

\subsection{Acknowledgements} We thank Georgios Dimitroglou Rizell, John Etnyre, Emmy Murphy, and Mike Sullivan  for discussions related to the work.
The first author was supported by the NSF grant (DMS-1510305).  The second author received support from the Simons Foundation, grant \#429536.

\section{Mapping cylinders and DGA homotopy}\label{sec:mapc}

In this section we consider a mapping cylinder construction, similar to the one for chain complexes, in the setting of DGAs.  
DGAs of a similar type arise naturally in computations of Legendrian contact homology, eg. in spinning constructions \cite{EK, RuSu1}.  However, we are not sure how extensively this construction is studied in existing literature on DG-algebras.  
Propositions \ref{prop:DGAmap}  and \ref{prop:DGAhmtp} show that properties of maps between DGAs can be translated into properties of the associated mapping cylinder DGAs.    In the later sections, we will apply these algebraic results by constructing a compact Legendrian $\overline{\Sigma}$ whose DGA is related to the induced map $f_\Sigma:\mathcal{A}(\Lambda_+) \rightarrow \mathcal{A}(\Sigma)$ via the mapping cylinder construction.  
A key result is a uniqueness theorem for the differential on a mapping cylinder DGA (Proposition \ref{prop:unique}); in our approach this algebraic result serves as a substitute for the abstract perturbations used in \cite{E2, EHK}.  

\subsection{Conventions and notations}

In this article, {\it we work only with algebras over $\Z/2$}, and {\bf DGAs (differential graded algebras)} are unital, associative 
 $\Z/2$-algebras with differentials  $\partial:\mathcal{A} \rightarrow \mathcal{A}$ that are graded derivations of degree $-1$, i.e. $\partial(x\cdot y) = \partial(x)\cdot y + x \cdot \partial(y)$.  The grading is by $\Z/N$, $\mathcal{A} = \oplus_{i \in \Z/N} \mathcal{A}_i$, $\mathcal{A}_{i} \cdot \mathcal{A}_j \subset \mathcal{A}_{i+j}$, for some fixed value of $N \geq 0$.  In particular, $N =0$ corresponds to the case of $\Z$-graded DGAs.   A {\bf DGA map} $f:(\mathcal{A}_1, \partial_1) \rightarrow (\mathcal{A}_2, \partial_2)$ is a grading preserving unital algebra homomorphism that is a chain map, i.e. $ \partial_2\circ f = f \circ \partial_1$.  Two DGA maps $f, g:(\mathcal{A}_1, \partial_1) \rightarrow (\mathcal{A}_2, \partial_2)$ are {\bf DGA homotopic} if there is a degree $1$ $(f,g)$-derivation $H:\mathcal{A}_1 \rightarrow \mathcal{A}_2$ such that $f-g = \partial_2 \circ H + H \circ \partial_1$ (where an {\bf $(f,g)$-derivation} is a graded linear map that satisfies $H(x\cdot y) = H(x) \cdot g(y) + (-1)^{|x|\cdot|H|} f(x) \cdot H(y)$.)
Note that any algebra map, derivation, or $(f,g)$-derivation is uniquely specified by its values on a generating set for $\mathcal{A}$, and when the generating set is free these values can be specified arbitrarily.  Moreover, as long as $f$ is an algebra map (resp. $f$ and $g$ are DGA maps and $H$ is an $(f,g)$-derivation) it is enough to verify the chain map equation (resp. the homotopy equation) on a generating set for $\mathcal{A}_1$.  

Given a finite set $\{x_1, \ldots, x_n\}$ we will use the notation $\Z/2\langle x_1, \ldots, x_n \rangle$ for the free associative unital $\Z/2$-algebra generated by the $x_i$.  Given a subset $S$ of an algebra $\mathcal{A}$, we use the notation $\mathcal{I}(S) \subset \mathcal{A}$ for the $2$-sided ideal generated by $S$.  If $S = \{y_1, \ldots, y_l\}$, we write $\mathcal{I}(S) = \mathcal{I}(y_1, \ldots, y_l)$.
 We will use the notation $\mathcal{A}*\mathcal{B}$ for the free product (categorically, the co-product) of unital associative algebras.  When $\mathcal{A}$ and $\mathcal{B}$ are free
(this is the only case we will need), $\mathcal{A}*\mathcal{B}$ exists and is the free algebra generated by the union of free generating sets for $\mathcal{A}$ and $\mathcal{B}$.  When $(\mathcal{A}, \partial_\mathcal{A})$ and $(\mathcal{B}, \partial_{\mathcal{B}})$ are DGAs, there is a unique DGA differential on $\mathcal{A}*\mathcal{B}$ extending $\partial_\mathcal{A}$ and $\partial_\mathcal{B}$.  This construction is functorial for DGA maps.

\subsection{Triangular DGAs and stable tame isomorphism}\label{sec:triDGA} 
 In the setting of Legendrian contact homology, the DGAs that are encountered are free and equipped with explicit ordered generating sets (given by Reeb chords ordered by height/action) and with differentials that respect the ordering.  We recall definitions and a useful result in this algebraic setting.  

We say that $(\mathcal{A},\partial)$ is a {\bf based DGA} if the algebra $\mathcal{A}$ is equipped with an ordered free generating set $\mathcal{G}= \{x_1, \ldots, x_m\}$ 
 so that $\mathcal{A} = \Z/2\langle x_1, \ldots, x_m\rangle$ and each $x_i$ is a homogeneous element of $\mathcal{A}$, i.e. $x_i$ belongs to a single graded component of $\mathcal{A}$.  A {\bf triangular DGA} is a based DGA $(\mathcal{A},\partial)$ whose differential satisfies 
\[
\partial x_i \in \Z/2\langle x_1, \ldots, x_{i-1} \rangle, \quad \mbox{for $1 \leq i \leq m$.}
\]

An {\bf elementary automorphism} of a based algebra $\mathcal{A} = \Z/2\langle x_1, \ldots, x_m\rangle$ is an algebra map $\varphi: \mathcal{A} \rightarrow \mathcal{A}$ such that for some $1 \leq i \leq m$, $\varphi(x_j) = x_j$ when $j \neq i$ and $\varphi(x_i) = x_i + w$ where $w \in \Z/2\langle x_1, \ldots, \widehat{x_i}, \ldots, x_m \rangle$.  An isomorphism of based algebras $\psi: \mathcal{A}_1 \rightarrow \mathcal{A}_2$ is called {\bf tame} if it can be written as a composition $\psi = \varphi_n \circ \cdots \circ \varphi_1 \circ \iota$ where $\iota: \mathcal{A}_1 \rightarrow \mathcal{A}_2$ extends a bijection of generating sets and the $\varphi_i$ are elementary automorphisms of $\mathcal{A}_2$.  

Recall that a {\bf stabilization} of a based DGA, $(\mathcal{A}, \partial_\mathcal{A})$ is a DGA of the form $(\mathcal{A} * S, \partial)$ where $S = \Z/2\langle a_1, b_1, \ldots, a_r,b_r \rangle$ is freely generated by generators $a_i, b_i$, $1 \leq i \leq r$ with degrees $|a_i| = |b_i|+1 = k_i$ and the differential on $\mathcal{A}$ satisfies  $\partial a_i = b_i$, $\partial b_i =0$, and $\partial|_{\mathcal{A}} = \partial_\mathcal{A}$.  Stabilizations have  associated inclusion and projection maps
\[
\iota: \mathcal{A} \hookrightarrow \mathcal{A}*S  \quad \mbox{and} \quad \pi: \mathcal{A}*S \rightarrow \mathcal{A}
\]
(where $\pi$ maps all generators of $S$ to $0$).  These maps are homotopy inverse to one another since $\pi \circ \iota = \mathit{id}_{\mathcal{A}}$ and there is a DGA homotopy $\iota \circ \pi \simeq \mathit{id}_{\mathcal{A}*S}$ given by
\[
\iota \circ \pi - \mathit{id}_{\mathcal{A}*S} = \partial K + K \partial
\] 
where $K: \mathcal{A}*S \rightarrow \mathcal{A}*S$ is the $(\iota \circ \pi, \mathit{id}_{\mathcal{A}*S})$ derivation that satisfies $K(b_i) = a_i$, $1 \leq i \leq r$, and vanishes on all other generators.

\begin{definition}  A {\bf stable (tame) isomorphism} of two DGAs $(\mathcal{A}, \partial_\mathcal{A})$ and $(\mathcal{B}, \partial_\mathcal{B})$ is a choice of stabilizations $(\mathcal{A} * S, \partial)$ and $(\mathcal{B} * S', \partial')$ together with a (tame) DGA isomorphism $\varphi: \mathcal{A}*S \rightarrow \mathcal{B}*S'$.
\end{definition}

\begin{remark} \label{rem:stable}
\begin{enumerate}
\item[(i)] Any stable isomorphism, $\varphi: \mathcal{A}*S \rightarrow \mathcal{B}*S'$, has an {\bf associated homotopy equivalence}, $h:\mathcal{A} \rightarrow \mathcal{B}$, defined by $h = \pi' \circ \varphi \circ \iota$ where $\iota: \mathcal{B} \rightarrow \mathcal{B}*S$ and $\pi':\mathcal{B}'*S' \rightarrow \mathcal{B}'$ are the inclusions and projections.  Note that $h$ is a DGA homotopy equivalence (since $\pi'$, $\varphi$, and $\iota$ are.)  

\item[(ii)] Stable (tame) isomorphisms can be composed in the following sense: If $\varphi_1: \mathcal{B}_1*S_1 \rightarrow \mathcal{B}_2*S_2$ and $\varphi_2: \mathcal{B}_2*S_2' \rightarrow \mathcal{B}_3*S_3'$ are stable isomorphisms, then
\[
\varphi: \mathcal{B}_1 * S_1*S_2' \stackrel{ \varphi_1 * \mathit{id}_{S_2'}}{\longrightarrow} \mathcal{B}_2* S_2 * S_2' = \mathcal{B}_2* S_2'*S_2 \stackrel{ \varphi_2 * \mathit{id}_{S_2}}{\longrightarrow} \mathcal{B}_3*S_3'*S_2.
\] 
is a stable (tame) isomorphism from $\mathcal{B}_1$ to $\mathcal{B}_3$.  Moreover, one can check that if $h_1:\mathcal{B}_1 \rightarrow \mathcal{B}_2$ and $h_2:\mathcal{B}_2 \rightarrow \mathcal{B}_3$ are the associated homotopy equivalences for $\varphi_1$ and $\varphi_2$, then $h_2 \circ h_1$ is the associated homotopy equivalence for $\varphi$.
\end{enumerate}
\end{remark}

In the triangular setting, the following proposition is useful for producing stable tame isomorphisms.  
\begin{proposition}  \label{prop:cancel}
Let $(\mathcal{A}, \partial)$ be a triangular DGA with ordered generating set $\{x_1, \ldots, x_m\}$.  Suppose that for some $i$
\[
\partial x_i = x_j +w
\]
where $w \in \Z/2\langle x_1, \ldots, x_{j-1} \rangle$.  Write $I = \mathcal{I}(x_i, \partial x_i)$ and $p:\mathcal{A} \rightarrow \mathcal{A}/I$ for the projection.  
\begin{enumerate}
\item The differential $\partial$ induces a differential $\partial^I:\mathcal{A}/I \rightarrow \mathcal{A}/I$ with $\partial^I\circ p = p \circ \partial$, and the quotient DGA $(\mathcal{A}/I, \partial^I)$ is itself a triangular DGA with respect to the generating set given by the equivalence classes $\{\overline{x}_1, \ldots, \widehat{j}, \ldots, \widehat{i}, \ldots, \overline{x}_m\}$.  

\item Set $S(y,z) = \Z/2\langle y,z \rangle$ to be the DGA with grading $|y| = |x_i|$ and $|z| = |x_j|$, and differential $\partial' y = z$ and $\partial' z = 0$.  Then, there exists a tame isomorphism of the form
\begin{equation} \label{eq:gh1}
\varphi = g * h: (\mathcal{A}/I) * S(y,z) \rightarrow \mathcal{A}. 
\end{equation}
Here, $g:\mathcal{A}/I \rightarrow \mathcal{A}$ and $h:S(y,z) \rightarrow \mathcal{A}$ are DGA maps satisfying
  $h(y) = x_i$, $h(z) = \partial x_i$;  and
\begin{equation} \label{eq:gh2}
g \circ p - \mathit{id}_\mathcal{A} = \partial \circ H + H \circ \partial.
\end{equation}
In the last equation,  $H$ is a $(g \circ p, \mathit{id}_\mathcal{A})$-derivation that satisfies $H(x_j) = x_i$ and $H(x_l) = 0$ when $l \neq j$.

\end{enumerate}   

\end{proposition}

\begin{remark}
Proposition \ref{prop:cancel} is proven in Theorem 2.1 of \cite{RuSu1}, and is implicit in invariance proofs for Legendrian contact homology going back to \cite{Che}.
\end{remark}

We record some additional observations for later use.

\begin{proposition} \label{prop:observe} The isomorphism $\varphi$ from (\ref{eq:gh1}) satisfies:
\begin{enumerate}
\item $\pi \circ \varphi^{-1} = p$ where $p:\mathcal{A} \rightarrow \mathcal{A}/I$ and $\pi:(\mathcal{A}/I) * S(y,z) \rightarrow \mathcal{A}/I$ are the projections.
\item Suppose $\mathcal{B} \subset \mathcal{A}$ is a based sub-DGA whose generating set $\{b_1, \ldots, b_m\} \subset \{x_1, \ldots, x_n\}$ does not contain $x_i$ or $x_j$.
 If we view $\mathcal{B}$ as a sub-algebra of $\mathcal{A}/I$ by identifying generators with their equivalence classes, then $\phi|_{\mathcal{B}} = \mathit{id}_{\mathcal{B}}$.
\end{enumerate}
\end{proposition}

\begin{proof}
Note that (1) is equivalent to $\pi = p\circ \varphi$ which we verify on generators.  First, we have $\pi(y) = \pi(z) =0$, and  $p \circ \varphi(y) = p(x_i)=0$ and $p\circ \varphi(z) = p(\partial x_i) = 0$.  
The other generators of $\mathcal{A}/I *S(y,z)$ are of the form $[x_l] = p(x_l)$ with $l \notin \{i,j\}$, and satisfy
\[
\pi([x_l]) = [x_l]
\]  
and
\[
p\circ \varphi([x_l]) = p\circ g([x_l]) = p \left( g \circ p(x_l) \right) = p\left( x_l + \partial \circ H (x_l) +H \circ \partial(x_l)\right) = p(x_l) = [x_l]
\]
where at the last equality we used that $H(x_l) =0$ and $\mbox{Im}(H) \subset \mathcal{I}(x_i) \subset \mbox{Ker}(p)$.

To check (2), compute
\[
\phi([b_i]) = g \circ p(b_i) = b_i + \partial \circ H(b_i) + H \circ \partial (b_i) = b_i 
\]
where $H \circ \partial(b_i) =0$ since the generator $x_j$ does not appear anywhere in $\partial(b_i)$.

\end{proof}

\subsection{Mapping cylinder DGAs}\label{sec:mapcyl}
  Let $\mathcal{A}$ and $\mathcal{B}$ be based graded algebras with generating sets $\{a_1, \ldots, a_m\}$ and $\{b_1, \ldots, b_n\}$, and consider 
\[
\mathcal{C} = \alg * \widehat{\mathcal{A}}* \mathcal{B}
\]
where the generating set for $\widehat{\mathcal{A}}$, $\{ \widehat{a}_1, \ldots, \widehat{a}_m \}$, is in bijection with the generating set of $\mathcal{A}$ but with grading shifted up by $1$, $|\widehat{a}_i| = |a_i|+1$.  We call $\mathcal{C}$ the {\bf mapping cylinder algebra} from $\mathcal{A}$ to $\mathcal{B}$.  

Next, assume that the subalgebras $\mathcal{A}$ and $\mathcal{B}$ are equipped with differentials $\partial_\mathcal{A}$ and $\partial_\mathcal{B}$ making them into DGAs.

\begin{proposition} \label{prop:DGAmap}
Suppose $\partial: \mathcal{C} \rightarrow \mathcal{C}$ is a differential such that
\begin{enumerate}
\item $\mathcal{A}$ and $\mathcal{B}$ are sub-DGAs with $\partial|_\mathcal{A} = \partial_\mathcal{A}$ and $\partial|_\mathcal{B} = \partial_\mathcal{B}$, and
\item for $1 \leq i \leq m$, we have
\begin{equation} \label{eq:dbi}
\partial \widehat{a}_i = f(a_i) + a_i + \gamma_i
\end{equation}
where $f(a_i) \in \mathcal{B}$ and $\gamma_i \in \mathcal{I}(\widehat{a}_1, \ldots, \widehat{a}_m)$ (the $2$-sided ideal generated by the $\widehat{a}_i$).
\end{enumerate}
Then, the extension of $f$ as an algebra homomorphism is a DGA map, $f:(\mathcal{A}, \partial_\mathcal{A}) \rightarrow (\mathcal{B}, \partial_\mathcal{B})$.
\end{proposition}

\begin{proof}
Define $\pi_\mathcal{B}: \mathcal{C} \rightarrow \mathcal{B}$ to be the unique algebra homomorphism satisfying
\[
\pi_\mathcal{B}\big|_{\mathcal{B}} = \mathit{id}_\mathcal{B}; \quad \pi_\mathcal{B}\big|_\mathcal{A} = f; \quad \mbox{and} \quad \pi_\mathcal{B}(\widehat{a}_i) = 0,  \,\, 1 \leq i \leq m.
\]
Post composing (\ref{eq:dbi}) with $\pi_\mathcal{B} \circ \partial$ gives
\begin{align*}
0 =& \pi_{\mathcal{B}} (\partial \circ f(a_i)) + \pi_\mathcal{B}(\partial a_i) + (\pi_\mathcal{B} \circ \partial)(\gamma_i) \\ 
 = & \partial_\mathcal{B} \circ f (a_i) + f \circ \partial_\mathcal{A}(a_i) + (\pi_\mathcal{B} \circ \partial)(\gamma_i),
\end{align*}
so the result follows provided that $\pi_\mathcal{B}\circ \partial$ vanishes on $\mathcal{I}( \widehat{a}_1, \ldots, \widehat{a}_m )$.  To verify this, since $\pi_\mathcal{B}\circ \partial: \mathcal{C} \rightarrow \mathcal{B}$ is a $(\pi_\mathcal{B}, \pi_\mathcal{B})$-derivation with $\pi_\mathcal{B}(\widehat{a}_i) =0$, it suffices to compute
\[
\pi_\mathcal{B}\circ \partial(\widehat{a}_i) = \pi_\mathcal{B}(f(a_i) + a_i + \gamma_i) = f(a_i) + f(a_i) + 0 = 0.
\]
\end{proof}

When $\partial:\mathcal{C} \rightarrow \mathcal{C}$ satisfies the conditions of Proposition \ref{prop:DGAmap}, we call $(\mathcal{C},\partial)$ a {\bf mapping cylinder DGA} for the DGA map $f:(\mathcal{A}, \partial_\mathcal{A}) \rightarrow (\mathcal{B},\partial_\mathcal{B})$.  Note that the differential of a mapping cylinder DGA is {\it not} uniquely determined by $f$, but see Proposition \ref{prop:unique} below for a strong uniqueness up to isomorphism result in the triangular case.  

\subsection{The standard mapping cylinder differential}
The next proposition can be viewed as a converse to the previous one, as it constructs a mapping cylinder DGA for any DGA map $f:(\mathcal{A}, \partial_\mathcal{A}) \rightarrow (\mathcal{B}, \partial_\mathcal{B})$.

\begin{proposition} \label{prop:standard}  Suppose that 
$f:(\mathcal{A}, \partial_\mathcal{A}) \rightarrow (\mathcal{B}, \partial_\mathcal{B})$ is a DGA map between based DGAs, and let $\mathcal{C}= \mathcal{A}*\widehat{\mathcal{A}}*\mathcal{B}$ be the mapping cylinder algebra.

Write $i: \mathcal{A} \rightarrow \mathcal{C}$ for the inclusion, and consider the $(f,i)$-derivation $\Gamma:\mathcal{A} \rightarrow \mathcal{C}$ defined on generators by $\Gamma(a_i) = \widehat{a}_i$.  Define $\partial: \mathcal{C} \rightarrow \mathcal{C}$ to be the unique derivation on $\mathcal{C}$ satisfying
\begin{align}
&\partial\big|_\mathcal{A} = \partial_\mathcal{A}, \quad \partial\big|_\mathcal{B} = \partial_\mathcal{B},  \label{eq:dAdC} \\
&\partial(\widehat{a}_i) = f(a_i) +a_i + \Gamma \circ \partial_\mathcal{A}(a_i).  \label{eq:bGamma}
\end{align}
Then, $(\mathcal{C}, \partial)$ is a DGA.
\end{proposition}
\begin{proof}
The derivation, $\partial$, has degree $-1$ when applied to generators, so we just need to verify that $\partial^2=0$.  First, note that the identity
\begin{equation} \label{eq:fiidentity}
f+i = \partial \circ \Gamma + \Gamma \circ \partial_\mathcal{A}
\end{equation}
holds on all of $\mathcal{A}$.  [Assuming that (\ref{eq:fiidentity}) holds when applied to $x$ and $y$, 
a straight forward computation that uses 
the multiplicative properties of $f, i, \Gamma$, and $\partial$ shows that (\ref{eq:fiidentity}) is also valid when applied to $xy$.  Moreover, the identity (\ref{eq:bGamma}) shows that (\ref{eq:fiidentity}) holds when applied to the generating set $a_i$.]

Now, (\ref{eq:dAdC}) shows that $\partial^2=0$ holds on generators of the form $a_i$ or $b_i$.  To verify that $\partial^2(\widehat{a}_i) =0$, compute using (\ref{eq:fiidentity}) that
\begin{align*}
\partial^2(\widehat{a}_i) =& \partial (f(a_i)) + \partial(a_i) + \partial\left(\Gamma \circ \partial_\mathcal{A}(a_i)\right) \\
 =& (\partial_\mathcal{B} \circ f)(a_i) + \partial_\mathcal{A}(a_i) + (\partial \circ \Gamma) \circ \partial_\mathcal{A}(a_i) \\
 =& (f \circ \partial_\mathcal{A})(a_i) + \partial_\mathcal{A}(a_i) + \left[ f \circ \partial_\mathcal{A}(a_i) + \partial_\mathcal{A}(a_i) + \Gamma \circ \partial_\mathcal{A}^2(a_i) \right] = 0.
\end{align*} 
\end{proof}

We refer to the DGA constructed in Proposition \ref{prop:standard} as the {\bf standard mapping cylinder DGA} associated to $f$.

\subsection{A uniqueness theorem}

Suppose now that $\mathcal{C}= \mathcal{A} *\widehat{\mathcal{A}}*\mathcal{B}$ is a mapping cylinder algebra, and assume the generating set for $\mathcal{C}$ is ordered as
	\begin{equation}  \label{eq:orderC}
	\{b_1, \ldots, b_n, a_1, \ldots, a_m, \widehat{a}_1, \ldots, \widehat{a}_m \}.
	\end{equation}

\begin{proposition}  \label{prop:unique}
Let $\partial_k: \mathcal{C} \rightarrow \mathcal{C}$, $k=1,2$ be DGA differentials such that:
\begin{enumerate}
\item  Both $(\mathcal{C},\partial_1)$ and $(\mathcal{C},\partial_2)$ are triangular DGAs with respect to the ordering of generators from (\ref{eq:orderC}). 
\item  The restrictions to $\mathcal{A}$ and $\mathcal{B}$ agree, 
\[
\partial_1\big|_{\mathcal{A}} = \partial_2\big|_{\mathcal{A}}  \quad \mbox{and}  \quad \partial_2\big|_{\mathcal{B}} = \partial_2 \big|_{\mathcal{B}}.
\]

\item For $k=1,2$, 
\[
\partial_k(\widehat{a}_i) = f(a_i) + a_i + \gamma^k_i, \quad 1 \leq i \leq m,
\]
 where $f(a_i) \in \mathcal{B}$ is independent of $k$ and $\gamma^k_i \in \mathcal{I}(\widehat{a}_1, \ldots, \widehat{a}_m)$.
\end{enumerate} 
Then, there exists a tame DGA isomorphism $\varphi: (\mathcal{C}, \partial_1) \rightarrow (\mathcal{C}, \partial_2)$ such that $\varphi\big|_{\mathcal{A}} = \mathit{id}_\mathcal{A}$ and $\varphi\big|_{\mathcal{B}} = \mathit{id}_\mathcal{B}$.
\end{proposition}  

\begin{remark}  Note that in the statement of Proposition \ref{prop:unique}, $\partial_1$ and $\partial_2$ do not necessarily need to be mapping cylinder differentials since the condition (2) only implies that $\partial_k(\mathcal{A}) \subset \mathcal{A}*\mathcal{B}$ so that $\partial_k\big|_\mathcal{A}$ may not define a differential on $\mathcal{A}$.  Although our application of the proposition will be to triangular mapping cylinder DGAs, we need to allow the weaker condition (2) to accommodate the inductive proof.  
\end{remark}

\begin{proof}
We use induction on $m$ (the number of $\widehat{a}_i$ generators).  When $m=0$ or $m=1$ we have $\partial_1 = \partial_2$.

For the inductive step, we use that since $\gamma^k_1 = 0$ (by the triangularity condition, no $\widehat{a}_i$ can appear in $\partial_k \widehat{a}_1$) we have
\[
\partial_1 \widehat{a}_1 = \partial_2 \widehat{a}_1 = a_1 + f(a_1).
\]
Write $I = \mathcal{I}(\widehat{a}_1, \partial_1 \widehat{a}_1) = \mathcal{I}(\widehat{a}_1, \partial_2 \widehat{a}_1)$.   Proposition \ref{prop:cancel} states that with respect to the  ordered generating set given by the equivalence classes of generators of $\mathcal{C}$ other than $a_1$ and $\widehat{a}_1$,  $\mathcal{C}/I$ is itself a triangular DGA with differential $\partial_k^I$ satisfying $p \circ \partial_k = \partial_k^I \circ p$ where $p: \mathcal{C} \rightarrow \mathcal{C}/I$ is the projection.  Moreover, the proposition provides DGA maps 
\[
g_k:(\mathcal{C}/I, \partial^I_k)  \rightarrow (\mathcal{C}, \partial_k) \quad \mbox{and} \quad h: ( S(y,z), \partial') \rightarrow (\mathcal{C}, \partial_k)
\]
such that
\begin{itemize}
\item $g_k*h$ is a tame DGA isomorphism;  
\item $h(y) = \widehat{a}_1$ and $h(z) = \partial_k \widehat{a}_1 = f(a_1) + a_1$; and
\item there is a $(g_k \circ p, \mathit{id}_{\mathcal{C}})$-derivation $H_k: \mathcal{C} \rightarrow \mathcal{C}$ with $H_k(a_1) = \widehat{a}_1$ and $H_k(x) =0$ when $x$ is any of the other generators, $b_i, \widehat{a}_i$, or $a_l$ with $l \neq 1$, such that
\begin{equation} \label{eq:gkHomotopy}
g_k \circ p - \mathit{id}_{\mathcal{C}} = \partial_k \circ H_k + H_k \circ \partial_k.
\end{equation} 
\end{itemize}

We will want to apply the inductive hypothesis to the $(\mathcal{C}/I, \partial_k^I)$.  To match the setting from the statement of this proposition, let $\mathcal{C}/I =   \mathcal{A}^I* \widehat{\mathcal{A}}^I* \mathcal{B}^I$ where the factors are the sub-algebras generated by the equivalence classes $[a_i]$ with $i>1$, $[\widehat{a}_i]$ with $i>1$, and all $[b_i]$  respectively.  We now verify the requirements (2) and (3) from the statement of this proposition; (1) has already been observed above.  For (2), let $x= b_i$ or $x=a_j$ with $j \neq 1$, so that $p(x)$ is an arbitrary generator for $\mathcal{B}^I$ or $\mathcal{A}^I$.  
 We have
\[
\partial^I_1 ( p (x))  = p \circ \partial_1(x) = p \circ \partial_2(x) = \partial^I_2( p(x)), 
\]
so $\partial^I_1\big|_{\mathcal{A}^I} = \partial^I_2\big|_{\mathcal{A}^I}$ and $\partial^I_1\big|_{\mathcal{B}^I}=\partial^I_2\big|_{\mathcal{B}^I}$.  
To verify (3),  note that $p(\gamma_i^k) \in I([\widehat{a}_2], \ldots, [\widehat{a}_m])$ holds as $p(\widehat{a}_1)=0$ implies that $p(I(\widehat{a}_1, \ldots, \widehat{a}_m)) \subset I([\widehat{a}_2], \ldots, [\widehat{a}_m])$.

With (1)-(3) verified for the $(\mathcal{C}/I, \partial^I_k)$, the induction produces a tame DGA isomorphism
\[
\varphi_0: (\mathcal{C}/I, \partial^I_1) \rightarrow (\mathcal{C}/I, \partial^I_2), \quad  \mbox{with} \quad \varphi_0\big|_{\mathcal{A}^I} = \mathit{id}_{\mathcal{A}^I}, \, \varphi_0\big|_{\mathcal{B}^I} = \mathit{id}_{\mathcal{B}^I}.  
\]
We then define $\varphi: (\mathcal{C}, \partial_1) \rightarrow (\mathcal{C}, \partial_2)$ as the composition of tame DGA isomorphisms
\[
(\mathcal{C}, \partial_1) \stackrel{(g_1*h)^{-1}}{\longrightarrow} (\mathcal{C}/I, \partial^I_1) *( S(y,z), \partial') \stackrel{\varphi_0*\mathit{id}}{\longrightarrow} (\mathcal{C}/I, \partial^I_2) *( S(y,z), \partial') \stackrel{g_2*h}{\longrightarrow}  (\mathcal{C}, \partial_2).
\]
To see that $\varphi|_\mathcal{A} = \mathit{id}_\mathcal{A}$ and $\varphi|_\mathcal{B} = \mathit{id}_\mathcal{B}$, we establish the following.

\begin{lemma} \label{lem:g1h} We have 
\[
(g_1*h)\big|_{\mathcal{A}^I*\mathcal{B}^I* S(y,z)} = (g_2*h)\big|_{\mathcal{A}^I*\mathcal{B}^I* S(y,z)}, \quad \mbox{and} \quad \mathcal{A}*\mathcal{B} \subset (g_k*h)\left( \mathcal{A}^I*\mathcal{B}^I*S(y,z)\right).
\]
\end{lemma}

Indeed, given the Lemma, for any $x \in \mathcal{A}*\mathcal{B}$, there exists some $w \in \mathcal{A}^I*\mathcal{B}^I*S(y,z)$ with $g_1*h(w) = g_2*h(w) = x$, and since $\varphi_0$ restricts to the identity on $\mathcal{A}^I*\mathcal{B}^I$ we have
\[
\varphi(x) = (g_2*h) \circ(\phi_0*\mathit{id})\circ(g_1*h)^{-1}(x) = (g_2*h) \circ(\phi_0*\mathit{id})(w) = (g_2*h)(w) = x.
\]
\end{proof}

\begin{proof}[Proof of Lemma \ref{lem:g1h}]

For $l\geq 1$, let $\mathcal{A}^I_l = \Z/2\langle [a_2], \ldots, [a_l] \rangle$ and $\mathcal{A}_l = \Z/2 \langle a_1, \ldots, a_l \rangle$.  Recall the homotopy operators, $H_1$ and $H_2$, from equation (\ref{eq:gkHomotopy}).  We prove the lemma by establishing the following.

\medskip

\noindent{\bf Inductive Statement:}  For $l \geq 1$,  
\begin{itemize}
\item[(i)] $g_1\big|_{\mathcal{A}^I_l * \mathcal{B}^I} = g_2\big|_{\mathcal{A}^I_l * \mathcal{B}^I}$,
\item[(ii)] $H_1\big|_{\mathcal{A}_l*\Z/2\langle \widehat{a}_1 \rangle * \mathcal{B}} = H_2\big|_{\mathcal{A}_l*\Z/2\langle \widehat{a}_1 \rangle * \mathcal{B}}$, and
\item[(iii)]   $(g_k*h)\left( \mathcal{A}^I_l*\mathcal{B}^I*S(y,z)\right) = \mathcal{A}_l*\Z/2\langle \widehat{a}_1 \rangle *\mathcal{B}$.
\end{itemize}

\medskip

\noindent {\bf Base Case:} $l=1$.  For $k=1,2$, $H_k$ is a $(g_k \circ p, \mathit{id}_\mathcal{C})$-derivation that vanishes on the generators of $\mathcal{B}$. This implies that $H_k$ vanishes on all of $\mathcal{B}$.  Therefore, since $\partial_k(b_i) \in \mathcal{B}$ (by triangularity), using (\ref{eq:gkHomotopy}) we have   
\[
g_k([b_i]) = b_i + (\partial_k \circ H_k)(b_i) + (H_k \circ \partial_k)(b_i) = b_i.
\]
This allows us to record that 
\begin{equation}  \label{eq:Cbase}
g_1|_{\mathcal{B}^I} = g_2|_{\mathcal{B}^I} \quad \mbox{and} \quad  \mathcal{B} = g_k\left( \mathcal{B}^I\right), 
\end{equation}
which establishes (i).  Since $p(\mathcal{A}_1 * \Z/2\langle \widehat{a}_1 \rangle *\mathcal{B}) \subset \mathcal{B}^I$, we get that
\[
(g_1 \circ p)\big|_{\mathcal{A}_1 * \Z/2\langle \widehat{a}_1 \rangle*\mathcal{B}} = (g_2 \circ p)\big|_{\mathcal{A}_1 * \Z/2\langle \widehat{a}_1 \rangle*\mathcal{B}} 
\]
and it follows that the $(g_k \circ p, \mathit{id}_{\mathcal{C}})$-derivations, $k=1,2$, $H_1$ and $H_2$ agree when restricted to $\mathcal{A}_1 * \Z/2\langle \widehat{a}_1 \rangle*\mathcal{B}$.  This is requirement (ii).  For (iii),  since $h(y) = \widehat{a}_1$ and $h(z) = a_1 + f(a_1)$ with $f(a_1) \in \mathcal{B}$, we get
\begin{equation} \label{eq:gkhbase}
(g_k*h)\left( \mathcal{B}^I*S(y,z)\right) \subset \mathcal{A}_1*\Z/2\langle \widehat{a}_1 \rangle *\mathcal{B}.
\end{equation}
Moreover, from (\ref{eq:Cbase}) there is some $w \in \mathcal{B}^I$ with $g_k(w) = f(a_1)$, and we get that both $a_1 = (g_k *h)(w+z)$ and $\widehat{a}_1$ belong to $(g_k*h)\left( \mathcal{B}^I*S(y,z)\right)$ so that in (\ref{eq:gkhbase}) the reverse inclusion also holds. 

\medskip

\noindent {\bf Inductive Step:}  Assume that $l \geq 2$ and (i)-(iii) hold for smaller $l$.   Using (\ref{eq:gkHomotopy}), we compute
\begin{align}
g_1([a_l]) = a_l + (\partial_1 \circ H_1)(a_l) + (H_1 \circ \partial_1)(a_l)  & = a_l + H_1(\partial_1(a_l)) \label{eq:alH1} \\
 &= a_l + H_2(\partial_2(a_l)) = g_2([a_l]). \notag
\end{align}
[At the 3rd equality we used that $\partial_1|_\mathcal{A} = \partial_2|_\mathcal{A}$ and that by triangularity $\partial_k(a_l) \in \mathcal{A}_{l-1}*\mathcal{B}$ so that (ii) from the inductive hypothesis can be applied.]  This establishes (i).  Since $H_1$ and $H_2$ are $(g_k \circ p, \mathit{id}_{\mathcal{C}})$-derivations that agree on the generating set, (i) implies (ii).  To establish (iii), note that since the $H_1(\partial_1(a_l))$ term in (\ref{eq:alH1}) belongs to $\mathcal{A}_{l-1}*\Z/2\langle \widehat{a}_1 \rangle *\mathcal{B}$, we get that $g_k([a_l]) \in \mathcal{A}_{l}*\Z/2\langle \widehat{a}_1 \rangle *\mathcal{B}$ so that the $\subset$ inclusion holds in (iii).  Moreover, by the inductive hypothesis (iii), we get $H_1(\partial_1(a_l)) \in (g_1 * h)\left( \mathcal{A}^I_{l-1}*\mathcal{B}^I*S(y,z)\right)$ so that $a_l = g_k([a_l]) + H_1(\partial_1(a_l)) \in (g_1 * h)\left( \mathcal{A}^I_{l}*\mathcal{B}^I*S(y,z)\right)$ and the reverse inclusion $\supset$ holds in (iii) also.
\end{proof}

The following corollary 
allows us to work with standard mapping cylinder DGAs when convenient.

\begin{corollary}  \label{cor:standard}
Let $\mathcal{C} = \mathcal{A} * \widehat{\mathcal{A}} * \mathcal{B}$, and suppose that $(\mathcal{C},\partial)$ is a mapping cylinder DGA for  $f:(\mathcal{A}, \partial_\mathcal{A}) \rightarrow (\mathcal{B},\partial_\mathcal{B})$.

If  $(\mathcal{C}, \partial)$ is triangular with respect to the ordering of generators from (\ref{eq:orderC}),
then there is a DGA isomorphism $\varphi:(\mathcal{C}, \partial) \rightarrow (\mathcal{C}, \partial_\Gamma)$ with $\varphi\big|_\mathcal{A} = \mathit{id}_\mathcal{A}$ and $\varphi\big|_\mathcal{B} = \mathit{id}_\mathcal{B}$ where $(\mathcal{C}, \partial_\Gamma)$ is the standard mapping cylinder DGA for $f$ (from Proposition \ref{prop:standard}).
\end{corollary}
\begin{proof}
Just note that $\partial_\Gamma$ is also triangular so that Proposition \ref{prop:unique} applies.
\end{proof}

\subsection{DGA homotopies from isomorphisms of mapping cylinders}

The following method can be used to produce DGA homotopies. 
In the statement of the proposition, we work with a pair of mapping cylinder algebras
\[
\mathcal{C}_1 = \mathcal{A}*\widehat{\mathcal{A}}*(\mathcal{B}_1) \quad \mbox{and} \quad \mathcal{C}_2 = \mathcal{A}* \widehat{\mathcal{A}}*(\mathcal{B}_2)
\]
associated to a pair of maps with the same domain, $\mathcal{A}$, but different codomains, $\mathcal{B}_1$ and $\mathcal{B}_2$.  We notate generating sets as $\{a_i \,|\, 1 \leq i \leq m\}$, $\{\widehat{a}_i \,|\, 1 \leq i \leq m\}$, $\{b^1_i \,|\, 1 \leq i \leq n_1\}$, and $\{b^2_i \,|\, 1 \leq i \leq n_2\}$.  For the triangularity hypothesis, we order the generating sets of $\mathcal{C}_1$ (resp. $\mathcal{C}_2$) so that the $b^1_i$ (resp. the $b^2_i$) are followed by the $a_i$ and then the $\widehat{a}_i$.  

\begin{proposition}\label{prop:DGAhmtp}
Suppose that $(\mathcal{C}_1, \partial_1)$ and $(\mathcal{C}_2, \partial_2)$ are triangular mapping cylinder DGAs for DGA maps $f_k:(\mathcal{A}, \partial_\mathcal{A}) \rightarrow (\mathcal{B}_k, \partial_{\mathcal{B}_k})$, $k=1,2$.
Suppose further that $\phi:(\mathcal{C}_1, \partial_1) \rightarrow (\mathcal{C}_2, \partial_2)$ is a DGA map such that
\begin{enumerate}
\item $\phi|_\mathcal{A} = \mathit{id}_\mathcal{A}$, and
\item $\phi(\mathcal{B}_1) \subset \mathcal{B}_2$.
\end{enumerate}
Then, $\phi \circ f_1$ and $f_2$ are DGA homotopic (as DGA maps $(\mathcal{A}, \partial_\mathcal{A}) \rightarrow (\mathcal{B}_2, \partial_{\mathcal{B}_2})$).  

\end{proposition}

\begin{proof}
Note that by replacing $\phi$ with $\phi' =  \varphi_2\circ \phi \circ \varphi_1^{-1}$ where $\varphi_1$ and $\varphi_2$ are isomorphisms from Corollary \ref{cor:standard}, we can assume that both $\partial_1$ and $\partial_2$ are the standard form differentials,
\[
\partial_k(\widehat{a}_i) = f_k(a_i) + a_i + \Gamma_k \circ \partial_\mathcal{A}(a_i)
\]
where $\Gamma_k:\mathcal{A} \rightarrow \mathcal{C}_k$ is the $(f_k,i)$-derivation ($i$ is the inclusion) with $\Gamma_k(a_i) = \widehat{a}_i$.  [Notice that since the $\varphi_k$ restrict to the identity on $\mathcal{A}$ and $\mathcal{B}_k$, $\phi'$ will still satisfy (1) and (2).]

Define an algebra homomorphism $\pi_{\mathcal{B}_2}: \mathcal{C}_2 \rightarrow \mathcal{B}_2$ to satisfy
\[
\pi_{\mathcal{B}_2}\big|_{\mathcal{A}} = f_2;  \quad \quad \pi_{\mathcal{B}_2}(\widehat{a}_i) = 0;  \quad \quad \pi_{\mathcal{B}_2}\big|_{\mathcal{B}_2} = \mathit{id}_{\mathcal{B}_2}.
\]
Notice that $\pi_{\mathcal{B}_2}$ is a DGA map, i.e. 
\[
\pi_{\mathcal{B}_2} \circ \partial_2 = \partial_{\mathcal{B}_2}\circ \pi_{\mathcal{B}_2}.
\]
[This holds on the sub-algebras $\mathcal{A}$ and $\mathcal{B}_2$ because $f_2$ and $\mathit{id}_{\mathcal{B}_2}$ are chain maps.  On a generator of the form $\widehat{a}_i$, we compute
\begin{align*}
\pi_{\mathcal{B}_2} \circ \partial_2(\widehat{a}_i) &= \pi_{\mathcal{B}_2}( f_2(a_i) + a_i + \Gamma_2 \circ \partial_\mathcal{A}(a_i)) \\
 &= f_2(a_i) + f_2(a_i) + 0 = 0 = \partial_{\mathcal{B}_2} \circ \pi_{\mathcal{B}_2}(\widehat{a}_i)
\end{align*}
where we used that $\mbox{Im}(\Gamma) \subset \mathcal{I}(\widehat{a}_1, \ldots, \widehat{a}_m)$.]

Now, since $\phi$ is also a DGA map, we have
\begin{equation} \label{eq:rhoDGA}
\pi_{\mathcal{B}_2} \circ \phi \circ \partial_{1} =\partial_{\mathcal{B}_2} \circ \pi_{\mathcal{B}_2}  \circ \phi.
\end{equation}
For any $\widehat{a}_i$, we can compute
\begin{align*}
(\pi_{\mathcal{B}_2} \circ \phi \circ \partial_{1})(\widehat{a}_i) & = (\pi_{\mathcal{B}_2} \circ \phi)( f_1(a_i) + a_i + \Gamma_1 \circ \partial_\mathcal{A}(a_i)) \\
 & = \pi_{\mathcal{B}_2}\left( \phi \circ f_1(a_i) + a_i + \phi \circ \Gamma_1 \circ \partial_\mathcal{A}(a_i)\right) \\
 & = \phi \circ f_1(a_i) + f_2(a_i) + (\pi_{\mathcal{B}_2} \circ \phi \circ \Gamma_1) \circ \partial_\mathcal{A}(a_i). 
\end{align*}
[At the second equality we used property (2) of $\phi$; at the third equality we used property (1) of $\phi$ and the definition of $\pi_{\mathcal{B}_2}$.]
In addition, 
we have
\[
\partial_{\mathcal{B}_2}\circ \pi_{\mathcal{B}_2} \circ \phi(\widehat{a}_i) = \partial_{\mathcal{B}_2} \circ( \pi_{\mathcal{B}_2} \circ \phi \circ \Gamma_1)(\widehat{a}_i).
\]
Combining the previous two calculations with (\ref{eq:rhoDGA}), we have shown that for any generator of $\mathcal{A}$,
\begin{equation} \label{eq:fandK}
\phi \circ f_1 (a_i) + f_{2}(a_i) = \partial_{\mathcal{B}_2} \circ K (a_i) + K \circ \partial_\mathcal{A}(a_i)
\end{equation} 
where $K = \pi_{\mathcal{B}_2} \circ \phi \circ \Gamma_1$.  Since $\Gamma_1$ is a $(f_1 , i)$-derivation (where $i: \mathcal{A} \hookrightarrow \mathcal{C}_1$ is the inclusion) and $\pi_{\mathcal{B}_2} \circ \phi$ is an algebra homomorphism, it follows that $K$ is an $(\pi_{\mathcal{B}_2} \circ \phi \circ f_1, \pi_{\mathcal{B}_2} \circ \phi \circ i)$-derivation.  The properties of $\phi$ and the definition of $\pi_{\mathcal{B}_2}$, give that
\[
(\pi_{\mathcal{B}_2} \circ \phi \circ f_1, \pi_{\mathcal{B}_2} \circ \phi \circ i) = ( \phi \circ f_1, f_{2}).
\]
Therefore, $K$ is actually a $(\phi \circ f_1, f_2)$-derivation, and the equality (\ref{eq:fandK}) shows that $\phi \circ f_1$ and $f_{2}$ are DGA homotopic.

\end{proof}

\section{The immersed DGA category}\label{sec:DGAcat}

  To have an appropriate target category for a Legendrian contact homology functor that allows immersed cobordisms, we consider a generalized class of DGA morphisms that we will call immersed DGA maps.  This notion is essentially equivalent to the bordered DGAs studied by Sivek in \cite{Sivek}.  We then introduce a notion of homotopy for immersed DGA maps and give two equivalent characterizations of the composition of immersed maps up to homotopy.  The section concludes by establishing a category of DGAs with morphisms given by homotopy classes of immersed maps and observing some relations with the standard homotopy category of DGAs.   

\subsection{Immersed DGA maps and immersed homotopy}

Let $(\mathcal{A}_1, \partial_1)$ and $(\mathcal{A}_2, \partial_2)$ be finitely generated, triangular DGAs over $\Z/2$.

\begin{definition}\label{def:immDGAmap}
An {\bf immersed DGA map} $M$ from $(\mathcal{A}_1, \partial_1)$ to $(\mathcal{A}_2, \partial_2)$ is a triple $M=((\mathcal{B}, \partial), f, i)$, consisting of a finitely generated, 
 triangular DGA $(\mathcal{B}, \partial)$ together with a pair of DGA maps $f:(\mathcal{A}_1, \partial_1) \rightarrow (\mathcal{B}, \partial)$ and $i:(\mathcal{A}_2, \partial_2) \hookrightarrow (\mathcal{B}, \partial)$ with the requirement that $i$ includes $\mathcal{A}_2$ into $\mathcal{B}$ as a based sub-DGA.  That is, $i$ maps the generating set of $\mathcal{A}_2$ injectively to a subset of the generating set of $\mathcal{B}$.  
\end{definition}
We often present an immersed DGA map as a diagram
\[
M = \big(\mathcal{A}_1 \stackrel{f}{\rightarrow} \mathcal{B} \stackrel{i}{\hookleftarrow} \mathcal{A}_2 \big).
\]
We may use the map $i$ to view $\mathcal{A}_2$ as a sub-DGA of $\mathcal{B}$.  With this identification, it is always possible to reorder the generating set of $\mathcal{B}$ as $\{a_1, \ldots, a_m, b_1, \ldots, b_n\} \subset \mathcal{B}_2$ where $\{a_1, \ldots, a_m \}$ is the generating set for $\mathcal{A}_2$ and so that $\partial$ is triangular with respect to the ordering.

\begin{definition}  \label{def:homotopy}
Two immersed DGA maps $M = \big(\mathcal{A}_1 \stackrel{f}{\rightarrow} \mathcal{B} \stackrel{i}{\hookleftarrow} \mathcal{A}_2 \big)$ and $M' = \big(\mathcal{A}_1 \stackrel{f'}{\rightarrow} \mathcal{B}' \stackrel{i'}{\hookleftarrow} \mathcal{A}_2 \big)$
 are said to be {\bf immersed homotopic}, notated as $M \simeq M'$, if there exists a stable tame isomorphism $\varphi: \mathcal{B}*S \rightarrow \mathcal{B}'*S$ such\footnote{The following equations do not distinguish notationally between $i$ , $i'$, $f$, $f'$ and  their compositions with $\iota:\mathcal{B} \rightarrow \mathcal{B}*S$ or $\iota':\mathcal{B}' \rightarrow \mathcal{B}'*S'$.} that
\begin{itemize}
\item  $\varphi \circ i = i'$; and 
\item $\varphi \circ f \simeq f'$ (DGA homotopy).
\end{itemize}
Diagramatically, the left half (resp. right half) of
\[
\xymatrix{ & \mathcal{B}*S   \ar[dd]^\varphi   & \\ \mathcal{A}_1 \ar[ru]^f \ar[rd]_{f'} & & \mathcal{A}_2 \ar[lu]_i \ar[ld]^{i'} \\ & \mathcal{B}'*S'} 
\]
is commutative up to DGA homotopy (resp. is commutative).  
\end{definition}

If we use the injections $i$ and $i'$ to view $\mathcal{A}_2$ as a sub-DGA of $\mathcal{B}$ and $\mathcal{B}'$, then the second requirement amounts to having $\varphi|_{\mathcal{A}_2} = \mathit{id}_{\mathcal{A}_2}$.

\begin{lemma} Immersed homotopy defines an equivalence relation on the set\footnote{Since we consider only finitely generated DGAs, the usual categorical requirement that the collection of immersed DGA maps between two DGAs forms a set can be met by restricting $\mathcal{B}$ to have those generators that do not belong to $\mathcal{A}_2$ be elements of some fixed 
 set $X$ such as $\mathbb{N}$.  We will not dwell on this point.}  of immersed DGA maps from $\mathcal{A}_1$ to $\mathcal{A}_2$.
\end{lemma}

\begin{proof}
Both reflexivity and symmetry are easily verified.  
 To verify transitivity, suppose that in addition we are given $\mathcal{A}_1 \stackrel{f''}{\rightarrow} \mathcal{B}'' \stackrel{i''}{\hookleftarrow} \mathcal{A}_2$ and a stable tame isomorphism $\phi':\mathcal{B}'*S_0' \rightarrow \mathcal{B}''*S''$ realizing an immersed homotopy between $\mathcal{A}_1 \stackrel{f'}{\rightarrow} \mathcal{B}' \stackrel{i'}{\hookleftarrow} \mathcal{A}_2$ and $\mathcal{A}_1 \stackrel{f''}{\rightarrow} \mathcal{B}'' \stackrel{i''}{\hookleftarrow} \mathcal{A}_2$.  Then, the stable tame isomorphism 
\[
\psi: \mathcal{B} * S*S_0' \stackrel{ \varphi * \mathit{id}_{S_0'}}{\longrightarrow} \mathcal{B}'* S' * S_0' = \mathcal{B}'* S_0'*S' \stackrel{ \varphi' * \mathit{id}_{S'}}{\longrightarrow} \mathcal{B}''*S''*S'.
\]
clearly restricts to the identity on $\mathcal{A}_2$, and satisfies
\[
\psi \circ f = (\varphi'* \mathit{id}_{S'}) \circ (\varphi* \mathit{id}_{S_0'}) \circ f = (\varphi'* \mathit{id}_{S'}) \circ \varphi \circ f \simeq (\varphi'* \mathit{id}_{S'}) \circ f'= \varphi' \circ f' \simeq f''.
\]

\end{proof}

\medskip

\subsection{Composition}  
\begin{definition} \label{def:Comp}  The {\bf composition}, $M_2 \circ M_1$, of immersed DGA maps $M_1= \big(\mathcal{A}_1 \stackrel{f_1}{\rightarrow} \mathcal{B}_1 \stackrel{i_1}{\hookleftarrow} \mathcal{A}_2\big)$ and $M_2 = \big(\mathcal{A}_2 \stackrel{f_2}{\rightarrow} \mathcal{B}_1 \stackrel{i_2}{\hookleftarrow} \mathcal{A}_3\big)$ is given by the diagram
\[
\xymatrix{ & & \mathcal{B} & & \\ & \mathcal{B}_1 \ar@{.>}[ru]^{p_1} & & \mathcal{B}_2 \ar@{.>}[lu]_{p_2} & \\ \mathcal{A}_1 \ar[ru]^{f_1} & & \mathcal{A}_2 \ar[lu]_{i_1} \ar[ru]^{f_2}  & & \mathcal{A}_3 \ar[lu]_{i_2}   } 
\]
where $\mathcal{B}$ is the categorical pushout of the maps $i_1$ and $f_2$.  Concretely, $\mathcal{B} = (\mathcal{B}_1*\mathcal{B}_2)/I$ where $I = \mathcal{I}(\{i_1(x)-f_2(x) \,|\, x \in \mathcal{A}_2\})$. The differential from $\mathcal{B}_1*\mathcal{B}_2$ satisfies $\partial(I) \subset I$, and therefore induces a differential that makes $\mathcal{B}$ into a DGA.  The dotted arrows $p_i:\mathcal{B}_i \rightarrow \mathcal{B}$ are the composition of the inclusion and projection maps $\mathcal{B}_i \rightarrow \mathcal{B}_1 * \mathcal{B}_2 \rightarrow \mathcal{B}$, and $M_2 \circ M_1=\big(\mathcal{A}_1 \stackrel{p_1\circ f_1}{\rightarrow} \mathcal{B} \stackrel{p_2 \circ i_2}{\hookleftarrow} \mathcal{A}_3\big)$.   

If the generating sets of $\mathcal{A}_2, \mathcal{B}_1$ and $\mathcal{B}_2$  are $\{a_1, \ldots, a_m\}$, $\{a_1, \ldots, a_m, b^1_1, \ldots, b^1_{m_1}\}$, and $\{b^2_1, \ldots, b^2_{m_2}\}$, then we make $\mathcal{B}$ into a triangular DGA using the ordered generating set 
\[
\left\{[b^2_{1}], \ldots, [b^2_{m_2}], [b^1_{1}], \ldots, [b^1_{m_1}]\right\}.
\]
\end{definition}

\begin{lemma} The composition $M_2 \circ M_1 = \big(\mathcal{A}_1 \stackrel{p_1\circ f_1}{\rightarrow} \mathcal{B} \stackrel{p_2 \circ i_2}{\hookleftarrow} \mathcal{A}_3\big)$ is indeed an immersed DGA map from $\mathcal{A}_1$ to $\mathcal{A}_3$.
\end{lemma}
\begin{proof}
Using that $[a_i] = [f_2(a_i)] \in \Z/2\langle [b^2_1], \ldots, [b^2_{m_2}] \rangle$ it is routine to verify that $\mathcal{B}$ is freely generated by $\{[b^2_{1}], \ldots, [b^2_{m_2}], [b^1_{1}], \ldots, [b^1_{m_1}]\}$ and that the differential on $\mathcal{B}$ is triangular.  In addition, since the equivalence classes of all of the generators of $\mathcal{B}_2$ remain generators of $\mathcal{B}$, $p_2 \circ i_2: \mathcal{A}_3 \rightarrow \mathcal{B}$ is an inclusion of based DGAs.
\end{proof}

The following proposition gives an alternate description of $M_2 \circ M_1$ up to immersed homotopy.

\begin{proposition}  \label{lem:CompositionAlt}  
Let $M_1 = \big(\mathcal{A}_1 \stackrel{f_1}{\rightarrow} \mathcal{B}_1 \stackrel{i_1}{\hookleftarrow} \mathcal{A}_2\big)$ and $M_2=\big(\mathcal{A}_2 \stackrel{f_2}{\rightarrow} \mathcal{B}_2 \stackrel{i_2}{\hookleftarrow} \mathcal{A}_3\big)$ be immersed DGA maps.

\begin{enumerate}
\item On the algebra $\mathcal{D} =  \mathcal{B}_1* \widehat{\mathcal{A}}_2* \mathcal{B}_2$ with ordered generating set
\[
\{b^2_1, \ldots, b^2_{m_2}, a_1, \ldots, a_m, \widehat{a}_1, \ldots, \widehat{a}_m, b^1_1, \ldots, b^1_{m_1}\}, \quad |\widehat{a}_i| = |a_i| +1
\]
there exists a differential $\partial$ making $(\mathcal{D}, \partial)$ into a triangular DGA and satisfying 
\begin{itemize}
\item $\mathcal{B}_1$ and $\mathcal{B}_2$ are sub-DGAs of $\mathcal{D}$; and
\item for each of the generators $\widehat{a}_1, \ldots, \widehat{a}_m$ of $\widehat{\mathcal{A}}_2$, we have 
\begin{equation} \label{eq:Ddiff}
\partial \widehat{a}_i = i_1(a_i) + f_2(a_i) + \gamma_i
\end{equation}
 where $\gamma_i \in \mathcal{I}(\widehat{a}_1, \ldots, \widehat{a}_m)$ (2-sided ideal generated by the $\widehat{a}_i$). 
\end{itemize} 
\item For any such $\partial$, the immersed map $\mathcal{A}_1 \stackrel{f_1}{\rightarrow} \mathcal{D} \stackrel{i_2}{\hookleftarrow} \mathcal{A}_3$ 
is immersed homotopic to the composition $M_2 \circ M_1$.
\end{enumerate}

\end{proposition}

\begin{proof}  (1):  To construct such a DGA $(\mathcal{D}, \partial)$ consider the obvious isomorphism of algebras
\[
\mathcal{D} = \mathcal{B}_1* \widehat{\mathcal{A}}_2* \mathcal{B}_2 \cong  (\mathcal{B}_1 * \mathcal{C})/I
\] 
where $\mathcal{C} = \mathcal{A}_2* \widehat{\mathcal{A}}_2 * \mathcal{B}_2$ and $I = \mathcal{I}(\{i_1(x) - j(x) \,|\, x \in \mathcal{A}_2\})$ with $i_1:\mathcal{A}_2 \hookrightarrow \mathcal{B}_1$ and $j: \mathcal{A}_2 \hookrightarrow \mathcal{C}$ the inclusions.  Any mapping cylinder differential $\partial_\mathcal{C}:\mathcal{C} \rightarrow \mathcal{C}$ for $f_2$ (for instance the standard one from Proposition \ref{prop:standard}) together with the differential from $\mathcal{B}_1$ induces a differential on $(\mathcal{B}_1 * \mathcal{C})/I$ such that the corresponding differential on $\mathcal{D}$ has the required properties. 

\medskip

(2):  We need to show that $\mathcal{A}_1 \stackrel{f_1}{\rightarrow} \mathcal{D} \stackrel{i_2}{\hookleftarrow} \mathcal{A}_3$ is immersed homotopic to $\mathcal{A}_1 \stackrel{p_1\circ f_1}{\rightarrow} \mathcal{B} \stackrel{p_2 \circ i_2}{\hookleftarrow} \mathcal{A}_3$ as constructed above with $\mathcal{B} =(\mathcal{B}_1*\mathcal{B}_2)/I$ where $I = \mathcal{I}(\{i_1(x) -f_2(x)\,|\, x\in \mathcal{A}_2 \})$.  
First, applying Proposition \ref{prop:cancel} inductively, we can obtain a stable tame isomorphism
\[
\varphi: \mathcal{D} \rightarrow \mathcal{D}/J *S
\]
with $J = \mathcal{I}(\widehat{a}_1, \partial \widehat{a}_1, \ldots, \widehat{a}_m, \partial \widehat{a}_m)$ satisfying
\begin{equation} \label{eq:pivarphiprop}
\pi \circ \varphi = q \quad \mbox{and} \quad  \varphi|_{\mathcal{B}_2} = q|_{\mathcal{B}_2}
\end{equation}
where $\pi: \mathcal{D}/J *S \rightarrow \mathcal{D}/J$ and $q:\mathcal{D} \rightarrow \mathcal{D}/J$ are the projections.  

[The inductive procedure is the following:  For $0 \leq k \leq m$, let $J_k = \mathcal{I}(\widehat{a}_1, \partial \widehat{a}_1, \ldots, \widehat{a}_k, \partial \widehat{a}_k)$.  Assume inductively that $\varphi_k:\mathcal{D} \rightarrow \mathcal{D}/J_k * S_k$ is a stable  tame isomorphism satisfying $\pi_k \circ \varphi_k = q_k$ and $\varphi_k|_{\mathcal{B}_2} = q_k$ with respect to the quotient maps $\pi_k:\mathcal{D}/J_k * S_k \rightarrow \mathcal{D}/J_k$ and $q_k:\mathcal{D} \rightarrow \mathcal{D}/J_k$, and also  
 that $\mathcal{D}/J_k$ is a triangular DGA generated by equivalence classes of generators of $\mathcal{D}$ other than $\widehat{a}_1, i_1(a_1), \ldots ,\widehat{a}_k, i_1(a_k)$.  Using (\ref{eq:Ddiff}) and the triangularity of $(\mathcal{D},\partial)$, we have $\partial \widehat{a}_{k+1} = i_1(a_{k+1}) + f_{2}(a_{k+1}) + \gamma_{k+1}$ with $\gamma_{k+1} \in J_k$.  Thus, in $\mathcal{D}/J_k$ we can apply Proposition \ref{prop:cancel} to the generators $[\widehat{a}_{k+1}]$ and $[i_1(a_{k+1})]$ to produce a stable tame isomorphism 
\[
\varphi':\mathcal{D}/J_k \rightarrow (\mathcal{D}/J_k)/\mathcal{I}([\widehat{a}_{k+1}], [\partial \widehat{a}_{k+1}]) *S' = \mathcal{D}/J_{k+1}*S'
\] that allows us to complete the inductive step by putting $\varphi_{k+1} = (\varphi'*\mathit{id}_{S_k}) \circ \varphi_k$.  To verify that $\pi_{k+1} \circ \varphi_{k+1} = q_{k+1}$, observe that
\[
\pi_{k+1} \circ \varphi_{k+1} = \pi'\circ \varphi' \circ \pi_k \circ \varphi_k = (q_{k,k+1}) \circ q_k = q_{k+1}
\] 
where $\pi': \mathcal{D}/J_{k+1}*S' \rightarrow \mathcal{D}/J_{k+1}$ and $q_{k,k+1}: \mathcal{D}/J_{k} \rightarrow \mathcal{D}/J_{k+1}$ are projection maps and the equalities $\pi' \circ \varphi' = q_{k,k+1}$ and $\pi_k \circ \varphi_k = q_k$ are Proposition \ref{prop:observe} (1) and the inductive assumption.   
To check that $\varphi_{k+1}|_{\mathcal{B}_2} = q_{k+1}$, note that the map $H: \mathcal{D}/J_k \rightarrow \mathcal{D}/J_k$ from statement of Proposition \ref{prop:cancel} vanishes on the sub-DGA $q_{k}(\mathcal{B}_2) \subset \mathcal{D}/J_k$ so that equations (\ref{eq:gh1}) and (\ref{eq:gh2}) allow us to see that for any $x \in \mathcal{B}_2$,
\[
\varphi_{k+1}^{-1}(q_{k+1}(x)) = q_k(x) + (\partial \circ H + H \circ \partial)(q_k(x)) = q_k(x).\quad \big] \]  

Note that from (\ref{eq:Ddiff}) we see that
\[
J = \mathcal{I}( \widehat{a}_1, \, i_1(a_1) - f_2(a_1), \ldots, \, \widehat{a}_m, \, i_1(a_m) -f_2(a_m) ),
\]
and this shows that there is an additional isomorphism $\beta: \mathcal{D}/J \rightarrow (\mathcal{B}_1*\mathcal{B}_2)/I$ that just maps $[\widehat{a}_i] \mapsto [0]$ and $[x] \mapsto [x]$ when $x \in \mathcal{B}_1$ or $x \in \mathcal{B}_2$.  Moreover, this is a DGA isomorphism (since, for $x \in \mathcal{B}_1 * \mathcal{B}_2$ in either $\mathcal{D}/J$ or $(\mathcal{B}_1* \mathcal{B}_2)/I$, the differential of $[x]$ is induced from the differentials on $\mathcal{B}_1$ and $\mathcal{B}_2$.)  In particular, 
\[
\overline{\varphi} = (\beta * \mathit{id}_S)\circ \varphi : \mathcal{D} \rightarrow (\mathcal{B}_1 *\mathcal{B}_2)/I *S
\]
 is a stable isomorphism.  To complete the proof we check that  $\overline{\varphi}$ produces the desired immersed homotopy.  First, compute
\begin{align*}
\overline{\varphi} \circ i_2 = (\beta * \mathit{id}_S)\circ \varphi \circ i_2 & = (\beta * \mathit{id}_S)\circ q|_{\mathcal{B}_2} \circ i_2 = \beta \circ q|_{\mathcal{B}_2} \circ i_2 = p_2 \circ i_2.
\end{align*}
[At the 2nd and 3rd equalities we used $\mbox{im}\, i_2 \subset \mathcal{B}_2$ together with (\ref{eq:pivarphiprop}) and then $\mbox{im}\, q \subset \mathcal{D}/J$.] 
Finally, using that $\iota:\mathcal{D}/J \rightarrow \mathcal{D}/J*S$ and $\pi: \mathcal{D}/J*S \rightarrow \mathcal{D}/J$ are homotopy inverse as well as (\ref{eq:pivarphiprop}), we find
\begin{align*}
\overline{\varphi} \circ f_1 \simeq (\mathcal{\beta} *\mathit{id}_S) \circ \iota \circ \pi \circ \varphi \circ f_1 = (\beta \circ q|_{\mathcal{B}_1}) \circ f_1 = p_1 \circ f_1.
\end{align*}
\end{proof}

\begin{proposition} \label{prop:Comp} Composition of immersed maps gives a well-defined, associative operation on immersed homotopy classes.
\end{proposition}

\begin{proof}  Up to changing $\mathcal{B}$ by a canonical isomorphism, composition of immersed maps is already associative before passing to homotopy classes.   Turning to well definedness, we suppose now that for $k=1,2$, we have immersed maps
\[
M_k = \left(\mathcal{A}_k \stackrel{f_k}{\rightarrow} \mathcal{B}_k \stackrel{i_k}{\hookleftarrow} \mathcal{A}_{k+1}\right) \quad \mbox{and} \quad M_k' = \left(\mathcal{A}_k \stackrel{f'_k}{\rightarrow} \mathcal{B}'_k \stackrel{i'_k}{\hookleftarrow} \mathcal{A}_{k+1}\right).
\]

\medskip

\noindent {\bf Step 1:} Show that if $M_1 \simeq M_1'$, then $M_2 \circ M_1 \simeq M_2 \circ M_1'$.

\medskip

Given an immersed homotopy
\[
\raisebox{1.5cm}{\xymatrix{ & \mathcal{B}_1*S   \ar[dd]^{\varphi_1}   & \\ \mathcal{A}_1 \ar[ru]^{f_1} \ar[rd]_{f'_1} & & \mathcal{A}_2 \ar[lu]_{i_1} \ar[ld]^{i_1'} \\ & \mathcal{B}_1'*S'}}, \quad \quad \begin{array}{l} \varphi_1 \circ f_1 \simeq f_1', \\  \\ \varphi_1 \circ i_1 = i_1', \end{array}
\]
observe that since $\varphi_1 \circ i_1 = i_1'$, $\varphi_1*\mathit{id}_{\mathcal{B}_2}: (\mathcal{B}_1 * S)* \mathcal{B}_2 \rightarrow (\mathcal{B}_1' * S')* \mathcal{B}_2$ induces a well-defined tame isomorphism 
\[
\varphi: (\mathcal{B}_1* S*\mathcal{B}_2)/I_S \rightarrow (\mathcal{B}_1'*S'*\mathcal{B}_2)/I_S'
\]
where $I_S$ and $I_S'$ are the two-sided ideals {\it in $(\mathcal{B}_1* S*\mathcal{B}_2)$ and $(\mathcal{B}_1'*S'*\mathcal{B}_2)$} generated by $\{i_1(x) -f_2(x) \,|\, x \in \mathcal{A}_2\}$ and $\{i_1'(x) -f_2(x) \,|\, x \in \mathcal{A}_2\}$ respectively.  Using $I$ and $I'$ for the ideals {\it in $\mathcal{B}_1*\mathcal{B}_2$ and $\mathcal{B}_1*\mathcal{B}_2$} with the same generating sets as $I_S$ and $I_S'$, we can use the canonical isomorphisms
\begin{equation} \label{eq:canonicaliso}
(\mathcal{B}_1* S*\mathcal{B}_2)/I_S = (\mathcal{B}_1*\mathcal{B}_2)/I * S \quad \mbox{and} \quad (\mathcal{B}_1* S'*\mathcal{B}_2')/I_S' = (\mathcal{B}_1*\mathcal{B}_2)/I' * S'
\end{equation}
to view $\varphi$ as a tame isomorphism, $\varphi:(\mathcal{B}_1*\mathcal{B}_2)/I * S \rightarrow (\mathcal{B}_1*\mathcal{B}_2)/I' * S'$.  
The required identities $\varphi \circ (p_1 \circ f_1) \simeq p'_1 \circ  f_1'$ and $\varphi \circ (p_2 \circ i_2) = p_2' \circ i_2$ are readily verified. [For the first, compute $\varphi \circ (p_1 \circ f_1) = p_1' \circ \varphi_1 \circ f_1 \simeq p_1' \circ f_1'$; for the second, compute $\varphi\circ (p_2 \circ i_2) = p_2' \circ \mathit{id}_{\mathcal{B}_2} \circ i_2$.]

\medskip 

\noindent {\bf Step 2:} Show that if $M_2 \simeq M_2'$, then $M_2 \circ M_1 \simeq M_2' \circ M_1$.

This step is more involved and is based on the following.

\begin{lemma}  \label{lem:M2}
Suppose we are given triangular DGAs $\mathcal{A}_2 \subset \mathcal{B}_1$, $\mathcal{B}_2$, and $\mathcal{B}'_2$ with respective generating sets $\{a_1, \ldots, a_m\}$, $\{a_1, \ldots, a_m, b^1_1, \ldots, b^1_{m_1}\}$,  $\{b^2_1, \ldots, b^2_{m_2}\}$, and $\{(b^2_1)', \ldots, (b^2_{m_2})'\}$ and DGA maps 
\[
\xymatrix{ & & & \mathcal{B}_2 \ar[dd]^{h_2} &  \\ \mathcal{A}_1 \ar[r]^{f_1} &  \mathcal{B}_1 & \mathcal{A}_2   \ar[l]_{i_1} \ar[ru]^{f_2} \ar[rd]_{f'_2}  & \\ & & & \mathcal{B}_2'  & } 
\]
such that 
\begin{itemize}
\item $i_1$ is the inclusion map, 
\item $h_2 \circ f_2 \simeq f_2'$, and
\item $h_2$ is a tame isomorphism.
\end{itemize}
Let $\mathcal{C} = (\mathcal{B}_1* \mathcal{B}_2)/I$ and $\mathcal{C}'=(\mathcal{B}_1 * \mathcal{B}_2')/I'$, with $I= \mathcal{I}(\{i_1(x)-f_2(x) \,|\, x \in \mathcal{A}_2\})$ and $I'=\mathcal{I}(\{i_1(x)-f_2'(x) \,|\, x \in \mathcal{A}_2\})$, be equipped with the free generating sets given by the equivalence classes $\{[b^2_1], \ldots, [b^2_{m_2}], [b^1_1], \ldots, [b^1_{m_1}]\}$ and $\{[(b^2_1)'], \ldots, [(b^2_{m_2})'], [b^1_1], \ldots, [b^1_{m_1}]\}$.
 Then, there is a tame isomorphism $\varphi:\mathcal{C} \rightarrow \mathcal{C}'$ such that
\begin{enumerate}
\item $\varphi \circ p_2 = p_2' \circ h_2$ where $p_2:\mathcal{B}_2 \rightarrow \mathcal{C}$ and $p_2':\mathcal{B}_2' \rightarrow \mathcal{C}'$ denote the projections;
\item $\varphi \circ (p_1 \circ f_1) \simeq (p_1' \circ f_1)$ where $p_1:\mathcal{B}_1 \rightarrow \mathcal{C}$ and $p_1':\mathcal{B}_1 \rightarrow \mathcal{C}'$ denote the projections; and
\item for $1 \leq i \leq m_1$, 
\[
\varphi([b^1_i]) = [b^1_i] + X_i
\]
where $X_i \in (\mathcal{C}')^{i-1}$ belongs to the subalgebra $ (\mathcal{C}')^{i-1} :=\Z/2\langle [(b^2_1)'], \ldots, [(b^2_{m_2})'], [b^1_1], \ldots, [b^1_{i-1}]\rangle$. 
\end{enumerate}
\end{lemma}

Assuming Lemma \ref{lem:M2}, if we are given a stable tame isomorphism $\varphi_2:\mathcal{B}_2*S \rightarrow \mathcal{B}_2'*S'$ that is part of an immersed homotopy $M_2 \simeq M_2'$, we can apply Lemma \ref{lem:M2} with $\varphi_2$ in place of $h_2:\mathcal{B}_2 \rightarrow \mathcal{B}_2'$ to produce (in conjunction with canonical isomorphisms as in (\ref{eq:canonicaliso})) a tame isomorphism $\varphi:(\mathcal{B}_1*\mathcal{B}_2)/I*S \rightarrow  (\mathcal{B}_1*\mathcal{B}'_2)/I*S'$.  
It is then straightforward to check that $\varphi$ produces the required immersed DGA homotopy $M_1 \circ M_2 \simeq M_1 \circ M_2'$.  [The identity $\varphi \circ p_2 = p_2' \circ \varphi_2$ results from (1) of Lemma \ref{lem:M2} and allows us to compute
\[
\varphi \circ (p_2 \circ i_2) = p_2' \circ \varphi_2 \circ i_2 = p_2' \circ i_2';
\]
that $\varphi\circ (p_1 \circ f_1) \simeq (p_1' \circ f_1)$ is just (2) of Lemma \ref{lem:M2}.]
\end{proof}

\begin{proof}[Proof of Lemma \ref{lem:M2}.]
The construction of $\varphi$ will factor through a third DGA, $\mathcal{D}$, associated to $\mathcal{B}_1$ and $\mathcal{B}_2$ as in Proposition \ref{lem:CompositionAlt}.  Let $\mathcal{D} = \mathcal{B}_1 * \widehat{\mathcal{A}}_2 * \mathcal{B}_2$ be equipped with the differential $\partial$ satisfying
\[
\partial|_{\mathcal{B}_k} = \partial_{\mathcal{B}_k}, \quad \partial( \widehat{a}_i ) = f_2(a_i) + i_1(a_i) + \Gamma \circ \partial_{\mathcal{A}_2} (a_i)
\]
where $\Gamma: \mathcal{A}_2 \rightarrow \mathcal{D}$ is the $(f_2, i_1)$-derivation with $\Gamma(a_j) = \widehat{a}_j$ for any generator $a_j$ of $\mathcal{A}_2$.  

\medskip

\noindent {\bf Step 1.}  Construct a DGA map $\psi: \mathcal{D} \rightarrow \mathcal{C}'$.

\medskip

Since $h_2 \circ f_2 \simeq f_2'$, there exists a $(h_2\circ f_2, f_2')$-derivation $K:\mathcal{A}_2 \rightarrow \mathcal{B}_2'$ satisfying 
\[
h_2 \circ f_2 - f_2' = \partial_{\mathcal{B}'_2} \circ K + K \circ \partial_{\mathcal{A}_2}.
\]
 Define $\psi: \mathcal{D} \rightarrow \mathcal{C}'$  
to satisfy
\begin{equation} \label{eq:defpsi}
\psi|_{\mathcal{B}_1} = p_1' \circ \mathit{id}_{\mathcal{B}_1}, \quad \psi|_{\mathcal{B}_2} = p_2' \circ h_2, \quad \mbox{and} \quad \psi(\widehat{a}_i) = [K(a_i)] 
\end{equation}
for generators $a_i \in \mathcal{A}_2$.  

To check that $\psi$ is a chain map, compute
\begin{align*}
(\psi \circ \partial)(\widehat{a}_i) =& \psi\left( f_2(a_i) + i_1(a_i) + \Gamma \circ \partial_{\mathcal{A}_2} (a_i) \right) \\
 = & [h_2 \circ f_2(a_i)] + [i_1(a_i)] + \psi \circ \Gamma( \partial_{\mathcal{A}_2}(a_i))  \\
 = & [h_2 \circ f_2(a_i) + f_2'(a_i) + K\circ \partial_{\mathcal{A}_2}(a_i)]  \\
 = & [\partial_{\mathcal{B}_2'} \circ K(a_i)] = \partial \circ \psi(\widehat{a}_i).
\end{align*}
At the 3rd equality we used that $i_1(a_i)-f_2'(a_i) \in I'$ and that for any $x \in \mathcal{A}_2$, $\psi\circ \Gamma(x) = [K(x)]$.  This last equality is verified by induction on the word length of $x$:  It is immediate from definitions when $x=a_i$, and when the result is known for $x_1, x_2 \in \mathcal{A}_2$ we can use that $\Gamma$ is a $(f_2, i_1)$-derivation  and $K$ is a $(h_2\circ f_2, f_2')$-derivation to compute
\begin{align*}
\psi\circ \Gamma(x_1\cdot x_2) =& \psi\circ \Gamma(x_1) \cdot \psi \circ i_1(x_2) + \psi\circ f_2(x_1) \cdot \psi \circ \Gamma(x_2) \\
 =& [K(x_1)] \cdot [i_1(x_2)] + [h_2 \circ f_2(x_1)] \cdot [K(x_2)] \\
=& [K(x_1)\cdot f_2'(x_2) +h_2 \circ f_2(x_1) \cdot K(x_2)] = [K(x_1\cdot x_2)],
 \end{align*}
completing the induction.

\medskip

\noindent {\bf Step 2.}  Construct a DGA map $g = g_1 \circ g_2 \circ \cdots \circ g_m \circ \alpha: \mathcal{C} \rightarrow \mathcal{D}$ satisfying 
\begin{enumerate}
\item[(i)]  $g([b^2_i]) = b^2_i$, for $1 \leq i \leq m_2$, and
\item[(ii)] $g([b^1_i]) = b^1_i + W_i$ where $W_i \in \mathcal{D}^{i-1} := \Z/2\langle b_1^2, \ldots, b^2_{m_2}, a_1, \ldots, a_m, \widehat{a}_1, \ldots, \widehat{a}_m, b_1^1, \ldots, b^1_{i-1}\rangle$.
\end{enumerate}

\medskip

As in Proposition \ref{lem:CompositionAlt}, consider the sequence of quotients $\mathcal{D}_0 = \mathcal{D}, \ldots, \mathcal{D}_m = \mathcal{D}/J_m$ where $J_l = \mathcal{I}(\{\widehat{a}_i, \partial \widehat{a}_i \,|\, 1 \leq i \leq l\})$.  As in the proof of Proposition \ref{lem:CompositionAlt}, the map
\[
 \alpha:\mathcal{C} \rightarrow \mathcal{D}_m, \quad \alpha([x]) = [x]
\]
is a DGA isomorphism.
 For $1 \leq l \leq m$, (since $\Gamma \circ \partial(a_l) \in J_{l-1}$) we have
\[
\mathcal{D}_l = \mathcal{D}_{l-1}/\mathcal{I}( [\widehat{a}_l], \partial [\widehat{a}_l] = [i_1(a_l)] + [f_2(a_l)]),
\]
so that $\mathcal{D}_l$ is triangular with respect to the ordered generating set
\[
\mathcal{D}_l = \Z/2\langle [b^2_1], \ldots, [b^2_{m_2}], [a_{l+1}], \ldots, [a_m], [\widehat{a}_{l+1}], \ldots, [\widehat{a}_m], [b^1_1], \ldots, [b^1_{m_1}] \rangle.
\]
Moreover, Proposition \ref{prop:cancel} provides a DGA map $g_l:\mathcal{D}_l \rightarrow \mathcal{D}_{l-1}$ such that
\begin{equation} \label{eq:glql}
g_l \circ q_l - \mathit{id}_{\mathcal{D}_{l-1}} = \partial \circ H_l + H_{l} \circ \partial
\end{equation} 
where $q_l: \mathcal{D}_{l-1} \rightarrow \mathcal{D}_l$ is the quotient map and $H_{l}: \mathcal{D}_{l-1} \rightarrow \mathcal{D}_{l-1}$ is a $(g_l \circ q_l, \mathit{id}_{\mathcal{D}_{l-1}})$-derivation with 
\[
H_{l}([a_l]) = [\widehat{a}_l] \quad \mbox{and} \quad H_l(x) = 0
\]
when $x \in \mathcal{D}_{l-1}$ is any generator other than $a_l$.

To verify (i), we check that $g_l([b_i^2]) = [b_i^2]$ for any $1 \leq l \leq m$ and $1 \leq i \leq m_2$.  Compute using (\ref{eq:glql}) 
\[
g_l([b_i^2]) = g_l \circ q_l([b^2_i]) = [b^2_i] + \partial H_l([b^2_i]) + H_l \partial([b^2_i]) = [b^2_i]
\]
where at the second equality we used that $\partial[b^2_i] \in \Z/2\langle [b_1^2], \ldots, [b^2_{m_2}] \rangle$ and $H$ vanishes on this subalgebra.

To verify (ii), for $1 \leq i \leq m_1$, we prove  using a decreasing induction on $l$ with $1 \leq l \leq n+1$ that
\begin{align*}
 &g_l \circ \cdots \circ g_n \circ \alpha([b_i^1]) = [b^1_i] + w_{l-1} \quad \quad \quad \mbox{where} \\
 & w_{l-1} \in \mathcal{D}_{l-1}^{i-1} = \Z/2\langle [b_1^2], \ldots, [b^2_{m_2}], [a_l], \ldots, [a_m], [\widehat{a}_l], \ldots, [\widehat{a}_m], [b^1_1], \ldots, [b^1_{i-1}] \rangle.
\end{align*}
When $l =n+1$, we have $w_n = 0$ since $\alpha([b^1_i]) = [b^1_i]$.  For the inductive step, compute
\begin{align*}
g_l([b_i^1] + w_{l}) &= g_l\circ q_l([b^1_i]) + g_l(w_l)  \\
 &= [b^1_i] + \partial H_l ([b^1_i]) + H_l\partial([b^1_i]) + g_l(w_l) = [b^1_i] + H_l(\partial[b^1_i]) + g_l(w_l).
\end{align*}
Since in the last sum $\partial[b^1_i] \in \mathcal{D}_{l-1}^{i-1}$ and $w_l \in \mathcal{D}_l^{i-1}$, we can set $w_{l-1} = H_l(\partial[b^1_i]) + g_l(w_l)$, and then apply the following claim to verify that $w_{l-1} \in \mathcal{D}_{l-1}^{i-1}$.

\medskip

\noindent {\bf Claim:}  For all $1 \leq i \leq m_1$,
\begin{align}
\label{eq:Claim1} &H_l(\mathcal{D}_{l-1}^{i-1}) \subset \mathcal{D}_{l-1}^{i-1} \quad\quad \quad \mbox{and} \\
\label{eq:Claim2} &g_l(\mathcal{D}_l^{i-1}) \subset \mathcal{D}_{l-1}^{i-1}.
\end{align}

\medskip

In turn, the claim is verified by (increasing) induction on $i$.  First, we check (\ref{eq:Claim2}).  Given a generator of $\mathcal{D}_l^{i-1}$, 
\[
[x] \in \{[b^2_1], \ldots, [b^2_{m_2}], [a_{l+1}], \ldots, [a_m], [\widehat{a}_{l+1}], \ldots, [\widehat{a}_m], [b^1_1], \ldots, [b^1_{i-1}] \}
\]
compute 
\[
g_l([x]) = g_l \circ q_l([x]) = [x] + \partial H_l([x]) + H_l \partial([x]) = [x] + H_l(\partial[x]),
\]
and since $\partial[x] \in \mathcal{D}_{l-1}^{i-2}$ equation (\ref{eq:Claim1}) from the inductive hypothesis shows that $g_l([x]) \in \mathcal{D}_{l-1}^{i-2}$ so that $g_l([x]) \in \mathcal{D}_{l-1}^{i-1}$.  With (\ref{eq:Claim2}) established, (\ref{eq:Claim1}) is verified by first checking $H_l(z) \in \mathcal{C}_{l-1}^{i-1}$ when $z \in \mathcal{C}_{l-1}^{i-1}$ is a generator.  (This is immediate from the definition, since $H_l$ vanishes on all generators other than $a_l$ and satisfies $H_l(a_l) = \widehat{a}_l$.)  Then, we can verify (\ref{eq:Claim1})  for arbitrary words in generators of $\mathcal{C}_{l-1}^{i-1}$ using induction.  To this end, note that 
\[
H_l(y\cdot z) = H_l(y) \cdot z + (g_l \circ q_l)(y) \cdot H_l(z) 
\]
and $q_l( \mathcal{D}_{l-1}^{i-1}) \subset \mathcal{D}_l^{i-1})$.  (This is verified on generators, since $q_l([\widehat{a}_l]) = 0$, $q_l([a_l]) = [f_2(a_l)] \in \Z/2\langle[b^2_1], \ldots, [b^2_{m_2}] \rangle \subset \mathcal{D}_l^{i-1}$.)  Thus, (\ref{eq:Claim2}) implies that $(g_l \circ q_l)(y) \in \mathcal{D}_{l-1}^{i-1}$.

\medskip

\noindent {\bf Step 3.}  Verify that $\varphi := \psi \circ g : \mathcal{C} \rightarrow \mathcal{C}'$ has the required properties.

\medskip

Note that (1) and (3) imply that $\varphi$ is a tame isomorphism.  \big[Given (1) and (3), $\varphi$ can be rewritten as a composition $\varphi_{m_1} \circ \cdots \circ \varphi_1 \circ \varphi_{0}$
 where $\varphi_i: \mathcal{C}' \rightarrow \mathcal{C}'$ with $1 \leq i\leq m_1$ is the elementary isomorphism that maps $[b^1_i] \mapsto [b^1_i] + X_i$ and fixes all other generators; and $\varphi_0: \mathcal{C} \rightarrow \mathcal{C}'$ fixes all $[b^1_i]$ generators and maps $[b^2_i] \mapsto [h_2(b^2_i)]$.  (Note that $\varphi_0$ is a tame isomorphism since $h_2$ is.)\big]

To verify (1), use (\ref{eq:defpsi}) and (i) from Step 2 to compute for any generator $b_i^2 \in \mathcal{B}_2$
\[
\varphi \circ p_2(b_i^2) = \psi \circ g([b_i^2]) = \psi(b_i^2) = p_2' \circ h_2(b_i^2).
\]
To verify (3), start with  (\ref{eq:defpsi}) and (ii) from Step 2 to get
\[
\varphi([b^1_i]) = \psi(b_i^1 + W_i) = [b_i^1]+ \psi(W_i),
\]
and then observe that $\psi(\mathcal{D}^{i-1}) \subset (\mathcal{C}')^{i-1}$ since $\psi(a_i) = [i_1(a_i)] = [f_2(a_i)] \in (\mathcal{C}')^0=\Z/2\langle [b_1^2], \ldots, [b^2_{m_2}] \rangle$ and $\psi(\widehat{a}_i) = [K(a_i)] \in (\mathcal{C}')^0$. 

Finally, to verify (2) observe the identity $\alpha \circ p_1 = q \circ j$ where 
\[
\mathcal{B}_1 \stackrel{j}{\rightarrow} \mathcal{D} \stackrel{q}{\rightarrow} \mathcal{D}_m
\]
are inclusion and projection, and then use that $q_l$ and $g_l$ are homotopy inverse to compute
\begin{align*}
\varphi \circ (p_1 \circ f_1) = \psi \circ g_1 \circ \cdots g_m \circ \alpha \circ p_1 \circ f_1 &= \psi \circ g_1 \circ \cdots \circ g_m \circ q \circ j \circ f_1 \\
& = \psi \circ g_1 \circ \cdots \circ g_m \circ q_m \circ \cdots \circ q_1 \circ j \circ f_1 \\
& \simeq \psi \circ j \circ f_1 = p_1' \circ f_1. 
\end{align*}
(At the last equality use (\ref{eq:defpsi}).)
\end{proof}

\subsection{The immersed DGA category}

{\it Let $n \geq 0$ be a fixed non-negative integer.}  We now define a category $\mathfrak{DGA}^n_{\mathit{im}}$  that we refer to as the {\bf immersed DGA category}.  The objects of $\mathfrak{DGA}^n_{\mathit{im}}$ are finitely generated, triangular DGAs over $\Z/2$ that are graded by $\Z/n$.  Morphisms in $\mathfrak{DGA}^n_{\mathit{im}}$ are immersed homotopy classes of immersed DGA maps.  As in Proposition \ref{prop:Comp}, composition of immersed maps is well-defined on homotopy classes, and this provides the composition operation for $\mathfrak{DGA}^n_{\mathit{im}}$.  

\begin{proposition}
As defined, $\mathfrak{DGA}^n_{\mathit{im}}$ is a category.
\end{proposition}
\begin{proof}
This follows from Proposition \ref{prop:Comp} and the observation that the immersed maps $\mathcal{A} \stackrel{\mathit{id}}{\rightarrow} \mathcal{A} \stackrel{\mathit{id}}{\hookleftarrow} \mathcal{A}$ are identity morphisms.
\end{proof}

\subsubsection{Connection with the DGA homotopy category} \label{sec:EmbeddedDGA}
Let  $\mathfrak{DGA}^n$ denote the ordinary homotopy category of $\Z/n$-graded, finitely generated, triangular DGAs (where morphisms are just DGA homotopy classes of maps).  The following result allows us to view $\mathfrak{DGA}^n$
as a (non-full) sub-category of $\mathfrak{DGA}^n_{im}$.  

\begin{proposition} \label{prop:InclusionFunctor} There is a functor $I: \mathfrak{DGA}^n \rightarrow \mathfrak{DGA}^n_{im}$ which is the identity on objects and acts on morphisms via
\[
I\left( \left[ \mathcal{A}_1 \stackrel{f}{\rightarrow} \mathcal{A}_2 \right] \right) =  \big[\big(\mathcal{A}_1 \stackrel{f}{\rightarrow} \mathcal{A}_2 \stackrel{\mathit{id}}{\hookleftarrow} \mathcal{A}_2 \big)\big].
\] 
Moreover, $I$ is injective on all $\mathit{Hom}$-spaces.
\end{proposition}

\begin{proof}
Clearly, $I$ preserves identities.  To see that $I$ preserves composition, recall that the immersed map
\[
\big(\mathcal{A}_2 \stackrel{f_2}{\rightarrow} \mathcal{A}_3 \stackrel{\mathit{id}}{\hookleftarrow} \mathcal{A}_3 \big) \circ \big(\mathcal{A}_1 \stackrel{f_1}{\rightarrow} \mathcal{A}_2 \stackrel{\mathit{id}}{\hookleftarrow} \mathcal{A}_2 \big)
\]
is $\mathcal{A}_1 \stackrel{p_1 \circ f_1}{\rightarrow} \mathcal{B} \stackrel{p_2}{\hookleftarrow} \mathcal{A}_3$ where $\mathcal{B} = \mathcal{A}_2*\mathcal{A}_3/\mathcal{I}(\{ x - f_2(x) \,|\, x \in \mathcal{A}_2 \})$.  The isomorphism $\varphi: \mathcal{A}_3 \rightarrow \mathcal{B}, \varphi(x) = [x]$ fits into a (fully) commutative diagram
\[
\xymatrix{ & \mathcal{A}_3   \ar[dd]^\varphi   & \\ \mathcal{A}_1 \ar[ru]^{f_2\circ f_1} \ar[rd]_{p_1 \circ f_1} & & \mathcal{A}_3 \ar[lu]_{\mathit{id}} \ar[ld]^{p_2} \\ & \mathcal{B}} 
\]
which shows that $I(f_2) \circ I(f_1) = I(f_2\circ f_1)$.

To verify injectivity, assume $f,g: \mathcal{A}_1 \rightarrow \mathcal{A}_2$ are such that $I(f) =I(g)$.  Then, we have an immersed homotopy
\[
\xymatrix{ & \mathcal{A}_2*S   \ar[dd]^\varphi   & \\ \mathcal{A}_1 \ar[ru]^{f} \ar[rd]_{g} & & \mathcal{A}_2 \ar[lu]_{\mathit{i}} \ar[ld]^{i'} \\ & \mathcal{A}_2*S'} 
\]
with $\varphi \circ f \simeq g$ and $\varphi \circ i = i'$ where $i$ and $i'$ are the inclusion maps.  The latter equality shows that $\varphi|_{\mathcal{A}_2} = \mathit{id}$, so we actually get $f = \varphi \circ f \simeq g$, i.e. $[f]=[g]$ in $\mathfrak{DGA}^n$.

\end{proof}

\section{Good Lagrangian cobordisms and conical Legendrian cobordisms}
\label{sec:defn}

After briefly reviewing Legendrian submanifolds in $1$-jet spaces in Section \ref{sec:leg},
we introduce good Lagrangian cobordisms in symplectizations 
and show that they can be viewed as conical Legendrian cobordisms in a $1$-jet space over a base space of $1$ higher dimension.   
In Section \ref{sec:Concatenate}, we examine the concatenation operation from these two equivalent vantage points, and then establish a lemma that reduces conical Legendrian isotopy to a combination of compactly supported isotopy and concatenation with certain standard form cobordisms.

\subsection{Legendrian submanifolds in $1$-jet spaces}\label{sec:leg}

Let $E$ be an $n$--dimensional manifold.
We work in the $1$--jet space $J^1E=T^*E\times \R$ with the canonical contact structure $\xi=\ker \alpha$, where $\alpha = dz -y dx$, $x=(x_1,\dots, x_n)$ and $y=(y_1, \dots, y_n)$ are local coordinates for $T^*E$, and $z$ is the coordinate for $\R$.
An example is $\R^3=J^1 \R$ with the standard contact structure $\xi=\ker(dz-y dx)$.

A {\bf Legendrian submanifold} $\Lambda$ is a smooth submanifold  in $J^1E$ such that $\dim(\Lambda) =n$ and $T_p\Lambda \subset \xi_p$ for all $p \in \Lambda$.  
We use two canonical projections:
the {\bf front projection} $\pi_{xz}$ projects $J^1E$ to $J^0E= E\times \R$, and
the {\bf Lagrangian projection} $\pi_{xy}$ projects $J^1E$ to $T^*E$.
An example of the front projection and the Lagrangian projection for Legendrian knots  in $J^1\R$ is shown in Figure \ref{trefoil}.

\begin{figure}[!ht]

\includegraphics[width=4in]{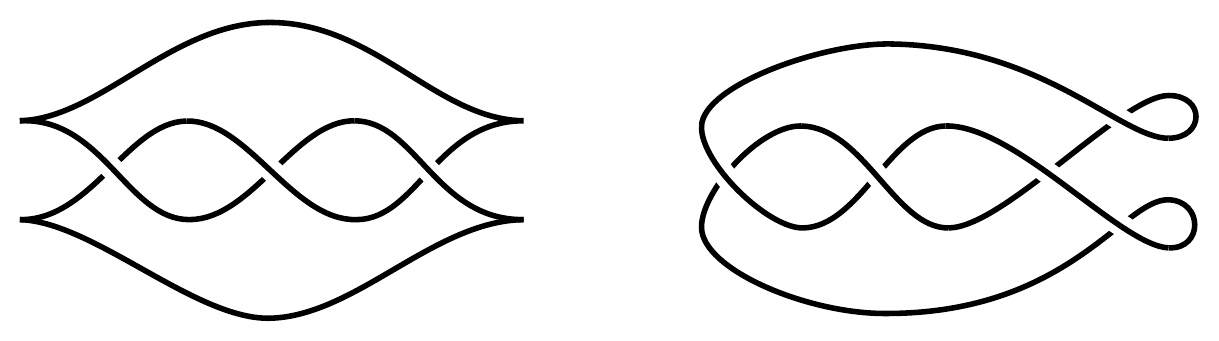}
\caption{The left and right figures are examples of the front projection and the Lagrangian projection of a Legendrian trefoil in $J^1\R$. }
\label{trefoil}
\end{figure}

Any function $f:E \rightarrow \R$ defines a Legendrian via its $1$-jet, $j^1f$, which appears above a coordinate neighborhood $U \subset E$ as $j^1f \cap J^1U = \{(x, \frac{\partial f}{\partial x}(x), f(x)) \,|\, x \in U\}$.  Notice that the front projection of $j^1f$ is simply the graph of $f$.  Moreover, outside of a codimension $1$ singular set, $\Lambda_{\mathit{sing}} \subset \Lambda$, any generic Legendrian $\Lambda \subset J^1E$ locally agrees with some $j^1f$.  When $\dim E = 1$ (resp. $\dim E = 2$) the image of $\Lambda_{\mathit{sing}}$ in the front projection $\pi_{xz}(\Lambda)$ appears as a collection of cusp points (resp. cusp edges and swallowtail points).  

\begin{definition}
For an oriented (possibly immersed) loop $\gamma$ in $\Lambda$ transverse to the singular set of the front projection, the {\bf Maslov index} of $\gamma$ is $$\mu(\gamma)= D(\gamma) - U(\gamma)$$ where $D(\gamma)$ (resp. $U(\gamma)$) is the number of times $\gamma$ passes a cusp on the front projection in the downward (resp. upward) direction.  The Maslov index extends to a homomorphism $\mu:H_1(\Lambda) \rightarrow \Z$. 
The {\bf Maslov number} $m(\Lambda)$ of $\Lambda$ 
is the greatest common factor of the Maslov index of all possible loops in $\Lambda$.

Given $n \in \Z_{\geq 0}$, a $\Z/n$-valued {\bf Maslov potential} for a Legendrian knot or surface $\Lambda$ is a locally constant function $\nu:\Lambda \setminus \Lambda_{\mathit{sing}}\rightarrow \Z/n$ with the property that at cusp points of $\pi_{xz}(\Lambda)$ we have $\nu(S_u) = \nu(S_l)+1$ (mod $n$) where $S_u$ and $S_l$ are upper and lower sheets of $\Lambda$ near the cusp point or cusp edge.  

\end{definition}

\begin{remark} \label{rem:M}
\begin{enumerate}
\item  A Legendrian $\Lambda$ admits a $\Z/n$-valued Maslov potential if and only if $n \,|\, m(\Lambda)$.
\item When $\Lambda$ is of dimension $1$, we have that $m(\Lambda)=2\,\mathit{rot}(\Lambda)$.
\end{enumerate}
\end{remark}

\subsubsection{Some constructions of Legendrian submanifolds}  \label{sec:construct} We will make use of the following methods for constructing and modifying Legendrian submanifolds.

\begin{itemize}

\item {\it Addition by a constant:}  Given a Legendrian $\Lambda \subset J^1E$ and a constant $A \in \R$ we obtain a new Legendrian submanifolds that we denote $j^1(\Lambda + A)$ by applying the contactomorphism $(x,y,z) \mapsto (x,y,z+A)$ to $\Lambda$, i.e by shifting the  $z$-coordinate  of all points of $\Lambda$ by $A$.

\item {\it Addition by a function:}  More generally, given a smooth function $h:E \rightarrow \R$ and $\Lambda \subset J^1E$, we form $j^1(\Lambda+ h(x))$ by applying the contactomorphism defined in local coordinates $(x_1, \ldots, x_n, y_1, \ldots, y_n,z)$ via
\[
(x_1, \ldots, x_n, y_1, \ldots, y_n, z) \mapsto ( x_1, \ldots, x_n, y_1 + \frac{\partial h}{\partial x_1}, \ldots, y_n +\frac{\partial h}{\partial x_n}, z+h(x)).
\]

\item {\it Cylinder with multiplication:}  Let $\Lambda \subset J^1M$ with $\dim M = 1$ and $f:\mathbb{R}_{>0} \rightarrow \mathbb{R}_{>0}$ be a positive function written with variable $f(s)$.  Supposing that $\theta \mapsto (x(\theta), y(\theta), z(\theta))$  parametrizes $\Lambda$, we define $j^1(f(s) \cdot \Lambda) \subset J^1(\R_{>0} \times M)$ to have the parametrization
$$(s,\theta) \rightarrow (x_1,x_2, y_1, y_2, z) = (s,\  x(\theta),\  f'(s)z(\theta),\  f(s)y(\theta),\  f(s)z(\theta)).$$
\end{itemize}
We can combine these operations to obtain Legendrians notated, for instance, as $j^1(f(s) \cdot \Lambda + A)$.  (Form the cylinder, then shift.)

\subsection{Good Lagrangian cobordisms}\label{sec:good}

Let $M$ be a $1$--dimensional manifold. Given two Legendrian submanifolds $\Lambda_+$ and $\Lambda_-$ in $J^1M$, we consider immersed exact Lagrangian cobordisms, which generalize the embedded exact Lagrangian cobordisms considered in \cite{EHK}.

\begin{definition}
Suppose  that $\Lambda_+$ and $\Lambda_-$ are Legendrian submanifolds in $(J^1M, \xi=\ker \alpha)$.
An {\bf (immersed) exact Lagrangian cobordism} $L$  from $\Lambda_-$ to $\Lambda_+$ in the symplectization, $\mathit{Symp}( J^1M)=\big(\R_t\times J^1M, \omega=d(e^t\alpha)\big)$, is  an (immersed) surface $i:L \rightarrow \mathit{Symp}( J^1M)$
 (see Figure \ref{imcob}) such that 
for some $T>0$,
\begin{itemize}
\item when restricted to $i^{-1}\left(\left((-\infty, -T] \cup [T,+\infty)\right) \times J^1M\right)$, $i$ is an embedding with 
\[
i(L)\cap \big([T, \infty) \times J^1M\big) = [T, \infty) \times \Lambda_+  \mbox{ and }  i(L)\cap \big((-\infty,-T] \times J^1M\big) = (-\infty,-T] \times \Lambda_- ,
\]
and $i^{-1}\big([-T, T] \times J^1M\big)$ is compact.
\item 
Moreover, we require that $i^*(e^t \alpha)= d\rho,$  for some smooth function \smallskip $\rho: L \to \R$ which is constant when $t>T$ or $t<-T$. 
The function $\rho$ is called a {\bf primitive}.
\end{itemize}
When the negative end $\Lambda_-$ is empty,  $L$ is an {\bf  (immersed) exact Lagrangian filling} of $\Lambda_+$.
\end{definition}

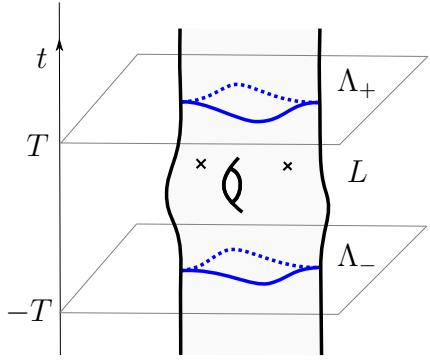
\begin{figure}[!ht]
\labellist
\pinlabel $t$ at -10 350
\pinlabel $T$ at -15 250
\pinlabel $-T$ at -25 60
\pinlabel $\Lambda_+$ at 350 320
\pinlabel $\Lambda_-$ at 350 120
\pinlabel $L$ at 350 220 
\endlabellist
\includegraphics[width=2in]{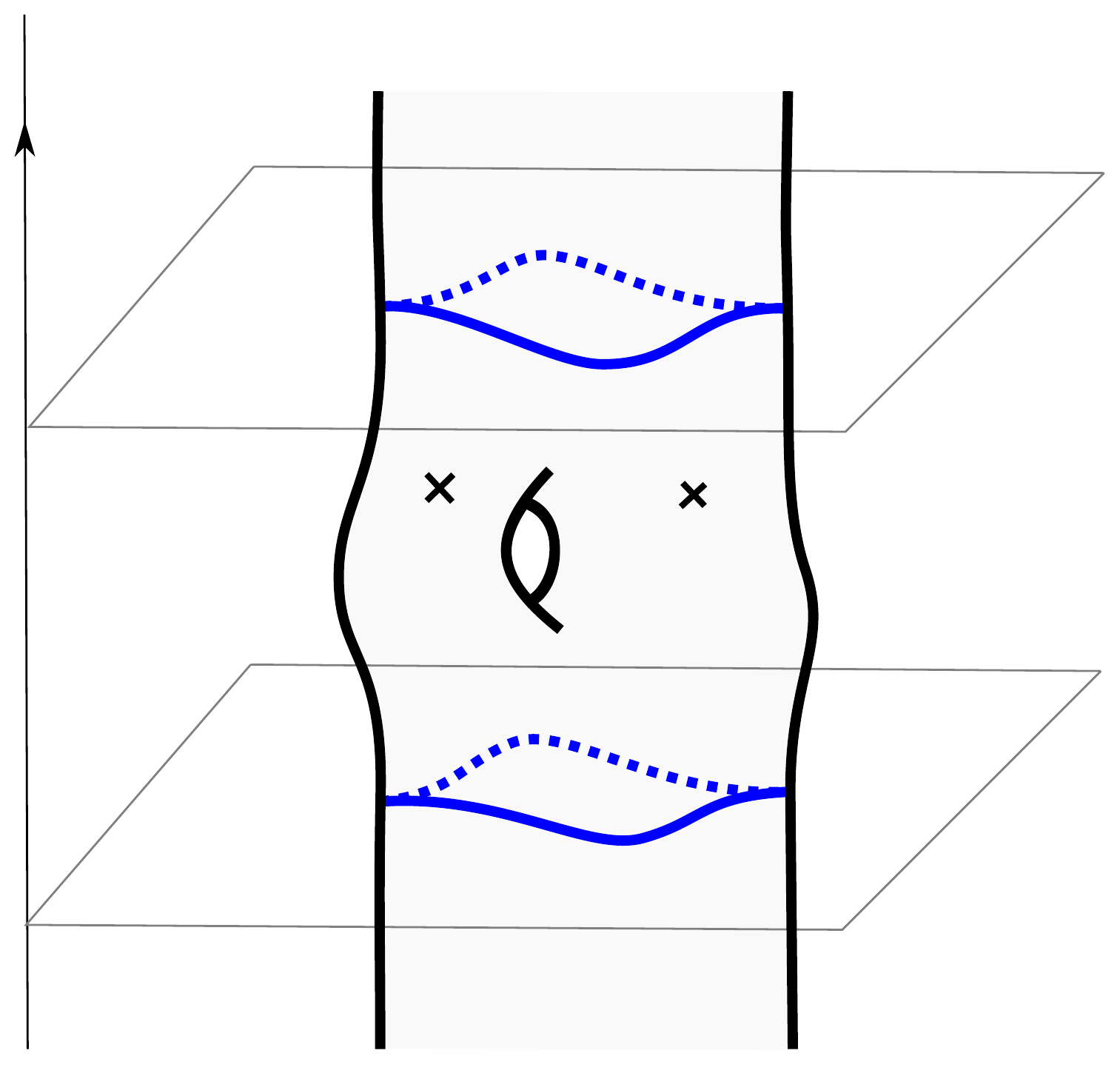}
\caption{An immersed exact Lagrangian cobordism $L$ from $\Lambda_-$ to $\Lambda_+$. Crossings denote  double points of $L$.}
\label{imcob}
\end{figure}

\begin{remark}\begin{itemize}
\item[(i)]  We sometimes suppress the map $i$ from notation and refer to $L$ itself as an immersed exact Lagrangian cobordism.
\item[(ii)]  Note that when $\Lambda_+$ and $\Lambda_-$ are single component knots, the primitive $\rho$ is automatically constant when $t>T$ and $t<-T$.
We require this to be true for the case that $\Lambda_+$ and $\Lambda_-$ are multi-components links. 
If this is not satisfied, the concatenation of two exact Lagrangian cobordisms may not be exact anymore.  See \cite{Ch2015}.
\end{itemize}
\end{remark}

An immersed exact Lagrangian cobordism $L$ in $\mathit{Symp}(J^1M)$ with a primitive $\rho$ can be lifted to a Legendrian surface $\Sigma^{Symp}$ in the contactization of $\mathit{Symp}(J^1M)$, i.e. the contact manifold $\big(\mathit{Symp}(J^1M)\times \R_w, ker(dw+e^t\alpha)\big)$, as the graph of $-\rho$.
We can view $L$ as the Lagrangian projection (to $\mathit{Symp}(J^1M)$)  of the Legendrian surface $\Sigma^{\mathit{Symp}}$.

\begin{definition}
A pair $(L,\rho)$ consisting of an immersed exact Lagrangian cobordism $L$ from $\Lambda_-$ to $\Lambda_+$ together with its primitive $\rho$ is called a {\bf good Lagrangian cobordism} if
its Legendrian lift $\Sigma^{\mathit{Symp}}$  in $\mathit{Symp}(J^1M)\times \R$  is an embedded Legendrian surface.
Two good Lagrangian cobordisms from $\Lambda_-$ to $\Lambda_+$ are called {\bf good Lagrangian isotopic} if they are isotopic through good Lagrangian cobordisms $\Lambda_-$ to $\Lambda_+$, i.e., there exists a smooth 
$1$-parameter family of good Lagrangian cobordisms $(L_t, \rho_t)$ from $\Lambda_-$ to $\Lambda_+$ that varies from one good Lagrangian cobordism to another.  
We may write $(L,\rho): \Lambda_- \rightarrow \Lambda_+$ to indicate that $(L,\rho)$ is a good Lagrangian cobordism from $\Lambda_-$ to $\Lambda_+$.
\end{definition}

After modifying $\Sigma^{\mathit{Symp}}$ by a small Legendrian isotopy (which modifies $L$ by good Lagrangian isotopy) it  can be assumed that $L$ is self-transverse and embedded except for double points.

\begin{remark}
When the exact Lagrangian cobordisms are embedded, the definition of good Lagrangian isotopy recovers the equivalence relation of exact Lagrangian isotopy in \cite{EHK}.
\end{remark}

\subsection{Conical Legendrian cobordisms.}\label{sec:conical}

Note that $\mathit{Symp}( J^1M) \times \R_w$ is contactomorphic to $J^1(\R_{>0}\times M)$ through the following contactomorphism:
\begin{equation} \label{eq:contactomorphism}
\begin{array}{rrcl}
\wt\Phi:& \ \mathit{Symp}( J^1M) \times \R_w& \to & J^1(\R_{> 0}\times M)\\
& (t, x, y, z, w) & \to &\big( (e^t, x), (z, e^t y), e^tz+w\big).\\
\end{array}
\end{equation}
The Legendrian lift $\Sigma^{\mathit{Symp}}$ of a good Lagrangian cobordism $L$ can then be mapped to a Legendrian $\wt\Phi(\Sigma^{\mathit{Symp}})$ in $J^1( \R_{>0}\times M)$; denote it by $\Sigma$.
Recall that when $t>T$ (resp. $t<-T$), $L$ is cylindrical over $\Lambda_+$ (resp. $\Lambda_-$), and the primitive $\rho$ is a constant, say $\rho=A_+$ (resp. $\rho=A_-$). 
Thus, on $J^1\big((e^{T}, \infty)\times M\big)$ and $J^1\big(( 0,e^{-T})\times M\big)$, $\Sigma$ can be parametrized by
\begin{equation}\label{eq:para}
(s, \ x_{\pm}(\theta_{\pm}),\  z_{\pm}(\theta_{\pm}),\  sy_{\pm}(\theta_{\pm}), \ sz_{\pm}(\theta_{\pm})-A_{\pm})
\end{equation}
where $s=e^t$ and $\theta_{\pm}$ parametrizes $\Lambda_{\pm}$ in $J^1 M$ through $(x_{\pm}(\theta_{\pm}), y_{\pm}(\theta_{\pm}), z_{\pm}(\theta_{\pm}))$.  That is, in the notation from Section \ref{sec:construct} the surface $\Sigma$ agrees with $j^1(s\cdot \Lambda_+-A_+)$ (resp. with $j^1(s\cdot \Lambda_--A_-)$ 
 when $s>e^T$ (resp.  $s<e^{-T}$).  This leads us to the following definition.

\begin{definition}\label{def:conical}
Suppose  that $\Lambda_+$ and $\Lambda_-$ are Legendrian submanifolds in $(J^1M, \xi=\ker \alpha)$.
A {\bf conical Legendrian cobordism} $\Sigma$  from $\Lambda_-$ to $\Lambda_+$ is  an (embedded) Legendrian submanifold  in $J^1(\R_{>0}\times M)$ such that 
for some $T>0$ and constants $A_\pm \in \R$,
\begin{itemize}
\item $\Sigma$ agrees with $j^1(s\cdot \Lambda_{+} -A_{+})$
when $s>e^T$,
\item $\Sigma$ agrees with $j^1(s\cdot \Lambda_{-} -A_{-})$
when $s<e^{-T}$, and
\item $\Sigma \cap J^1([e^{-T},e^{T}] \times M)$ is compact.
\end{itemize}

\end{definition}

As one can observe from the equation \eqref{eq:para}, when $s>e^T$ (resp. when $s<e^{-T}$), each $s$--slice of the front projection $\pi_{xz}(\Sigma)$ looks like 
$\pi_{xz}(\Lambda_{+})$ (resp. $\pi_{xz}(\Lambda_{-})$ ) with a linear transformation applied to the $z$-coordinate
and each $x$-slice of $\pi_{xz}(\Sigma)$ is a family of linear functions of $s$.  See Figure \ref{fg:conLeg}.

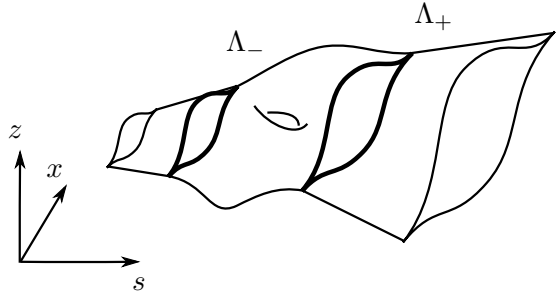
\begin{figure}[!ht]
\labellist
\small
\pinlabel $z$ at 6 54
\pinlabel $x$ at 20 40
\pinlabel $s$ at 50 0
\pinlabel $\Lambda_-$ at  88 85
\pinlabel $\Lambda_+$ at 155 95
\endlabellist
\includegraphics[width=3in]{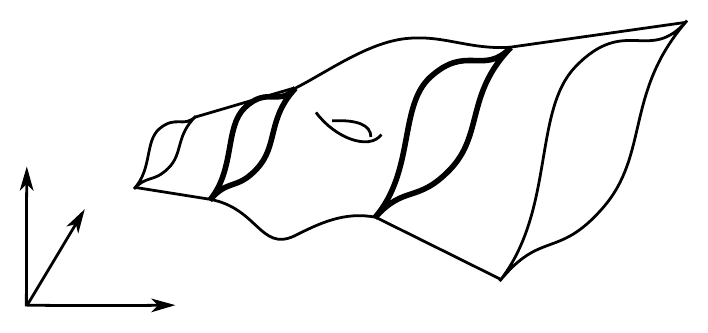}
\caption{The front projection of a conical Legendrian cobordism.}
\label{fg:conLeg}
\end{figure}
\begin{definition}
Two conical Legendrian cobordisms from $\Lambda_-$ to $\Lambda_+$ are {\bf conical Legendrian isotopic} if they are isotopic through conical Legendrian cobordisms from $\Lambda_-$ to $\Lambda_+$.
\end{definition}

The above discussion is summarized in the following.

\begin{proposition} \label{prop:Conical} For any Legendrians $\Lambda_+, \Lambda_- \subset J^1M$, the above correspondence
\[
(L,\rho) \, \longrightarrow \, \Sigma = \widetilde{\Phi}(\Sigma^{\mathit{Symp}})
\]
gives a  bijection from the set of good Lagrangian cobordisms from $\Lambda_-$ to $\Lambda_+$ 
to the set of conical Legendrian cobordisms from $\Lambda_-$ to $\Lambda_+$.  Moreover, under this bijection the equivalence relation of good Lagrangian isotopy corresponds to conical Legendrian isotopy.
\end{proposition}

\subsection{Concatenation of cobordisms}

\label{sec:Concatenate}

Let $(L,\rho)$ be a good Lagrangian cobordism in $\mathit{Symp}(J^1M)$.
Consider the map of translation by $t_0$ in the $t-$direction in $\mathit{Symp}(J^1M)$ 
$$T_{t_0}: (t, x, y, z)\to (t+t_0, x, y , z)$$
and denote the composition of $L$ with $T_{t_0}$ by $t_0 \odot L$.
Note that $(t_0 \odot L, e^{t_0} \rho)$ is also a good Lagrangian cobordism since it has $e^{t_0} \rho$ as a primitive when $\rho: L \rightarrow \R$ is a primitive for $L$.

\begin{definition}
Let    $(L_1,\rho_1): \Lambda_0 \rightarrow \Lambda_1$ and $(L_2,\rho_2): \Lambda_1 \rightarrow \Lambda_2$ be a pair of good Lagrangian cobordisms, and let $\sigma >0$ be large enough so that there exists $C \in \R$ such that $L_1$ (resp. $\sigma \odot L_2$) agrees with $\R \times \Lambda_1$ in $(C-1, +\infty)\times J^1M$  (resp. in $(-\infty, C+1) \times J^1M$).
We form their {\bf concatenation}, $(L_2,\rho_2) \circ_\sigma (L_1,\rho_1) = (L_2 \circ_\sigma L_1, \rho)$, as follows:  
\begin{itemize}
\item  Remove $[C+1,+\infty) \times \Lambda_1$ from $L_1$, remove $(-\infty, C-1] \times \Lambda_1$ from $\sigma \odot L_2$, and glue the remaining portions of $L_1$ and $\sigma \odot L_2$ together along $(C-1, C+1) \times \Lambda_1$ to get
\[
L_2 \circ_{\sigma} L_1
= \bigg( L_1 \cap \Big((-\infty, C+1)\times J^1M\Big) \bigg) \cup  \bigg( (\sigma\odot L_2) \cap \Big((C-1, +\infty)\times J^1M\Big) \bigg).
\]
\item To construct the primitive $\rho:L_2 \circ_{\sigma} L_1 \rightarrow \R$, let $A_1^+ \in \R$ be the constant value of $\rho_1$ when $t > C-1$ (assuming $\Lambda_1 \neq \emptyset$), and let $\widetilde{\rho}_2 = e^{\sigma} \rho_2+K$ be a primitive on $\sigma \odot L_2$ with $K \in \R$ chosen so that $\widetilde{\rho}_2$ agrees with $A^+_1$ when $t < C+1$, if $\Lambda_1 \neq \emptyset$; if $\Lambda_1 = \emptyset$, then put $K=0$. The primitive $\rho$ results from piecing $\rho_1$ and $\widetilde{\rho}_2$ together.   
\end{itemize}
\end{definition}

Taking Legendrian lifts and applying the contactomorphism (\ref{eq:contactomorphism}) gives a corresponding concatenation operation for conical Legendrian cobordisms in $J^1(\R_{>0} \times M)$.  To describe it independently, note that the $\R$-action on $\symp(J^1M)$ leads to a multiplicative $\R_{>0}$-action on $J^1(\R_{>0} \times M)$ given by
\[
\tau\odot((s,x), (r,y), z) = ((\tau s, x), (r, \tau y), \tau z),  \quad \tau \in \R_{>0}.
\]
\begin{definition}
Let $\Sigma_1, \Sigma_2 \subset J^1(\R_{>0} \times M)$ be a pair of conical Legendrian cobordisms   with $\Sigma_i: \Lambda_{i-1} \rightarrow \Lambda_i$, $i=1,2$. 
Let $C>1$ be such that $\Sigma_1$ agrees with $j^1(s\cdot \Lambda_1- A_1^+)$ when $ s \geq C-1$ (or so that $\Sigma_1$ is empty in this region if $\Lambda_1 = \emptyset$).  Choose $\tau>0$ large enough so that 
$\tau \odot \Sigma_2$ has the form $j^1(s\cdot \Lambda_1-\tau A_2^-)$ for all $s \leq C+1$ (where $\Sigma_2$ has the from $j^1(s\cdot \Lambda_1 - A_2^-)$ for $s$ near $0$).  Then, 
\[
\Sigma_2 \circ_{\tau} \Sigma_1 = \bigg( \Sigma_1 \cap J^1\Big((0, C+1)\times M\Big) \bigg) \cup  \bigg( j^1(\tau\odot \Sigma_2 +K)\cap J^1\Big((C-1, +\infty)\times M\Big) \bigg)
\]
where the constant $K= \tau A_2^-- A_1^+$ is chosen so that on $J^1((C-1,C+1)\times M)$ both pieces agree with $j^1(s \cdot \Lambda_1 -A_1^+)$ (or defined to be $K=0$, if $\Lambda_1= \emptyset$).  
\end{definition}

\begin{figure}[!ht]
\labellist

\pinlabel $\Sigma_2$ at 30 74
\pinlabel $\Sigma_1$ at 113 74
\pinlabel $\circ$ at 72 72
\pinlabel $\Sigma_1$ at 233 74
\pinlabel $2$ at 314 60
\pinlabel $=$ at 172 72

\endlabellist
\includegraphics[width=4.5in]{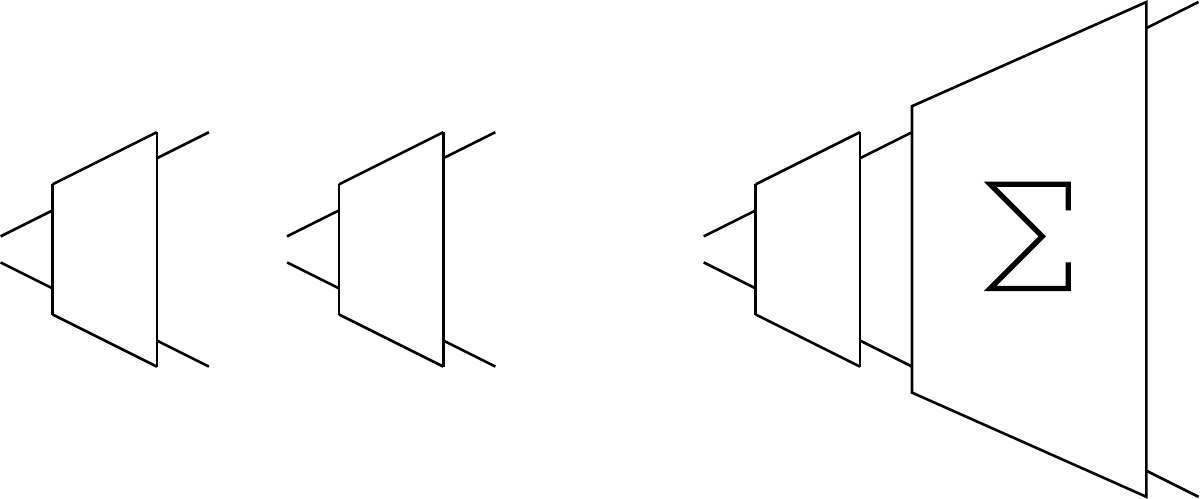}
\caption{A schematic depiction of concatenation for conical Legendrian cobordisms.}
\label{fig:conic}
\end{figure}
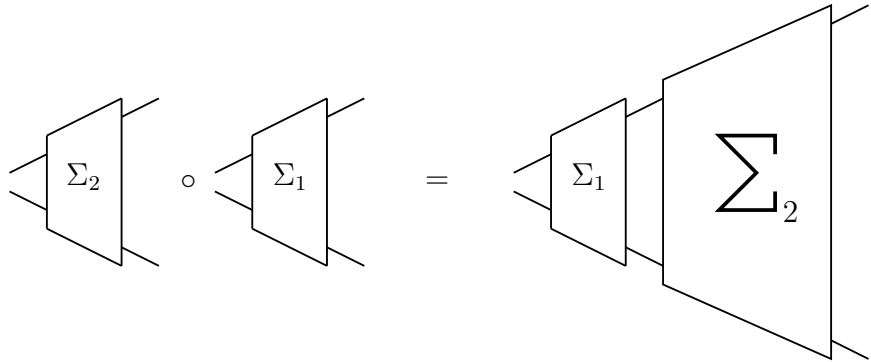

Note that $j^1(\tau\odot \Sigma_2+K)$ is conical Legendrian isotopic to $\Sigma_2$. In addition,
while the construction of $L_2 \circ_\sigma L_1$ and $\Sigma_2\circ_\tau \Sigma_1$ depends on the choices of $\sigma$ and $\tau$, the conical Legendrian isotopy class of $\Sigma_2 \circ_\tau \Sigma_1$ (and so also the good Lagrangian isotopy  class of $L_2 \circ_\sigma L_1$)  depends only on $\Sigma_1$ and $\Sigma_2$.   
Thus we have the following proposition.

\begin{proposition} \label{prop:wd}
The concatenation of conical Legendrian isotopy classes (good Lagrangian isotopy classes) of conical Legendrian cobordisms (good Lagrangian cobordisms) is well-defined and associative.

\end{proposition}
When considering concatenations up to conical Legendrian isotopy (or good Lagrangian isotopy) we omit the $\tau$ or $\sigma$ and simply write $\Sigma_2 \circ \Sigma_1$ or $L_2 \circ L_1$.

\begin{remark} There is a natural category whose objects are Legendrians and morphisms are conical Legendrian cobordisms up to conical Legendrian isotopy. However, to form the domain of the functor of interest, we will need to add some additional data to objects later.  See Section \ref{sec:functor} below.  
\end{remark}

We close the section by relating conical Legendrian isotopy to compactly supported Legendrian isotopy using concatenation.
Note that a conical Legendrian isotopy is not necessary compactly supported.
The conical ends may change by vertical shifts during the isotopy. 
However, it can be related to compactly supported Legendrian isotopy in the following way.

\begin{lemma}\label{lem:diffconandcom}
Suppose that $\Sigma_i$, $i=1,2$, are two conical Legendrian cobordisms from $\Lambda_-$ to $\Lambda_+$ that are conical Legendrian isotopic. One can modify $\Sigma_1$ to $\Sigma_1'$ through doing
\begin{enumerate}
\item a global shift in the vertical direction; and 
\item concatenation with a conical Legendrian cobordism  of the form $j^1(s\cdot \Lambda_+-h(s))$ with parameter $\tau=1$, as introduced earlier in the section,
\end{enumerate}
such that $\Sigma_1'$ and $\Sigma_2$ are compactly supported Legendrian isotopic.
\end{lemma}
\begin{proof}
After shifting $\Sigma_1$ globally in the vertical direction, one can assume that $\Sigma_i$ agree on the negative end.
However, they may have different positive ends.
Suppose that the positive ends of $\Sigma_i$ are $j^1(s\cdot \Lambda_+-A_i)$ when $s \geq S_i$ for $i=1,2$  and $A_1\neq A_2$.
We build a conical Legendrian cobordism $\Sigma_0$ from $\Lambda_+$ to $\Lambda_+$ to switch the positive end of $\Sigma_1$ to the positive end of $\Sigma_2$.
In particular $\Sigma_0= j^1(s\cdot \Lambda_+-h(s))$ , where $h(s)$ is a smooth    function on $\R$ such that $h(s)=A_1$ when $s<T_1$ for some $T_1>S_1+2$ 
    and $h(s)=A_2$ when $s>T_2$ for some $T_2>T_1$.
Let $\Sigma_1'$ be the concatenation of $\Sigma_0$ and $\Sigma_1$ with parameter $\tau=1$ , i.e., $\Sigma_1'=\Sigma_0\circ_1\Sigma_1$.
Note that $\Sigma_0$ is conical Legendrian isotopic to the trivial conical cobordism $j^1(s\cdot \Lambda_+-A_1)$ since $h(s)$ is homotopic to the constant function $A_1$.
It follows that $\Sigma_1'$ is conical Legendrian isotopic to $\Sigma_1$, and thus is conical Legendrian isotopic to $\Sigma_2$.
Observe that  $\Sigma_1'$ agrees with $\Sigma_2$ on both the positive and the negative ends.
We claim that $\Sigma_1'$ and $\Sigma_2$ are compactly supported Legendrian isotopic. 
Let $\Sigma^t$, $t\in[0,1]$ be a smooth $1$-parameter family of 
conical Legendrian cobordisms from $\Lambda_-$ to $\Lambda_+$ that is a conical Legendrian isotopy from $\Sigma_1'$ to $\Sigma_2$.
One can assume that there is  $S_+\in \R_{>0}$ such that $\Sigma^t$ are conical over $\Lambda_+$ when $s>S_+$.
If the positive ends of $\Sigma^t$ do not always agree with the positive end of $\Sigma_2$, using a modification of the above construction, one can construct a smooth family of cobordisms $\Sigma_0^t$ from $\Lambda_+$ to $\Lambda_+$  to switch the positive end of $\Sigma^t$ to the positive end of $\Sigma_2$.
Concatenating $\Sigma^t$ with $\Sigma_0^t$ gives a conical Legendrian isotopy from $\Sigma_1'$ to $\Sigma_2$ whose positive end do not change through the isotopy.
Similarly, one can precompose with an appropriate family of cobordisms to arrange the negative ends to also remain constant so that the result is a compactly supported Legendrian isotopy from $\Sigma_1'$ to $\Sigma_2$.
\end{proof}

\section{Review of Legendrian contact homology} \label{sec:LCH}

  In this section, we review the Legendrian contact homology DGA and the alternate formulation of its differential via gradient flow trees.  Continuation maps associated to Legendrian isotopies as in \cite{EESR2n+1} are recalled in some detail, and in Proposition \ref{prop:inv2} we observe a refined invariance statement that arises under certain assumptions on the action of Reeb chords.      
	In \ref{sec:gene}, we discuss a generalization for non-closed Legendrians with Morse ends following \cite{EK}.  

\subsection{Review of Legendrian contact homology}

Let $(\Lambda, \nu)$ be a Legendrian in a $1$-jet space, $J^1E$, equipped with a choice of $\Z/n$-graded Maslov potential.  Assuming $\Lambda$ is suitable generic, the important work of Etnyre, Ekholm, and Sullivan in \cite{EESR2n+1, EES} associates a $\Z/n$-graded DGA to $(\Lambda,\nu)$ which we will typically denote as $(\mathcal{A}(\Lambda),\partial)$ (with the Maslov potential suppressed from notation).  We briefly review the construction.

\subsubsection{The algebra, $\mathcal{A}(\Lambda)$}
 Recall that in $J^1E$ the Reeb field is $\partial_z$, and a {\bf Reeb  chord} of  $\Lambda$ is a vertical (in $z$-direction) line segment with both ends on the Legendrian.  Reeb chords are in bijection with the double points of the Lagrangian projection, $\pi_{xy}(\Lambda)\subset T^*E$.  
  When defining $\mathcal{A}(\Lambda)$ it should be assumed that all Reeb chords of $\Lambda$ are non-degenerate with endpoints away from the singular set $\Lambda_{\mathit{sing}}$.  This can be arranged by a small Legendrian isotopy.  Then, using $\pi_x:J^1E \rightarrow E$ for the base projection, for any Reeb chord, $c$, we can find  a neighborhood $U \subset E$ of $\pi_x(c)$  and a pair of locally defined functions $f_i,f_j:U \rightarrow \R$, $f_i>f_j$, such that $\Lambda$ agrees with $j^1f_i$ (resp. $j^1f_j$) in a neighborhood of the upper (resp. lower) endpoint of $c$ (with respect to the $z$-direction). The {\bf local difference function} $F_{i,j}:=f_i-f_j$ has a 
  non-degenerate critical point at $\pi_x(c)$, and Reeb chords are in bijection with the critical points of such locally defined difference functions for $\Lambda$.  

The algebra, $\mathcal{A}(\Lambda)$, is the free non-commutative, unital algebra over $\Z/2$ 
 generated by Reeb chords of $\Lambda$.  As a vector space over $\Z/2$, it has a basis consisting of words in the Reeb chords of $\Lambda$.

\subsubsection{Gradings of Reeb chords}\label{sec:grading}

A $\Z$-valued grading is assigned to each Reeb chord $c$ of $\Lambda$ depending on a choice of {\bf capping path}, $\gamma_c$, for $c$ which is a   path on $\Lambda$ from $c_u$ to $c_l$ where $c_u$ and $c_l$ are the upper and lower endpoints of $c$ respectively.   (See \cite[Section 2.3.3]{EENS} about the way to define  capping paths for mixed Reeb chords, i.e.  Reeb chords whose endpoints belong to two different components of $\Lambda$). 
Define  the grading  by 
\begin{equation} \label{eq:CZ}
|c|= CZ(\gamma_c)-1,
\end{equation}
where $CZ(\gamma_c) \in \Z$ is the Conley-Zehnder index discussed in \cite{EESR2n+1}.

\begin{remark}
Note that the grading of Reeb chords depends on the choice of capping paths. 
For pure Reeb chords, i.e. Reeb chords starting and ending at the same component of $\Lambda$, the grading is well defined mod the Maslov number $m(\Lambda)$.  
\end{remark}

When $\Lambda$ is equipped with a $\Z/n$-valued Maslov potential, $\nu$, a well-defined $\Z/n$-grading of Reeb chords arises as 
\begin{equation} \label{eq:Maslov}
|c|= \nu(c_u)-\nu(c_l) + \mathit{ind}(F_{i,j})-1 \in \Z/n,
\end{equation}
where $c_u$ and $c_l$ are the upper and lower endpoints of $c$ and $\mathit{ind}(F_{i,j})$ is the Morse index of the local difference function with critical point at $\pi_x(c)$.

\begin{remark}  \label{rem:twograding}
On each component of $\Lambda$ any two Maslov potentials differ by a constant.
Therefore, the grading of a pure Reeb chord does not depend on the choice of $\nu$.  Moreover, it agrees with the reduction modulo $n$ of the integer grading from (\ref{eq:CZ}).  The grading of mixed Reeb chords depends on the choice of Maslov potential.  However, for a given Maslov potential it is always possible to choose capping paths for mixed Reeb chords so that the two gradings from (\ref{eq:CZ}) and (\ref{eq:Maslov}) agree modulo $n$.  In the following, we assume this is the case and do not distinguish notationally between the $\Z$-valued and $\Z/n$-valued gradings of Reeb chords.  
\end{remark}

\subsubsection{The differential}\label{sec:diff}

The differential $\dd$ is defined by counting certain rigid holomorphic disks.  Following \cite{EES}, 
let $J$ be an almost complex structure on $T^*E$ that is compatible with the standard symplectic structure, $\omega = \sum_{i} dx_i \wedge dy_i$, and {\bf adapted}
to $\Lambda$, i.e
\begin{itemize}
\item in a neighborhood of each double point of $\pi_{xy}(\Lambda)$ there exists a choice of local coordinates where $J$ agrees with the standard complex structure and the two sheets of $\pi_{xy}(\Lambda)$ are real analytic,\footnote{Note that \cite{EES} uses the additional terminology ``$L$ is admissible'' with respect to $J$ for the requirement about real analyticity of $L$.}
\item the triple $(T^*E, \omega, J)$ has finite geometry at infinity 
as explained in \cite[Section 2.1]{EES}. 
\end{itemize}

For any Reeb chords $a, b_1, \dots, b_m$    of $\Lambda$, we consider a moduli space, $\M^A_J(a; b_1, \dots, b_m)$,
of boundary punctured $J$--holomorphic disks, $u$, up to conformal reparametrization, with boundary mapped to $\pi_{xy}(\Lambda)$ and having, in counter-clockwise order, a  positive puncture at $a$ and negative punctures at $b_1, \ldots, b_m$.  Moreover, the restriction of $u$ to the boundary together with the union of capping paths $(-\gamma_a)\cup \gamma_{b_1}\cup \ldots \cup \gamma_{b_m}$ represents the homology class $A \in H_1(\Lambda)$.    See \cite{EESR2n+1, EES} for a detailed definition of $\M^A_J(a;b_1, \ldots, b_m)$.

The formal dimension of $\M^A_J(a; b_1, \dots, b_m)$ is
\begin{equation}\label{eq:dim}
\dim \M^A_J(a; b_1, \dots, b_m)=|a|-|b_1|- \cdots -|b_m|-1+ \mu(A),
\end{equation}
where $\mu(A)$ is the Maslov index as introduced in Section \ref{sec:leg}.  
In \cite{EES}, it is shown that a generic choice of $J$ that is adapted to $\Lambda$ will be {\bf regular}, i.e. all the moduli spaces of $J$--holomorphic disks of formal dimension $\leq 1$ 
are transversely cut out. 
When $J$ is regular and $\dim \M^A_J(a; b_1, \dots, b_m)=0$, a holomorphic disk $u\in \M^A_J(a;b_1, \dots, b_m)$ is called {\bf rigid}.

\begin{remark}\label{defninz}
  The formula for the formal dimension of the moduli space $ \M^A_J(a; b_1, \dots, b_m)$ uses the $\Z$-valued grading of Reeb chords from (\ref{eq:CZ}).  
   Note that the value is independent of the choice of capping paths; changes in $|a|$ and $|b_i|$ caused by the change of capping paths are cancelled by the change to the $\mu(A)$ term. 
\end{remark}

Given a choice of $J$ that is regular and adapted to $\Lambda$, a differential on $\mathcal{A}(\Lambda)$ is defined on generators by summing over rigid holomorphic disks 
$$\displaystyle{\dd(a)= \sum_{\dim \M^A_J(a; b_1, \dots, b_m)=0} |\M^A_J(a; b_1, \dots, b_m)| \ b_1 \cdots b_m}$$
and is extended to $\alg(\Lambda)$ to satisfy the Leibniz Rule.
\begin{proposition}[\cite{EESR2n+1,EES}] \label{prop:EES}
The differential $\partial: \mathcal{A}(\Lambda) \rightarrow \mathcal{A}(\Lambda)$ satisfies $\partial^2 =0$ and has degree $-1$ with respect to the $\Z/n$-grading on $\mathcal{A}(\Lambda)$ arising from any choice of Maslov potential, $\nu$, on $\Lambda$.  Moreover, the stable tame isomorphism type of the $\Z/n$-graded DGA $(\mathcal{A}(\Lambda), \partial)$ is a Legendrian isotopy invariant of $(\Lambda, \nu)$.  
\end{proposition}

The DGA $(\mathcal{A}(\Lambda), \partial)$ is called the {\bf Legendrian contact homology DGA} of $\Lambda$.  We occasionally denote it as  $(\mathcal{A}(\Lambda, J), \partial)$ when we want to emphasize the choice of $J$ involved in the differential.

\subsubsection{The action filtration}
Each Reeb chord $c$ has an associated {\bf action} $\fraka(c)$ defined by the length of $c$, i.e. $\fraka(c) = z_u-z_l>0$ where $z_u$ and $z_l$ are the $z$-coordinates of the upper and lower endpoints of $c$.  
By the Stokes Theorem, for a holomorphic disk $u\in \M^A_J(a; b_1, \dots, b_m)$ the energy $E(u) = \int_{D^2} u^*\omega$,   which agrees with the area
 of the image of $u$,  satisfies $E(u)= \fraka(a)- \fraka(b_1)- \cdots -\fraka(b_m)$
 so that 
\begin{equation} \label{eq:action}
\fraka(a) > \sum_i \fraka(b_i)
\end{equation}
whenever  $u$ is non-constant.    
It follows that $(\mathcal{A}(\Lambda), \dd)$ is in fact a {\it triangular} DGA (as in Section \ref{sec:triDGA}) with respect to any ordering of Reeb chords with non-decreasing action.

\subsubsection{Continuation maps} \label{sec:continuation}
The proof of Legendrian invariance (see Section 2.5 of \cite{EESR2n+1} for the case of Legendrians in $J^1\R^n$,  and Section 2.4 of \cite{EES} for the extension to general 1-jet spaces following the method of Section 4.3 of \cite{EESorientation}) associates a stable tame isomorphism to a suitable Legendrian isotopy $\Phi = (\Lambda_t, J_t)$ equipped with a homotopy between regular almost complex structures for $\Lambda_0$ and $\Lambda_1$.  Since we will need to use some particulars of this stable tame isomorphism later, we review the construction.  

In loc. cit., it is shown 
that when $(\Lambda_0, J_0)$ and $(\Lambda_1,J_1)$ are regular so that LCH DGAs are defined, a Legendrian isotopy between $\Lambda_0$ and $\Lambda_1$ can be approximated by an ``admissible'' Legendrian isotopy $\Lambda_t$ and equipped with a family of almost complex structures $J_t$ with the property that there is a sequence of values
\[
0=t_0<t_1< \ldots< t_N = 1
\]
such that the LCH DGA of each $(\Lambda_{t_i}, J_{t_i})$ is defined, denote it $(\mathcal{A}_i,\partial_i)$, and for $0\leq i < N$ the DGAs $(\mathcal{A}_i, \partial_i)$ and $(\mathcal{A}_{i+1}, \partial_{i+1})$ are related by a stable tame isomorphism arising from either
\begin{itemize}
\item[(A)] a $(-1)$-disk, or
\item[(B)] the birth/death of a pair of Reeb chords.
\end{itemize}

In the case (A) of a $(-1)$-disk, for any $t \in[t_i, t_{i+1}]$ the Reeb chords of $\Lambda_t$ are identified with those of $\Lambda_{t_i}$ in a canonical way.  Moreover, for some $t_{i}< t < t_{i+1}$ there exists a disk $u \in \M^A_J(a; b_1, \dots, b_m)$ of formal dimension $-1$, and the elementary isomorphism $\psi:\mathcal{A}_i \rightarrow \mathcal{A}_{i+1}$ defined on Reeb chords by
\begin{equation}  \label{eq:negative1}
\psi(x) = \left\{\begin{array}{cr} a+ b_1\cdots b_m, & x =a, \\  x,  & \mbox{else}, \end{array}\right.
\end{equation}
is a DGA isomorphism.  

For considering case (B), suppose the pair of canceling Reeb chords, $a$ and $b$, exists at $t_{i+1}$ but not at $t_i$.  The remaining Reeb chords of $\Lambda_{i+1}$ are canonically identified with the Reeb chords of $\Lambda_i$.
Moreover, \cite{EESR2n+1, EESorientation} (see Lemmas 2.14 and 2.15 of \cite{EESR2n+1} or Lemma 4.27 of \cite{EESorientation}) show that $\partial_{i+1}a = b+ v$ where $v$ belongs to the sub-algebra generated by Reeb chords with smaller action than $b$, and the projection map 
\[
p:\mathcal{A}_{i+1} \rightarrow \mathcal{A}_i,  \quad a \mapsto 0, \, b \mapsto v, \quad p|_{\mathcal{A}_i} = \mathit{id}
\] 
induces a DGA isomorphism  $\mathcal{A}_{i+1}/I \cong \mathcal{A}_i$ where $I = \mathcal{I}(a, \partial a)$.  As in Proposition \ref{prop:cancel}, there is then a unique stable tame isomorphism
\[
\psi:  \mathcal{A}_i *S(e, f) \rightarrow \mathcal{A}_{i+1}
\]
satisfying 
\begin{equation}  \label{eq:Bmap}
\psi(e) = a, \quad \psi(f) = b +v, \quad \mbox{and} \quad \psi \circ p - \mathit{id}_{\mathcal{A}_{i+1}} = \partial_{i+1} \circ H + H \circ \partial_{i+1}
\end{equation}
 where $H: \mathcal{A}_{i+1} \rightarrow \mathcal{A}_{i+1}$ is the $(\psi \circ p, \mathit{id}_{\mathcal{A}_{i+1}})$-derivation with $H(b) =a$ and $H(x) =0$ for any other Reeb chord.  

Given such an ``admissible'' Legendrian isotopy $(\Lambda_t,J_t)$, we can define the associated {\bf continuation map} to be the stable tame isomorphism
\[
\varphi: \mathcal{A}_0*S_0 \rightarrow \mathcal{A}_1*S_1
\]
obtained by ``composing'' (in the manner specified in Remark \ref{rem:stable}) the above stable tame isomorphisms.

\begin{remark} \begin{enumerate}
\item In \cite{Kalman}, Kalman showed that for an isotopy of $1$-dimensional Legendrians $\Lambda_t \subset J^1\R$ the DGA homotopy class of the continuation map only depends on the homotopy class of $\Lambda_t$ (viewed as a path in the space of Legendrian submanifolds).  In the higher dimensional case, this result does not appear to have yet been established in the literature. 
\item In the $1$-dimensional case, the notion of ``admissible isotopy'' must be expanded from that of the dimension $\geq 2$ case from \cite{EESR2n+1,EES} to also allow triple point moves.  However, the description of continuation maps given above remains valid.  Following \cite[Section 6]{ENS} (see also \cite{Che}), there are two types of triple point moves referred to in \cite{ENS} as Moves I and II.  In Move I, there is a (constant) $(-1)$-disk with 2 negative punctures whose image is the triple point, and the continuation map is as in (\ref{eq:negative1}).  In Move II the analogous constant disk has formal dimension $0$, and the DGA is unchanged by the 
move.    
\end{enumerate}
\end{remark}

We will make use of the following mild strengthening of the invariance result from Proposition \ref{prop:EES}.
\begin{proposition} \label{prop:inv2} Let $(\Lambda_t,J_t)$ be an admissible Legendrian isotopy, and suppose  $0<B_1< B_2$ are such that
\begin{itemize}
\item[(i)] for all $0\leq t \leq 1$, none of the Reeb chords of $\Lambda_t$ have action belonging to the interval $[B_1,B_2]$; and
\item[(ii)] all birth/death pairs of Reeb chords have action belonging to $(0, B_1)$.   
\end{itemize}
Then, the associated continuation map $\varphi: \mathcal{A}_0*S_0 \rightarrow \mathcal{A}_1*S_1$ restricts to a DGA isomorphism $\varphi: \mathcal{A}^{< B_1}_0*S_0 \rightarrow \mathcal{A}^{< B_1}_1*S_1$ where $\mathcal{A}^{< B_1}_i$ is the sub-algebra generated by Reeb chords $c$ with $\fraka(c) < B_1$.
\end{proposition}

\begin{proof}
As a consequence of the composition operation from Remark \ref{rem:stable} it suffices to verify the statement in the case $\varphi$ arises from either (A) a (-1)-disk or (B) a birth/death of Reeb chords.  In case (A), note that the energy estimate (\ref{eq:action}) together with (\ref{eq:negative1}) shows that both $\varphi$ and $\varphi^{-1}$ preserve the sub-algebras $\mathcal{A}^{< B_1}_i$.  In case (B), the continuation map has the form $\varphi: \mathcal{A}_0* S(e,f) \rightarrow \mathcal{A}_1$.  Since the pair of cancelling Reeb chords have action less than $B_1$, we can use notation $\{x_1, \ldots, x_n\}$ and $\{y_1, \ldots, y_n\}$ for the generating sets of $\mathcal{A}^{< B_1}_0* S(e,f)$ and $\mathcal{A}^{< B_1}_1$ respectively ordered with increasing action of Reeb chords and with the stabilization generators $e$ and $f$ in the positions of the cancelling Reeb chords $a$ and $b$ of $\mathcal{A}_{1}$.  The identities from  (\ref{eq:Bmap}) show that any generator $x_i \in \mathcal{A}^{< B_1}_0* S(e,f)$ has $\varphi(x_i) = y_i + w$ where $w \in \Z/2\langle y_1, \ldots, y_{i-1}\rangle$. [To verify this in the case when $x_i$ is a Reeb chord, use the identity $\varphi \circ p - \mathit{id}_{\mathcal{A}_{1}} = \partial_{1} \circ H + H \circ \partial_1$ from (\ref{eq:Bmap}), and an inductive argument.]  It follows that $\varphi$ and $\varphi^{-1}$ preserve the sub-algebras $\mathcal{A}^{< B_1}_0* S(e,f)$ and $\mathcal{A}^{< B_1}_1$ as required.      
\end{proof}

\subsection{Gradient flow trees}\label{sec:GFT}
We will briefly describe the gradient flow trees in this section. 
One can find more details in \cite[Section 2.2]{Ekh}.  Our definition appears slightly different from that in \cite{Ekh} since we only consider flow trees with a single positive puncture.  See \cite{RuSu2} for a comparison of the two definitions. 

A {\bf domain tree} is a (connected) oriented tree, $\Gamma$, equipped with the following additional structure:
\begin{enumerate}
\item A choice of {\bf initial vertex}, $v_0$, which is an external (i.e., 1-valent) vertex of $\Gamma$  such that all edges of $\Gamma$ are oriented away from $v_0$. 
\item An ordering of the outgoing edges at each internal (i.e., valence $\geq 2$) vertex of $\Gamma$.
\item An assignment of a length $l(e) \in (0, +\infty]$ to each edge $e$ of $\Gamma$ such that internal edges have $l(e) < +\infty$ and the initial edge (starting at $v_0$) has $l(e) = +\infty$.
\item To each edge $e$ of a domain tree we assign an interval $I_e$:
\begin{itemize}
\item The initial edge has $I_e= (-\infty, 0]$ unless $e$ is the only edge of $\Gamma$.  In the latter case, either $I_e = (-\infty, 0]$ or $(-\infty, +\infty)$. 
\item Other edges have $I_e= \left\{\begin{array}{cr} [0,l(e)] & \mbox{when $l(e) < +\infty$}, \\  \left[0,+\infty\right),  & \mbox{when $l(e) = +\infty$}.  \end{array}\right.$
\end{itemize}
\end{enumerate} 

Let $\Lambda$ be a Legendrian submanifold in $J^1E$, and suppose $g$ is a choice of metric for $E$.  Recall that away from the co-dimension $1$ subset $\Lambda_{\mathit{sing}} \subset \Lambda$, 
any point of $\Lambda$ has a neighborhood whose front projection agrees with the graph of a {\bf local defining function} $f:U\subset E \rightarrow \R$ for $\Lambda$.

\begin{definition} \label{def:GFT}
A {\bf gradient flow tree (GFT)} of $(\Lambda, g)$ is a domain tree $\Gamma$ together with a collection of maps $\gamma: I_e \rightarrow E$ together with pairs of $1$-jet lifts $\gamma^\pm: I_e \rightarrow J^1E$, $\pi_{x} \circ \gamma^{\pm} = \gamma$ such that:
\begin{enumerate}
\item On the interior of each $I_e$, $z(\gamma^+) > z(\gamma^-)$, and $\gamma'(s) = -\nabla(F_{i,j})_{\gamma(s)}$ 
with $F_{i,j}$ the local difference function $f_{i}-f_j$ where $f_i$ and $f_j$ are local defining functions for $\Lambda$ near $\gamma^+(s)$ and $\gamma^-(s)$.

\item When $I_e$ has an end at $\pm\infty$, $\lim_{s \rightarrow \pm \infty} \gamma(s) = p$ where $p$ is a critical point of $F_{i,j}$,  i.e. a Reeb chord of $\Lambda$.  
\item At each internal vertex, $v \in \Gamma$, the $1$-jet lifts of the incoming edge, $e_0$, and the outgoing edges, ordered as $e_1, \ldots, e_m$, fit together continuously by satisfying
\[
\gamma_0^+(v) = \gamma_1^+(v); \quad \gamma_i^-(v) = \gamma_{i+1}^+(v), \, 1 \leq i \leq {m-1}; \quad \mbox{and } \quad  \gamma_{m}^-(v) = \gamma_0^-(v). 
\]
\item At each external vertex $v$ not corresponding to an infinite end of $I_e$, $\gamma^+(v) = \gamma^-(v) \in \Lambda_{\mathit{sing}}$.
\end{enumerate} 
\end{definition}

The Reeb chord $a$ corresponding to the critical point that $\Gamma$ limits to at $-\infty$ along the initial edge of $\Gamma$ is called the {\bf positive puncture} of $\Gamma$, while the critical points $b_i$ that occur as limits at $+\infty$ along other external edges are {\bf negative punctures}. As Reeb chords project to double points of $\pi_{xy}(\Lambda)$, conditions (2)-(4) in Definition \ref{def:GFT} show that the projections of all $\gamma^+$ and $-\gamma^-$ (orientation reverse) patched together forms a single closed curve, $C$, on $\pi_{xy}(\Lambda) \subset T^*E$.  The punctures of $\Gamma$ can then be labeled $a, b_1, \ldots, b_n$ as they appear according to the orientation of $C$.    Moreover, adding the union of capping paths $(-\gamma_a) \cup \gamma_{b_1} \cup \cdots \cup \gamma_{b_n}$ to the $1$-jet lifts $\Gamma$ we obtain a homology class $A(\Gamma) \in H_1(\Lambda)$.    

Two GFTs, $\Gamma$ and $\Gamma'$ are considered {\bf equivalent} if there is a homeomorphism between their domain trees that preserves the additional data associated to edges and vertices, and if the edge maps $\gamma, \gamma': I_e \rightarrow E$ are related by orientation preserving reparametrization of $I_e$.  (Though, a non-identity reparametrization is possible only when $I_e = (-\infty, \infty)$, and this can only occur if $\Gamma$ consists of a single edge.)  Given $A \in H_1(\Lambda)$, let $\T_g^A(a;b_1, \dots, b_m)$  denote the moduli space of GFTs  having $A(\Gamma) = A$ with a single positive puncture at $a$ and negatives punctures at $b_1, \dots, b_m$ up to equivalence.

In \cite[Section 3]{Ekh}, a {\bf formal dimension}, denoted here as $\dim(\Gamma)$ 
is associated to $\T^A_g(a;b_1, \dots, b_m)$ and is shown to agree with $\dim(\mathcal{M}^A(a;b_1, \ldots, b_m)$, i.e.  
\begin{equation} \label{eq:Tformal}
\dim \T_g(a;b_1, \dots, b_m)=|a|-|b_1|- \cdots -|b_m|-1+\mu(A),
\end{equation}
where $\mu(A)$ is the Maslov index of $A$. 
As in Remark \ref{defninz}, $\dim \T_g(a;b_1, \dots, b_m)$ is well defined in $\Z$.
 A metric $g$ on $E$ is called  {\bf regular} for $\Lambda$ when  there are no GFTs of negative formal dimension and each $0$-dimensional GFT space is a collection of finitely many GFTs that are transversely cut out in the sense of \cite[Section 3]{Ekh}.  There it is shown that regularity can be achieved by a generic choice of $g$.
 When $g$ is regular for $\Lambda$, a GFT of formal dimension $0$ is 
called {\bf rigid}.
 Thus, a rigid GFT $\Gamma\in \T_g(a;b_1, \dots, b_m)$ satisfies 
$$|a|- |b_1|- \cdots -|b_m|+ \mu(A(\Gamma))=1.$$

The main reason we introduce the GFTs is that they can be used to count rigid holomorphic disks following the Proposition below.  In the statement, for $0< \lambda \leq 1$ let $s_\lambda: J^1E \rightarrow J^1E$, $s_\lambda(x,y,z) = (x, \lambda y, \lambda z)$ denote the fiber rescaling.  Note that $s_\lambda(\Lambda)$ is Legendrian isotopic to $\Lambda$ and Reeb chords of $s_\lambda(\Lambda)$ are in bijection with those of $\Lambda$. 

\begin{proposition}[\cite{Ekh}] \label{Legcor}
Let $\Lambda\subset J^1E$ be a generic $n$--dimensional Legendrian submanifold for $n=1$ or $2$. 
For any  metric $g$  regular to $\Lambda$, there exists $\lambda_0 >0$ such that for all $0< \lambda < \lambda_0$ after perturbing $s_\lambda(\Lambda)$ by a arbitrarily small Legendrian isotopy, there exists an almost complex structure $J$ on $T^*E$ 
so that the rigid $J$--holomorphic disks with boundary on $\widetilde{\Lambda}$ are in $1$-to-$1$ correspondence to the rigid GFTs of $(\Lambda, g)$.
As a consequence,
the Legendrian contact homology differential of $\Lambda$ can be computed by counting rigid GFTs as 
$$\displaystyle{\dd (a)=\sum_{\dim\T^A_g(a;b_1, \cdots, b_m)=0} |\T^A_g(a;b_1, \cdots, b_m)| \ b_1 \cdots b_m}.$$
\end{proposition}

\subsection{A generalization to the non-closed case}\label{sec:gene}

In \cite[Section 2]{EK} a generalization of the DGA to non-closed Legendrians with {\it Morse minimum ends} is considered.  We recall (a minor variation of) their construction.

\begin{definition}  Let $\Lambda \subset J^1M$ be  a $1$-dimensional Legendrian.  A standard {\bf Morse minimum end} modelled on $\Lambda$  is a Legendrian surface of the form 
\begin{equation}  \label{eq:MorseEnds}
j^1(q(s)\cdot \Lambda+K) \subset J^1( (s_0-\delta, s_0+\delta) \times M)
\end{equation}
 (with notation as in Section \ref{sec:construct})  for some $s_0\in \R_{>0}$ and $\delta >0$ where $q(s)$ is a positive quadratic function having its minimum at $s_0$ and $K \in \R$.  
 \end{definition}

Let $\Sigma \subset J^1((s_--\delta,s_++\delta)\times M)$ be a 
Legendrian surface 
having Morse minimum ends at $s_-$ and $s_+$.
That is,   
in $J^1((s_--\delta,s_-+\delta)\times M)$ (resp. $J^1((s_+-\delta,s_++\delta)\times M)$) $\Sigma$ agrees with a standard Morse minimum end modeled on $\Lambda_-$ (resp. $\Lambda_+$) for some $\Lambda_\pm \subset J^1M$.  
For considering the Legendrian contact homology DGA of $\Sigma$, 
we use  an almost complex structure on $T^*((s_--\delta, s_++\delta)\times M)$ that is adapted to $\Lambda$, and we add the requirement that $J$ is {\bf standard at the ends}: 
\begin{itemize}
\item  In $T^*((s_--\delta, s_-+\delta)\times M) = T^*(s_--\delta,s_-+\delta) \times T^*M$,  $J$ has the form $J = J_0 \oplus J_-$ where $J_0$ is the standard complex structure on $T^*(s_--\delta,s_-+\delta) \cong (s_--\delta, s_-+\delta) \times i\R \subset \C$ and $J_-$ is an almost complex structure on $T^*M$ adapted to $\Lambda_-$ and regular.  A similar requirement is imposed in $T^*((s_+-\delta, s_++\delta)\times M)$.  
\end{itemize}
 In addition, when considering GFTs for $\Sigma$ we say a metric $g$ on $\R_{>0}\times M$ is {\bf standard at the ends} if:
\begin{itemize}
\item In $(s_\pm-\delta, s_\pm+\delta)\times M$, $g$ has the form  $g_\R\oplus g^\pm_M$  where $g^\pm_M$ are regular for $\Lambda_\pm$. 
\end{itemize}  

The following proposition is observed in \cite[Section 3]{EK} as a consequence of \cite[Lemmas 3.4 and 3.5]{EK}.
\begin{proposition}[\cite{EK}]  \label{prop:EK}  When $\Sigma$ has Morse minimum ends and $J$ is standard at the ends, the Legendrian contact homology DGA $(\mathcal{A}(\Sigma, J),\dd)$ is well-defined, and its stable tame isomorphism type is invariant under Legendrian isotopies through Legendrians with fixed Morse minimum ends. 
  Moreover, the DGA of $\Sigma$ can be computed in the sense of Proposition \ref{Legcor} by counting rigid gradient flow trees with respect to a metric that is standard at the ends.
\end{proposition}

\begin{remark}
 The setting of \cite{EK} is very slightly different than our setting since \cite{EK} restricts to Legendrians in $\R^{2n+1}$ and uses the standard complex structure on $\C^n$.  The Lemmas 3.4 and 3.5 from \cite{EK} are easily adapted to our setting as follows.  Lemma 3.4 is the statement that holomorphic disks with positive puncture at a point with $s=s_\pm$ (resp. $s \in(s_-,s_+)$) must have their entire image in the $s=s_\pm$ slice of $T^*((s_--\delta, s_++\delta)\times M)$ (resp. within the region where $s\in [s_-,s_+]$);  Lemma 3.5 is the analogous statement for flow trees.  The proof of Lemma 3.4 is  based on having a $(J,J_0)$-holomorphic projection map  $T^*((s_\pm-\delta, s_\pm+\delta)\times M) \rightarrow T^*(s_\pm-\delta, s_\pm+\delta) \subset \mathbb{C}$, and in our setting this is arranged by the requirement that $J$ be standard at the ends.  The proof of Lemma 3.5 just uses that since the $\frac{\partial}{\partial s}$ derivative of all local difference functions vanishes at $s_\pm$, the gradient vector fields all point tangent to $T^*M$.  This follows from the form of the metric $g = g_\mathbb{R} \times g_M$ at the Morse ends.
\end{remark}

Note that when $(\Sigma_t, J_t)$ is an admissible isotopy with $\Sigma_t$ and $J_t$ independent of $t$ in the Morse ends, a continuation map $\varphi$ arises just as in Section \ref{sec:continuation}.
\begin{lemma} \label{lem:identity2} \begin{enumerate}
\item The Reeb chords at $s=s_\pm$ are in bijective correspondence with Reeb chords of $\Lambda_{\pm}$ and generate sub-DGAs $\mathcal{A}(\Lambda_\pm) \subset \mathcal{A}(\Sigma)$.
\item The continuation map $\varphi$ associated to an admissible isotopy $(\Sigma_t, J_t)$ with $\Sigma_t$ and $J_t$ independent of $t$ in the Morse ends restricts to the identity on the sub-DGAs $\mathcal{A}(\Lambda_\pm)$.
\end{enumerate}
\end{lemma} 
\begin{proof} That the differential of $\mathcal{A}(\Sigma)$ preserves Reeb chords at $s=s_\pm$ and agrees with the differential of $\mathcal{A}(\Lambda_\pm)$ follows from \cite[Lemma 3.4]{EK} and the form of $J$ at $s=s_\pm$.  In proving (2), observe that \cite[Lemma 3.4]{EK} again shows  there are no $(-1)$-disks with positive punctures $s=s_\pm$, and no birth/death Reeb chord can appear in the differential of any of the Reeb chords at $s=s_{\pm}$.  Thus, the formulas (\ref{eq:negative1})  and (\ref{eq:Bmap}) show that $\varphi$ is the identity on these subalgebras. (See also, Proposition \ref{prop:observe} (2).)
\end{proof}

\section{The immersed LCH functor I:  Construction of immersed maps}
\label{sec:functorI} 

In this section we give the definition of an immersed DGA map 
\[
M_\Sigma = \left(\mathcal{A}(\Lambda_+, g_+) \stackrel{f_\Sigma}{\rightarrow} \mathcal{A}(\Sigma) \stackrel{i}{\hookleftarrow} \mathcal{A}(\Lambda_-, g_-) \right)
\]
induced by a conical Legendrian cobordism $\Sigma:\Lambda_- \rightarrow \Lambda_+$.  Although the construction uses a choice of metric for $\Sigma$, the immersed homotopy type of $M_\Sigma$ depends only on the choices of regular metrics $g_\pm$ for the $1$-dimensional Legendrians $\Lambda_\pm$.  The various ingredients that form $M_\Sigma$ are defined using certain {\it Morse cobordisms} constructed from $\Sigma$, and these are introduced in Section \ref{sec:Morse}.  With the Morse cobordisms in hand, the definition of $M_\Sigma$ appears in Section \ref{sec:DefM} along with the statement of the main properties of the construction, with proofs deferred until Section \ref{sec:proof}.   In particular, the construction determines a  functor between an appropriately defined category of Legendrians with conical cobordisms and the category of DGAs with homotopy classes of immersed maps constructed in Section \ref{sec:DGAcat}.  This categorical formulation is presented in Section \ref{sec:functor}.  In Section \ref{sec:embed} we consider the case of embedded Lagrangian cobordisms.

\subsection{Associated Morse cobordisms} \label{sec:Morse}
Let $\Sigma \subset J^1(\R_{>0} \times M)$ be a conical Legendrian cobordism from $\Lambda_-$ to $\Lambda_+$.  In this section, we introduce related cobordisms obtained by modifying $\Sigma$ near its ends that we denote $\widetilde{\Sigma}$, and $\overline{\Sigma}_0$.  These cobordisms no longer have conical ends, and will be referred to as {\bf Morse cobordisms}.  Their construction is summarized as follows:
\begin{enumerate}
\item  $\widetilde{\Sigma}$ has a Morse minimum end at $s=s_-$ and a Morse maximum end at $s= s_M$.
\item  $\overline{\Sigma}_0$ has a Morse minimum end at $s=s_-$ and has a Morse maximum at $s=s_M$ followed by a Morse minimum end at $s=s_m$.   
\end{enumerate}
See Figure \ref{contrast}.  The Morse cobordism $\widetilde{\Sigma}$ is used to define the DGA of a conical Legendrian cobordism $\mathcal{A}(\Sigma)$ as well as the cobordism map $f_\Sigma$.  The second Morse cobordism $\overline{\Sigma}_0$ will be used in establishing properties of $\mathcal{A}(\Sigma)$ and $f_\Sigma$.

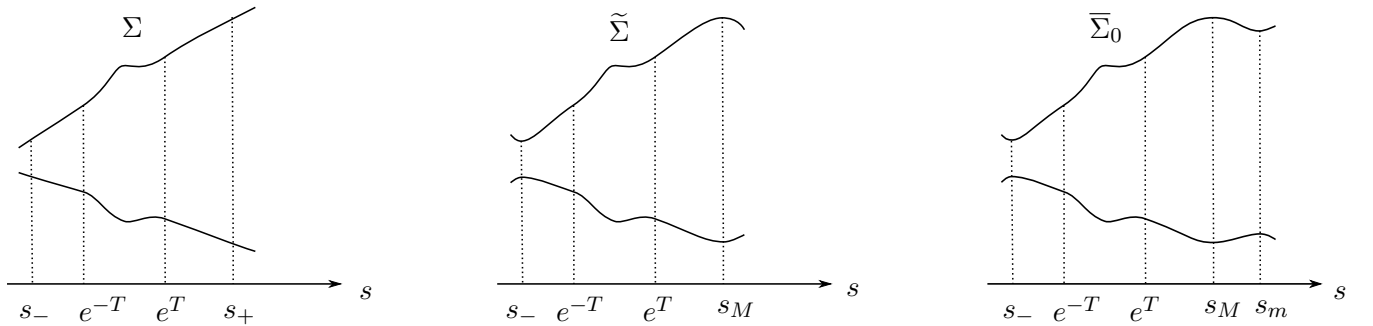
\begin{figure}[!ht]
\labellist
\small

\pinlabel $\Sigma$ at 70 140
\pinlabel $s_-$ at 20 -8
\pinlabel $e^{-T}$ at 55 -5
\pinlabel $e^T$ at 90 -5
\pinlabel $s_+$ at 125 -8
\pinlabel $s$ at 190 3 

\pinlabel $\ol\Sigma_0$ at 570 140

\pinlabel $s_-$ at 525 -8
\pinlabel $e^{-T}$ at 555 -5
\pinlabel $e^T$ at 590 -5
\pinlabel $s_M$ at 630 -5
\pinlabel $s_m$ at 655 -5
\pinlabel $s$ at 690 3

\pinlabel $s_-$ at 270 -8
\pinlabel $e^{-T}$ at 300 -5
\pinlabel $e^T$ at 340 -5
\pinlabel $s_M$ at 380 -5
\pinlabel $\wt{\Sigma}$ at 320 140
\pinlabel $s$ at 440 5
\endlabellist
\includegraphics[width=7in]{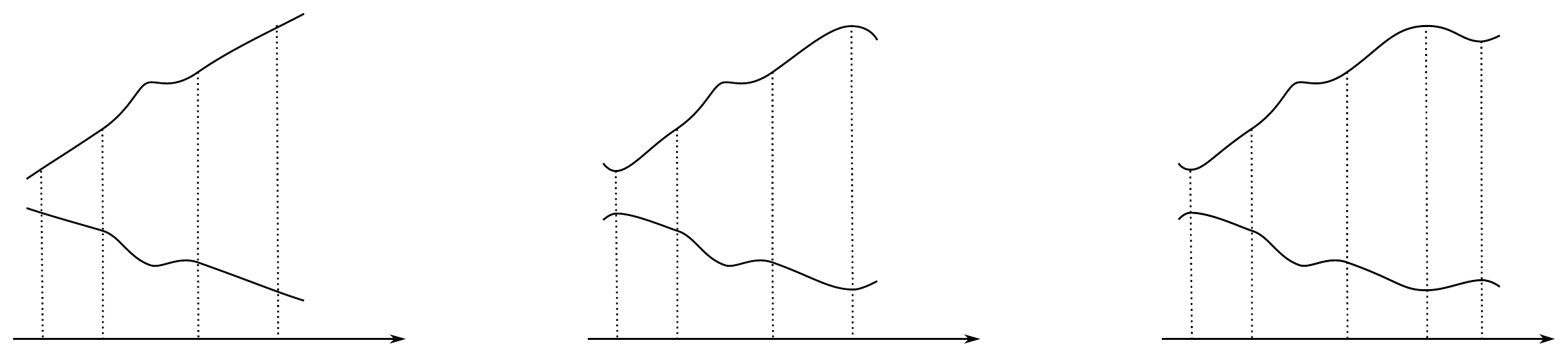}
\vspace{0.1in}

\caption{Schematic pictures of $\Sigma$, $\wt\Sigma$ and $\ol{\Sigma}_0$ in the projection to  $J^0(\R_{>0})$.}
\label{contrast}
\end{figure}

\medskip

\subsubsection{Construction of $\widetilde{\Sigma}$}\label{sec:wtSigma}
  Recall that for $T \gg 0$,  $\Sigma$ agrees with $j^1(s\cdot \Lambda_--A_-)$ on $J^1\big((0, e^{-T})\times M \big)$.  Let  $s_-, s_M$,  and  $\epsilon>0$ be positive numbers such that  
\[
s_-+3\epsilon<e^{-T} \quad \mbox{and} \quad s_M-3\epsilon >e^T,
\] 
and choose Morse functions $h_-, h_+: \R_{>0}\to \R$ satisfying
$$ h_-(s)=\begin{cases} s & \mbox{ if } s>s_-+2\epsilon\\
B_-+(s-s_-)^2 & \mbox{ if } s<s_-+\epsilon,\\
\end{cases}
\quad \mbox{and} \quad 
h_+(s)=\begin{cases} s & \mbox{ if } s<s_M-2\epsilon\\
B_M-(s-s_M)^2 & \mbox{ if } s>s_M-\epsilon,\\
\end{cases}
$$
for some number $B_-,B_M>0$ and having $h_{\pm}>0$ on $(s_--\e, s_M+\e)$ and  $(h_\pm)'>0$ on $(s_-,s_M)$.
Define the Morse cobordism $\wt\Sigma$ to be the Legendrian surface in $J^1\big((s_--\epsilon, s_M+\epsilon)\times M\big)$ that
\begin{itemize}
\item agrees with $\Sigma$ in $J^1\big([e^{-T}, e^T]\times M\big)$; 
\item with $j^1(h_-(s)\cdot \Lambda_--A_-)$ in $J^1\big((s_--\epsilon, e^{-T}]\times M\big)$; 
\item and with $j^1( h_+(s) \cdot \Lambda_+ -A_+)$ in $J^1\big([e^T, s_M+\epsilon)\times M\big)$.  
\end{itemize} See Figure \ref{contrast}.

We refer to the pair of functions $(h_-, h_+)$ that specifies $\widetilde{\Sigma}$ as the choice of {\bf end data} for the Morse cobordism.

\subsubsection{Construction of $\overline{\Sigma}_0$}\label{sec:Sigmabar}

The construction of $\ol{\Sigma}_0$ comes from a slight extension of $\wt\Sigma$.
Let $s_m$ be a number such that $s_m-3\e'>s_M$ for some small
$\e'>0$ and define a positive Morse function $\ol h_+(s): \R_{>0}\to \R$ such that

\begin{enumerate}
\item  $$\ol h_+(s) =\begin{cases} h_+(s) & \mbox{ if } s<s_M+\e'\\
B_m+(s-s_m)^2 & \mbox{ if } s>s_m-\e',
\end{cases}$$
where $B_m>0$ is some constant; and
\item $\ol h_+(s)$ has a unique local maximum at $s_M$, a unique local minimum at $s_m$, and no other critical points.
\end{enumerate}

Let $\ol\Sigma_0$ be the Legendrian surface in $J^1\big((s_--\epsilon, s_m+\epsilon')\times M\big)$ that agrees with $\wt\Sigma$ in $J^1\big((s_--\epsilon, s_M]\times M\big)$  
and with $j^1(\ol h_+(s) \cdot \Lambda_+ -A_+)$ in $J^1\big([s_M, s_m+\epsilon')\times M\big)$.  (See Figure \ref{contrast}).

\subsubsection{Reeb chords of $\widetilde{\Sigma}$}  In the following, we write $R(\Sigma), R(\Lambda_\pm)$, and $R(\widetilde{\Sigma})$ for the Reeb chords of the respective Legendrians, and assume all of these Reeb chords are provided consistent $\Z$- and 
$\Z/n$-gradings via a choice of capping paths and the Maslov potential $\nu$ on $\Sigma$ (as in Remark \ref{rem:twograding}).   Given a subset $I\subset \R_{>0}$, let $R\big(\widetilde{\Sigma}; I\big) \subset R(\widetilde{\Sigma})$ consist of those Reeb chords whose image under the projection map $J^1(\R_{>0}\times M) \to \R_{>0}$ is in $I$. 

\begin{proposition} \label{prop:chords1}
The Reeb chords of $\widetilde{\Sigma}$ satisfy
\[
R(\widetilde{\Sigma}) = R(\widetilde{\Sigma}; \{s_-\}) \cup R(\widetilde{\Sigma}; (s_-,s_M)) \cup R(\widetilde{\Sigma}; \{s_M\})
\]
and there are bijections 
\begin{equation} \label{eq:chords1}
R\big(\wt\Sigma; \{s_-\}\big) \cong R(\Lambda_-)  \quad  R\big(\wt\Sigma; (s_-, s_M)\big) \cong R(\Sigma), \quad \mbox{and} \quad  R(\widetilde{\Sigma}; \{s_M\}) \cong R(\Lambda_+)
\end{equation}
where the first two bijections preserve the grading and the third bijection has $|\widehat{a}_i| = |a_i|+1$ for corresponding Reeb chords $\widehat{a}_i \in R(\widetilde{\Sigma}; \{s_M\})$ and $ a_i \in R(\Lambda_+)$.
\end{proposition}

\begin{proof} 
In a neighborhood of the slice $s=s_-$, 
a local difference function for $\widetilde{\Sigma}$ has the form $F_{i,j}(s, x)=h_-(s)f_{i,j}(x)$, where $f_{i,j}$ is a local difference function for $\Lambda_-$.  Therefore, critical points with positive critical value, i.e. Reeb chords of $\Sigma$, near $s=s_-$ occur only at $s_-$ and are in bijection (preserving the Morse index) with the critical points of the $f_{i,j}$, i.e. Reeb chords of $\Lambda_-$.  
A similar discussion applies near the slice $s=s_M$, except that since $h_+$ has a local maximum rather than a local minimum, the index of a critical point $(s_M, x)$ for $F_{i,j}$ is $1$ larger than the index of $x$ as a critical for $f_{i,j}$.  This results in the grading shift under the bijection between Reeb chords in $R(\widetilde{\Sigma}; \{s_M\})$ and $R(\Lambda_+)$.

Since the only critical points of $h_\pm$ are at $s_-$ and $s_M$, the above are the only Reeb chords of $\widetilde{\Sigma}$ outside of $J^1([e^{-T},e^T]\times M)$.  Moreover, all Reeb chords of $\Sigma$ are located within $J^1([e^{-T},e^T]\times M)$ and $\widetilde{\Sigma}$ agrees with $\Sigma$ there.
\end{proof}

\subsubsection{GFTs and almost regular metrics}\label{sec:wtgft} Assume that $g_\pm$ are metrics on $M$ that are regular for $\Lambda_\pm$.   
\begin{definition}  We say that a metric $g$ on $\R\times M$ is {\bf cylindrical over $(g_-,g_+)$ at the ends of $\widetilde{\Sigma}$} if 
there exists $\delta>0$ so that $g$ has the form $g_\R \times g_-$ (resp. $g_\R\times g_+$) in $(0, s_-+\delta]\times M$ (resp. $[s_M-\delta, +\infty)\times M$).
\end{definition}
Given such a metric $g$, let $\mathcal{T}_g(\widetilde{\Sigma};I)$ denote those GFTs of $(\wt\Sigma,g)$ with image in $I \subset \R_{>0}$ under the projection $\R_{>0} \times M \rightarrow \R_{>0}$.  Note that we do not consider the Reeb chords at positive and negative punctures to belong to the image of a GFT (since edges only approach Reeb chords in the limit), so GFTs in $\mathcal{T}_g(\widetilde{\Sigma};(s_-,s_M))$ may have punctures at $\{s=s_-\}$ or $\{s=s_M\}$.

\begin{proposition} \label{prop:GFT1}
\begin{enumerate}
\item The GFTs of $(\widetilde{\Sigma}, g)$ can be divided into three disjoint sets
\[
\mathcal{T}_g(\widetilde{\Sigma}) = \mathcal{T}_g(\widetilde{\Sigma};\{s_-\}) \sqcup \mathcal{T}_g(\widetilde{\Sigma};(s_-,s_M)) \sqcup \mathcal{T}_g(\widetilde{\Sigma};\{s_M\}),
\]
and there are bijections 
\[
\mathcal{T}_g(\widetilde{\Sigma};\{s_-\}) \cong \mathcal{T}_{g_-}(\Lambda_-)  \quad \mbox{and} \quad \mathcal{T}_g(\widetilde{\Sigma};\{s_M\}) \cong \mathcal{T}_{g_+}(\Lambda_+). 
\]
\item Any GFT with a positive puncture at $s=s_-$ is in $\mathcal{T}_g(\widetilde{\Sigma};\{s_-\})$.  Any GFT with a negative puncture at $s=s_M$ is in $\mathcal{T}_g(\widetilde{\Sigma};\{s_M\})$.
\end{enumerate}
\end{proposition} 
\begin{proof}
Since $h_-(s)$ has a local minimum at $s=s_-$, the $\frac{\dd}{\dd s}$ component of $-\nabla F_{i,j}$ vanishes at $s_-$ and points towards $s=s_-$ for nearby values of $s$.  Moreover, the $TM$ component of $\nabla F_{i,j}$ is just $h_-(s) \nabla f_{i,j}$.  (We have used that $g= g_\R \times g_-$ near $s=s_-$.) Similarly, all $\nabla F_{i,j}$ point tangent to the slice $s= s_M$ with the $\frac{\dd}{\dd s}$ component of $-\nabla F_{i,j}$ pointing away from $s=s_M$ for nearby values of $s$.  The vanishing of $\frac{\dd}{\dd s}$ along $s=s_-$ and $s=s_M$ implies that no GFT can pass out of $\{s_-\}$, $\{s_M\}$, or  $(s_-,s_M)$ at an interior point of any edge.  This establishes the stated decomposition of $\mathcal{T}_g(\widetilde{\Sigma})$, and the bijections in (1) follow since the identity $\nabla F_{i,j}=h_-(s_-) \nabla f_{i,j}$ shows that flow lines the $-\nabla F_{i,j}$ in $s=s_-$ are just reparametrizations of flow lines of $\nabla f_{i,j}$.  (A similar statement applies at $s=s_M$).   
Finally, the above statement about directionality of the $\frac{\dd}{\dd s}$ component in neighborhoods of  $s=s_-$ and $s=s_M$ implies  (2) of the proposition.  (For instance, any GFT with a positive puncture at $s=s_-$ must have its initial edge limiting to $s=s_-$ in negative time.  Therefore, the interior of the edge, so also the rest of the tree, must be entirely in the slice $s=s_-$.)
\end{proof}

As above, let $g$ be cylindrical over $(g_-,g_+)$ at the ends of $\widetilde{\Sigma}$ where $g_\pm$ are regular metrics for $\Lambda_\pm$.
\begin{definition}
The metric $g$ is called {\bf almost regular} 
if it is regular with respect to the GFTs in $\mathcal{T}_g(\widetilde{\Sigma};\{s_-\}) \sqcup \mathcal{T}_g(\widetilde{\Sigma};(s_-,s_M))$, i.e. this subset does not contain any GFTs of negative formal dimension and all $0$-dimensional GFTs are transversally cut out.   
\end{definition}

\begin{remark} \label{rem:almostreg}
\begin{enumerate}
\item The bijections of $\mathcal{T}_g(\widetilde{\Sigma};\{s_-\})$ and $\mathcal{T}_g(\widetilde{\Sigma};\{s_M\})$ with  $\mathcal{T}_{g_-}(\Lambda_-)$ and $\mathcal{T}_{g_+}(\Lambda_+)$ show that the regularity requirement is then automatically satisfied in $\mathcal{T}_g(\widetilde{\Sigma};\{s_-\})$ (since $(\Lambda_-, g_-)$ is assumed regular).  However, regularity will generally fail in $\mathcal{T}_g(\widetilde{\Sigma};\{s_M\})$ since the shift in the degree of Reeb chords in (\ref{eq:chords1}) changes the formal dimension of trees from $\mathcal{T}_{g_+}(\Lambda_+)$.  (See Section \ref{sec:fmapproof} below.)

\item The almost regular condition can be arranged by perturbing $g$ within $(s_-,s_M)\times M$.  
\end{enumerate}
\end{remark}

\subsection{Definition of the immersed DGA map $M_\Sigma$} \label{sec:DefM}  

We are now ready to define the various pieces that form the immersed map
\[
M_\Sigma = \left(\mathcal{A}(\Lambda_+, g_+) \stackrel{f_\Sigma}{\rightarrow} \mathcal{A}(\Sigma,g) \stackrel{i}{\hookleftarrow} \mathcal{A}(\Lambda_-, g_-)\right).
\]  To this end, fix a choice of end data $(h_-,h_+)$ to form $\widetilde{\Sigma}$ and choose  
an almost regular metric $g$ for $\widetilde{\Sigma}$ with the form $g_{\R}\times g_\pm$ in the Morse ends where $g_-$ and $g_+$ are regular metrics for $\Lambda_-$ and $\Lambda_+$.  

First, we will define the DGA $\mathcal{A}(\Sigma)= \mathcal{A}(\Sigma,g)$ as in the usual definition of LCH with generating set given by the Reeb chords
\[
R\big(\widetilde{\Sigma}; [s_-,s_M)\big) \cong R(\Lambda_-) \cup R(\Sigma)
\]
and with differential defined by counting rigid GFTs of $(\widetilde{\Sigma},g)$.   For emphasis, the generators of $\mathcal{A}(\Sigma)$ correspond to Reeb chords of $\Sigma$ {\it and} Reeb chords of $\Lambda_-$.  

\begin{theorem}  \label{prop:SigmaDGA}
\begin{enumerate}
\item The DGA $(\alg(\Sigma, g),\dd)$ is well-defined, i.e. $(\alg(\Sigma, g),\dd)$ is a DGA.
\item The above bijection  $R(\Lambda_-) \cong  R\big(\wt\Sigma; \{s_-\}\big)$ induces an inclusion of  DGAs,
\[
i:\mathcal{A}(\Lambda_-, g_-) \hookrightarrow \mathcal{A}(\Sigma, g).  
\] 
\end{enumerate} 
\end{theorem}
\begin{proof}[Proof of (2).]  That the inclusion map $i$ commutes with differentials follows from the bijection $\mathcal{T}_g(\widetilde{\Sigma};\{s_-\}) \cong \mathcal{T}_{g_-}(\Lambda_-)$ from Proposition \ref{prop:GFT1}.
\end{proof}

The proof of (1) is addressed in the following Section \ref{sec:proof}.   We refer to $\mathcal{A}(\Sigma)$ as the {\bf Legendrian contact homology DGA}  
of the conical Legendrian cobordism $\Sigma$.

 Next, let $\widehat{a} \in R\big(\wt\Sigma; \{s_M\}\big)$ be a Reeb chord corresponding to $a\in \Lambda_+$, and 
 let $b_1, \dots, b_m$ be Reeb chords in $R\big(\wt\Sigma; [s_-, s_M)\big)$.
Denote by $\T_g(\widehat{a}; b_1,\dots,b_m) \subset \mathcal{T}_g(\widetilde{\Sigma};(s_-,s_M))$ the space of  GFTs for $(\wt\Sigma,g)$ with the given punctures.
\begin{definition}\label{defnf}
Given a pair $(\wt\Sigma, g)$ where $g$ is a metric  almost regular to $\wt\Sigma$, the {\bf cobordism map} $f_{\Sigma}:\alg(\Lambda_+, g_+)\to \alg(\Sigma,g)$ is defined on generators  of $\alg(\Lambda_+)$ through
\[
\displaystyle{f_{\Sigma}(a)=\sum_{\dim\T_g(\widehat{a}; b_1,\dots,b_m)=0}} |\T_g(\widehat{a}; b_1,\dots,b_m)| \ b_1\cdots b_m
\]
and extended to $\alg(\Lambda_+)$ to be an algebra homomorphism.
\end{definition}

\begin{theorem}\label{DGAmap}
The  cobordism 
map $f_{\Sigma}: (\alg(\Lambda_+, g_+), \dd_{\Lambda_+})\to (\alg(\Sigma, g), \dd_\Sigma)$  is a DGA map.
\end{theorem}
The proof will be given in Section \ref{sec:fmapproof}.

We now define the  {\bf induced immersed DGA map} associated to $\Sigma$ to be the immersed map $M_\Sigma$ resulting from combining the maps $i$ and $f_\Sigma$ associated to an almost regular metric $g$ for $\Sigma$ together with choice of end data $(h_-,h_+)$ used to form $\widetilde{\Sigma}$.    It satisfies the following invariance property.

\begin{theorem}\label{inva}  
The immersed homotopy type (as in Definition \ref{def:homotopy}) of the immersed DGA map
\[
M_\Sigma = \left(\mathcal{A}(\Lambda_+, g_+) \stackrel{f_\Sigma}{\rightarrow} \mathcal{A}(\Sigma,g) \stackrel{i}{\hookleftarrow} \mathcal{A}(\Lambda_-, g_-)\right)
\]
depends only on $(\Lambda_-,g_-)$, $(\Lambda_+, g_+)$, and the conical Legendrian isotopy type of $\Sigma$.
\end{theorem}
In particular, Theorem \ref{inva} establishes the following:
\begin{corollary}\label{cor:DGAinva} For any fixed $(g_-,g_+)$, the stable tame isomorphism type of the LCH DGA $\mathcal{A}(\Sigma, g)$ is a conical Legendrian isotopy invariant of $\Sigma$.
\end{corollary}

Moreover, induced immersed DGA maps have the following behavior under concatenation.

\begin{theorem}\label{thm:con}
Let $(\Lambda_k,g_k)$, $k=0,1,2$ be Legendrians in $J^1M$ equipped with regular metrics, and let $\Sigma_i:\Lambda_{i-1} \rightarrow \Lambda_i$, $i=1,2$ be conical Legendrian cobordisms from $\Lambda_{i-1}$ to $\Lambda_i$.  Then, we have
\[
M_{\Sigma_2 \circ \Sigma_1} \simeq M_{\Sigma_1} \circ M_{\Sigma_2}  \quad \quad \mbox{(immersed DGA homotopy).}
\]
\end{theorem}

The proofs of the above Theorems \ref{inva} and \ref{thm:con} appear in Section \ref{sec:proof}.  As a corollary of Theorem \ref{thm:con}, we get a computation of $\mathcal{A}(\Sigma_2 \circ \Sigma_1)$ up to stable tame isomorphism.

\begin{corollary}
The Legendrian contact DGA of $\Sigma_2 \circ \Sigma_1$ is stable tame isomorphic to a DGA, $\mathcal{B}$, that is the  categorical pushout (in the category of unital DGAs) of the inclusion map $i_1:\mathcal{A}(\Lambda_1) \hookrightarrow \mathcal{A}(\Sigma_2)$ (the inclusion) and $f_{\Sigma_1}: \mathcal{A}(\Lambda_1) \rightarrow \mathcal{A}(\Sigma_1)$,
\[
\xymatrix{ \mathcal{A}(\Sigma_2) \ar@{.>}[r] & \mathcal{B}  
\\ \mathcal{A}(\Lambda_1) \ar[u]^{i_1} \ar[r]^{f_{\Sigma_1}} & \mathcal{A}(\Sigma_1) \ar@{.>}[u] }  \quad \quad  \raisebox{-.8cm}{$\mathcal{B}*S \cong \mathcal{A}(\Sigma_2 \circ \Sigma_1)*S'$.}
\]
\end{corollary}
\begin{remark}  In \cite{Sivek} this result is established for $0$-dimensional Legendrians with $1$-dimensional cobordisms up to isomorphism rather than stable tame isomorphism.  
 \end{remark}

\subsection{A categorical formulation}\label{sec:functor}

The results of this section can be succinctly stated in the language of category theory.  
Given a non-negative integer $n \geq 0$, we consider two categories $\leg^n_{im}$ and $\dga^n_{im}$.  The category $\dga^n_{im}$ is the immersed DGA category constructed in Section \ref{sec:DGAcat}: the objects are triangular $\Z/n$-graded DGAs  
and the morphisms are ($\Z/n$-grading preserving) immersed DGA maps up to immersed homotopy.
The category $\leg^n_{im}$ will be called the {\bf$n$-graded immersed Legendrian  category}.  The objects are triples $(\Lambda, g, \nu)$
consisting of   a Legendrian submanifold $\Lambda \subset J^1M$ equipped with a regular metric $g$ and a $\Z/n$-valued Maslov potential $\nu$.  
A morphism from $(\Lambda_-, g_-, \nu_-)$ to $(\Lambda_+, g_+, \nu_+)$ is an equivalence class of pairs $(\Sigma, \nu)$ where $\Sigma \subset J^1(\R_{>0}\times M)$ is a conical Legendrian cobordism  from $\Lambda_-$ to $\Lambda_+$ with $\nu|_{\Lambda_\pm} = \nu_{\pm}$ and the equivalence is conical Legendrian isotopy.  Composition is concatenation as in Section \ref{sec:Concatenate}; note that Maslov potentials extend over the concatenation in a well-defined way.

\begin{remark}
The choice of Maslov potential is included with objects and morphisms in $\leg^n_{im}$ in order to specify well defined $\Z/n$-gradings as in Remark \ref{rem:twograding}. 
Note that as in Remark \ref{rem:M} all Legendrians in $\leg^n_{im}$ have $n\,|\, m(\Lambda)$ (the Maslov number of $\Lambda$). 
\end{remark}

We define the {\bf immersed LCH functor} 
\[
F^n: \leg^n_{im}\to \mathfrak{DGA}^n_{im}
\]
on objects by 
\[
F^n(\Lambda, g, \nu) = \big(\mathcal{A}(\Lambda, g), \partial\big)
\]
 with the Maslov potential $\nu$ inducing the $\Z/n$-grading on $\mathcal{A}(\Lambda, g)$ and on morphisms by 
\[
F^n(\Sigma, \nu) = [M_\Sigma] 
\]
where $[M_\Sigma]$ denotes the immersed homotopy class of the induced map $M_\Sigma$ (with $\Z/n$-gradings).

\begin{proposition}\label{prop:imfunctor}
The map $F^n: \leg^n_{im}\to \mathfrak{DGA}^n_{im}$ is a contravariant functor.
\end{proposition}

\begin{proof}
That $F^n$ is well-defined and reverses compositions follows from Theorems \ref{inva} and \ref{thm:con}.  To check that $F^n$ also preserves identity morphisms, we need to compute that $f_\Sigma = \mathit{id}_{\mathcal{A}(\Lambda)}$ when $\Sigma: (\Lambda,g) \rightarrow ( \Lambda,g)$ is the identity cobordism, $j^1(s \cdot \Lambda)$.  In this case, $\widetilde{\Sigma} = j^1(h(s) \cdot \Lambda)$ where $h:(s_--\epsilon, s_M+\epsilon) \rightarrow \R_{>0}$ has a non-degenerate min (resp. max) at $s=s_-$ (resp. $s=s_M$) and $h'(s)>0$ for $s \in (s_-, s_M)$.  Local difference functions for $\widetilde{\Sigma}$ have the form $F_{i,j}(s,x) = h(s)\cdot f_{i,j}(x)$ where $f_{i,j}$ is a local difference function for $\Lambda$, and using the product metric $g_\R \times g$ on $\R_{>0}\times M$ leads to
\[
\nabla F_{i,j}(s,x) = \left(h'(s)f_{i,j}(x), h(s)  \nabla f_{i,j}(x)\right).
\] 
A trajectory of $\nabla F_{i,j}$ defined on an interval $I$ then has the form $\gamma:I \rightarrow \R_{>0} \times M$, $\gamma(t) = \left( \alpha(t), \beta(t) \right)$ with
\[
 \beta'(t) = h(\alpha(t)) \cdot \nabla f_{i,j}(\beta(t)).
\]
Thus, $\beta$ is just a reparametrization of a trajectory of $\nabla f_{i,j}$. As a result, projecting all edges of $\Gamma$ to $M$ (and reparametrizing in an appropriate manner) produces a GFT $\Gamma_M$ for $\Lambda$.  For computing $f_\Sigma$, we are concerned with the case where $\Gamma$ has a positive puncture at $\widehat{a} \in R(\widetilde{\Sigma};\{s_M\})$\vspace{0.02in} and negative punctures $b_1, \ldots, b_n \in R(\widetilde{\Sigma};\{s_-\})$.  Now, the grading shift in the bijection $R(\widetilde{\Sigma};\{s_M\}) \cong R(\Lambda_+)$ implies that
\[
\dim \Gamma = \dim \Gamma_M +1.
\]
As $g$ is regular for $\Lambda$, we see that $\dim \Gamma \geq 0$ and the inequality is strict whenever $\Gamma_M$ is non-constant.  Thus, the only way $\dim \Gamma = 0$ is if $\Gamma$ consists of a single edge whose $M$ component is constant at a Reeb chord of $\Lambda$.  It follows that $g_\R \times g$ is almost regular for $\Sigma$ and $f_\Sigma$ is the identity as required.
\end{proof}

\subsection{The embedded case} \label{sec:embed} Many investigations involving functoriality of LCH currently in the literature focus on the case of embedded exact Lagrangian cobordisms considered up to exact Lagrangian isotopy.  In this section, we consider the embedded case and observe the relationship with the induced maps from \cite{EHK}.

Under the bijection from Proposition \ref{prop:Conical}, a conical Legendrian cobordism $\Sigma: \Lambda_- \rightarrow \Lambda_+$ in $ J^1(\R_{>0} \times M)$ corresponds to an {\it embedded} exact Lagrangian cobordism $L \subset \mathit{Symp}(J^1M)$ with primitive $\rho:L \rightarrow \R$  if and only if $\Sigma$ has no Reeb chords.  For fixed $n\geq 0$, we can consider a category of Legendrians with embedded Lagrangian cobordisms, $\mathfrak{Leg}^n$, having the same objects as $\mathfrak{Leg}^n_{im}$, but allowing equivalence classes of only those $\Sigma$ that correspond to embedded exact Lagrangians as morphisms.  Moreover, we modify the notion of equivalence so that $\Sigma_1 \sim \Sigma_2$ only when they are related by a conical Legendrian isotopy, $\Sigma_t$,  such that $\Sigma_t$ has no Reeb chords for all $1\leq t \leq 2$.  Equivalently, via  Proposition \ref{prop:Conical} the morphisms of $\mathfrak{Leg}^n$ correspond to  embedded exact Lagrangian cobordisms $(L,\rho)$ in $\mathit{Symp}(J^1M)$ considered up deformations $(L_t,\rho_t)$ where  $L_t$ is a compactly supported isotopy through exact Lagrangians with the primitive $\rho_t$ allowed to vary outside a compact set.  There is then a well defined functor $j:\mathfrak{Leg}^n \rightarrow \mathfrak{Leg}^n_{im}$ that is the identity on objects and has $j([\Sigma]) = [\Sigma]$, although $j$ may not be injective on morphism sets.  (See Question \ref{Q:strong}.)

For a morphism $[\Sigma]$ in $\mathfrak{Leg}^n$, the absence of Reeb chords of $\Sigma$ implies that $\mathcal{A}(\Sigma) = \mathcal{A}(\Lambda_-)$, so that $f_\Sigma$ contains all of the interesting information from the immersed DGA map  
\[
M_\Sigma =\left( \mathcal{A}(\Lambda_+) \stackrel{f_\Sigma}{\rightarrow} \mathcal{A}(\Lambda_-) \stackrel{\mathit{id}}{\leftarrow} \mathcal{A}(\Lambda_-) \right).
\]
Recall from Section \ref{sec:EmbeddedDGA} the category  $\mathfrak{DGA}^n$ (morphisms are just DGA maps up to DGA homotopy) and its inclusion $I:\mathfrak{DGA}^n \rightarrow \mathfrak{DGA}^n_{im}$ into the immersed DGA category.  
\begin{proposition}  The correspondence 
\[
\Sigma: (\Lambda_-,g_-) \rightarrow (\Lambda_+,g_+) \quad \leadsto \quad f_\Sigma: \mathcal{A}(\Lambda_+, g_+) \rightarrow \mathcal{A}(\Lambda_-,g_-)
\] 
gives a well-defined contravariant functor $f:\mathfrak{Leg}^n \rightarrow \mathfrak{DGA}^n$.  Moreover, there is a commutative diagram.
\begin{equation} \label{eq:FunctorDiag}
\xymatrix{ \mathfrak{Leg}^n_{im} \ar[r]^{F} & \mathfrak{DGA}_{im}^n   
\\  \mathfrak{Leg}^n \ar[u]^j \ar[r]^{f} & \mathfrak{DGA}^n \ar[u]^I }
\end{equation}
\end{proposition}
\begin{proof}
For a morphism $[\Sigma]$ in $\mathfrak{Leg}^n$, from the definitions, we have that $F \circ j([\Sigma]) = [M_\Sigma] = I([f_\Sigma])$.  Moreover, since $j$ is well-defined, we get that if $[\Sigma] =[\Sigma']$ in $\mathfrak{Leg}^n$ then $I([f_\Sigma]) = I([f_{\Sigma'}])$, and since $I$ is injective on hom-spaces (see Proposition \ref{prop:InclusionFunctor}) we conclude that $[f_\Sigma] = [f_{\Sigma'}]$ in   $\mathfrak{DGA}^n$ so that $f$ is well-defined on hom-spaces.  Moreover, the commutativity of (\ref{eq:FunctorDiag}) together with the fact that $I$ is injective on hom-spaces implies that $f$ preserves compositions since
\[
I(f(\Sigma_2 \circ \Sigma_1)) = F(j(\Sigma_2 \circ \Sigma_1)) = F\circ j(\Sigma_1) \circ F\circ j(\Sigma_2) = I(f(\Sigma_1)) \circ I(f(\Sigma_2)) = I(f(\Sigma_1) \circ f(\Sigma_2)).
\]
Finally, it has already been verified that $f_\Sigma$ is the identity map when $\Sigma$ is an identity cobordism.
\end{proof}

Note that Corollary \ref{cor:L} from the introduction follows immediately.

\begin{remark}[Comparison with \cite{EHK}] 
For Legendrian knots in $J^1\R$, the articles \cite{EHK,E2} associate an induced cobordism map 
\[
\Phi_{L,J}:\mathcal{A}(\Lambda_+, J_+) \rightarrow \mathcal{A}(\Lambda_+, J_-)
\]
to an embedded exact Lagrangian cobordisms $L \subset \mathit{Symp}(J^1\R)$ equipped with suitable choices of almost complex structure $J_\pm$ and $J$ such that  the DGA homotopy class of $\Phi_{L,J}$ depends only on the exact Lagrangian isotopy type of $L$.  The map $\Phi_{L,J}$ is defined by a count of holomorphic curves with punctures asymptotic to Reeb chords at the positive and negative ends of $\mathit{Symp}(J^1\R)$.  
In addition, \cite{EHK} considers  cobordism maps defined by counts of GFTs on an associated Morse cobordism, precisely as in the definition of our $f_\Sigma$, and observe that with a suitable choice of almost complex structure and fiber rescaling the two maps agree. See Theorem 1.6 of \cite{EHK} for more details.   In particular, all the explicit computations of cobordism maps in Section 6 of \cite{EHK}  are in terms of GFTs, and  apply to our $f_\Sigma$ as well.
A discussion of a holomorphic curve approach to constructing immersed DGA maps for immersed cobordisms appears in Section \ref{sec:SFTpers} below. 
\end{remark}

\section{The immersed LCH functor II:  Proof of statements}  \label{sec:proof}
This section contains the proofs of the Theorems \ref{prop:SigmaDGA} (1), \ref{DGAmap}, \ref{inva}, and \ref{thm:con} stated in the previous section.   The main idea is to show that the DGA of a Legendrian surface $\ol\Sigma$, obtained from perturbing the Morse cobordism $\ol\Sigma_0$ from Section \ref{sec:Morse}, is a mapping cylinder DGA for the $f_{\Sigma}$ map.  The algebraic results from Section \ref{sec:mapc} can then be applied to obtain the stated properties of the induced map $M_\Sigma$.   
After introducing $\ol\Sigma$, the proofs of Theorem \ref{prop:SigmaDGA} (1) and \ref{DGAmap} appear at the end of \ref{sec:fmapproof}.  The proofs of the remaining theorems appear at the end of \ref{sec:concate}.

\subsection{$f_{\Sigma}$ is a DGA map}  \label{sec:fmapproof}
Recall that $\Sigma: \Lambda_- \rightarrow \Lambda_+$ is a conical Legendrian cobordism, and that in defining the map $f_{\Sigma}: \mathcal{A}(\Lambda_+,g_+) \rightarrow \mathcal{A}(\Sigma,g)$ we have made a choice of end data $(h_-,h_+)$ to construct the Morse cobordism $\widetilde{\Sigma},$ as well as a choice of almost regular metric $g$ for $\widetilde{\Sigma}$ that is conical over $(g_-,g_+)$ at the ends. Let $\overline{h}_+$ be a choice of extension of $h_+$ used to define $\overline{\Sigma}_0$ as in \ref{sec:Morse}.  
We now consider the Reeb chords of $\overline{\Sigma}_0$, using notation as in Section \ref{sec:DefM}.  

\begin{lemma} \label{lem:Reeb3} 
 The Reeb chords of $\overline{\Sigma}_0$ satisfy
\[
R\big(\ol \Sigma_0\big)= R\big(\ol\Sigma_0; \{s_m\}\big) \sqcup R\big(\ol\Sigma_0; \{s_M\}\big)\sqcup R\big(\ol\Sigma_0; [s_-,s_M)\big),
\]
and there are grading preserving bijections
\[
R\big(\ol\Sigma_0; \{s_m\}\big) \cong R(\Lambda_+), \quad \quad R\big(\ol\Sigma_0; \{s_M\}\big)\cong \widehat{R(\Lambda_+)}, \quad \quad R\big(\ol\Sigma_0; [s_-,s_M)\big) \cong R\big(\widetilde{\Sigma}; [s_-,s_M)\big),
\]
where $\widehat{R(\Lambda_+)}$ denotes $R(\Lambda_+)$ with all gradings shifted up by $1$.  
\end{lemma}
\begin{proof}This follows from the construction of $\ol h_+(s)$ as in the proof of Proposition \ref{prop:chords1}.  
\end{proof}
Supposing the Reeb chords of $\Lambda_+$ are $a_1, \dots, a_n$ we use the notations $a_i$ (resp. $\widehat{a}_i$) for the corresponding Reeb chords in $R\big(\ol\Sigma_0; \{s_m\}\big)$ (resp. $R\big(\ol\Sigma_0; \{s_M\}$).

Next, we study the GFTs of $\big(\ol\Sigma_0,g\big)$.

\begin{lemma} \label{lem:T} 
The  GFT space $\T_g\big(\ol\Sigma_0\big)$ can be divided into
$$\T_g\big(\ol\Sigma_0\big)=\T_g\big(\ol\Sigma_0; [s_-,s_M)\big)
\sqcup \T_g\big(\ol\Sigma_0; \{s_M\}\big)\sqcup \T_g\big(\ol\Sigma_0; (s_M,s_m)\big)\sqcup
\T_g\big(\ol\Sigma_0; \{s_m\}\big).$$
Moreover, any GFT for $(\overline{\Sigma}_0, g)$ with the positive puncture at a Reeb chord in $R\big(\overline{\Sigma}_0;\{s_m\}\big)$ is in $\mathcal{T}_g\big(\overline{\Sigma}_0;\{s_m\}\big)$. Any  GFT with a negative puncture at a Reeb chord in $R\big(\overline{\Sigma}_0;\{s_M\}\big)$ is in $\mathcal{T}_g(\overline{\Sigma}_0;\{s_M\})$. 
\end{lemma}

\begin{proof}
This follows as in the proof of Proposition \ref{prop:GFT1}.
\end{proof}

Write 
\[
\mathcal{T}_g\big(\overline{\Sigma}_0;(s_-, s_M) \big) = \mathcal{T}_g^1\big(\overline{\Sigma}_0;(s_-, s_M) \big) \sqcup \mathcal{T}_g^2\big(\overline{\Sigma}_0;(s_-, s_M) \big),
\]
where $\mathcal{T}_g^1\big(\overline{\Sigma}_0;(s_-, s_M) \big)$(resp. $\mathcal{T}_g^2\big(\overline{\Sigma}_0;(s_-, s_M) \big)$) denotes the subset of $\mathcal{T}_g\big(\overline{\Sigma}_0;(s_-, s_M)\big)$ consisting of GFTs whose positive puncture is at a Reeb chord in $R\big(\ol\Sigma_0;(s_-, s_M)\big)$ (resp. in $R\big(\ol\Sigma_0; \{s_M\} \big)$). 
We also write 
$$\T^1_g\big(\ol\Sigma_0; [s_-, s_M)\big)=\T^1_g\big(\ol\Sigma_0; (s_-, s_M)\big)\sqcup \T_g\big(\ol\Sigma_0; \{s_-\}\big).$$
\begin{lemma}\label{BarGFT}
Moreover, there are $1-1$ correspondences as follows:
\begin{enumerate}
\item  the GFTs in $\T_g\big(\ol\Sigma_0; \{s_M\}\big)$ $\longleftrightarrow$ the  GFTs for $(\Lambda_+, g_+)$;
\item the GFTs in $\T_g\big(\ol\Sigma_0; \{s_m\}\big)$ $\longleftrightarrow$ the GFTs for $(\Lambda_+, g_+)$;
\item the rigid GFTs in $\T_g\big(\ol\Sigma_0; (s_M, s_m)\big)$ $\longleftrightarrow$ the negative gradient flow lines of $\ol h_+(s)$ that flow from $\wh a_i$ to $a_i$;
\item the rigid GFTs in $\T^1_g\big(\ol\Sigma_0; [s_-, s_M)\big)$
$\longleftrightarrow$ the  rigid GFTs counted by the differential of $\big(\alg(\Sigma,g), \partial\big)$; 
\item the rigid GFTs in $\T^2_g\big(\ol\Sigma_0; (s_-, s_M)\big)$ $\longleftrightarrow$ the  rigid GFTs counted by the $f_{\Sigma}$ map;
\item the GFTs in $T_g\big(\ol\Sigma_0; \{s_-\}\big)$ $\longleftrightarrow$ the GFTs for $(\Lambda_-, g_-)$.
\end{enumerate}
\end{lemma}
 \begin{proof} 
 Similar to the proof of Proposition \ref{prop:GFT1}, 
the construction of $\ol h_+(s)$ and the assumption on $g$ imply (1), (2), and (6). The bijection (3) follows from the computation of $f_\Sigma$ in the case of the identity cobordism as in the proof of Proposition \ref{prop:imfunctor}.
(4) and (5) follow since the GFTs of $(\overline{\Sigma}_0,g)$ and $(\widetilde{\Sigma},g)$ in $[s_-,s_M]\times M$ are in bijection as $\widetilde{\Sigma}$ and $\overline{\Sigma}_0$ are identical in this region.  
\end{proof}

Typically, the metric $g$ will not be regular for $\ol\Sigma_0$.
Indeed, according to the correspondence (1) in the Lemma \ref{BarGFT}, a rigid GFT $\Gamma\in \T_{g_+}(a; b_1, \dots, b_m)$ of $\Lambda_+$ corresponds to a GFT $\wh \Gamma\in \T_g(\wh a; \wh b_1, \dots, \wh b_m)$ in $\T_g\big(\ol \Sigma_0; \{s_M\}\big)$.
However, this correspondence does not preserve the formal dimension of trees since (using (\ref{eq:Tformal})) we have
\begin{align*}
\dim(\widehat{\Gamma}) & = |\widehat{a}| -|\widehat{b}_1|- \cdots - |\widehat{b}_m| -1 +\mu(A(\widehat{\Gamma})) \\
 & =  (|a|+1) -(|b_1|+1) - \cdots - (|b_m|+1) -1 +\mu(A(\Gamma)) \\
 & = \dim(\Gamma) +(1-m). 
\end{align*}
In particular, the slice $\{s=s_M\}$ may contain trees of negative formal dimension.   
Thus, a perturbation is required to achieve regularity for  $(\ol\Sigma_0,g)$.

\begin{proposition} \label{prop:sigmabar}  
It is possible to perturb $(\overline{\Sigma}_0,g)$ by an arbitrarily small Legendrian isotopy
supported in a small neighborhood of $\overline{\Sigma}_0 \cap \{s = s_M\}$ to a Legendrian $(\overline{\Sigma}, \overline{g})$ satisfying:
\begin{enumerate}
\item $\overline{g}$ is regular to $\ol\Sigma$;  
\item $\mathcal{A}(\Sigma,g)$ is a sub-DGA of $(\mathcal{A}(\ol\Sigma, \overline{g}), \partial)$; and
\item $(\mathcal{A}(\ol\Sigma, \overline{g}), \partial)$ is a mapping cylinder DGA (as introduced in Section \ref{sec:mapcyl}) with corresponding DGA map
$$
f_{\Sigma}: \alg(\Lambda_+,g_+) \rightarrow \alg(\Sigma,g)
$$
where we use the bijections from Lemma \ref{lem:Reeb3} to identity
\[
\mathcal{A}(\ol\Sigma) = \alg(\Lambda_+) \ast \wh{\alg(\Lambda_+)}  * \alg(\Sigma).
\]
\end{enumerate}

\end{proposition}

Note that in establishing item (2) we will show that $\mathcal{A}(\Sigma,g)$ is a DGA, so this will complete the proof of Theorem \ref{prop:SigmaDGA} (1).  

\begin{proof}
According to \cite{Ekh}, the regularity of $g$ can be achieved  with a small perturbation.  
The perturbation to produce $\ol\Sigma$ can be made in a small neighborhood of the slice $s=s_M$ so that the rigid GFTs whose positive puncture is not in the slice $s=s_M$ do not change in the perturbation.
As a consequence of Lemmas \ref{lem:Reeb3} and  \ref{BarGFT},   
$(\alg(\Sigma, g), \dd)$ and  $(\alg(\Lambda_+,g_+), \dd)$ are identified with sub-DGAs of  $\big(\alg(\ol \Sigma, \overline{g}), \dd\big)$.  In particular, $(\alg(\Sigma, g), \dd)$ is a DGA.  

It remains to show that the differential on $\mathcal{A}(\overline{\Sigma}, \overline{g})$ satisfies that for any $\wh a_i\in \wh{\alg(\Lambda_+)}$,
\begin{equation}  \label{eq:diffAplusi}
\partial \widehat{a}_i = f(a_i) + a_i + \gamma_i
\end{equation}
with $\gamma_i \in I(\widehat{a}_1, \ldots, \widehat{a}_n)$ (the $2$-sided ideal generated by  $\{\widehat{a}_1, \ldots, \widehat{a}_n\} = R(\overline{\Sigma}_0; \{s_M\})$).   
Since the flow trees from (3) and (5) of Lemma \ref{BarGFT} that would define lead to the $f(a_i)+a_i$ term are transversally cut out by $g$, they persist under small enough perturbation.  Thus, to show that (\ref{eq:diffAplusi}) can be achieved it is enough to verify that there are no broken flow trees for $(\overline{\Sigma}_0, g)$ from $\widehat{a}_i$ to $d_1, \ldots, d_r \in \mathcal{A}(\Sigma) \cup \mathcal{A}(\Lambda_+)$, $r \geq 0$ with non-positive formal dimension and having their first level (i.e. the one with the puncture at $\widehat{a}_i$) in $\mathcal{T}_g(\ol\Sigma_0; \{s_M\})$.  

Suppose that $\Gamma_0$ is such a broken tree.
As shown in Lemma \ref{lem:T}, the GFTs of $\ol\Sigma_0$ with a negative puncture at $R\big(\ol\Sigma_0;  \{s_M\}\big)$ stays in $\T_g\big(\ol\Sigma_0; \{ s_M\}\big)$.
Thus
 the union of those broken pieces of $\Gamma_0$ that are contained in $\mathcal{T}_g\big(\overline{\Sigma}_0; \{s_M\}\big)$ forms a broken tree, $\Gamma'$, from $\widehat{a}_i$ to some sequence of Reeb chords $\widehat{a}_{j_1}, \ldots, \widehat{a}_{j_s}$.
Then $\Gamma_0$ is obtained from $\Gamma'$ by attaching further broken trees $\Gamma_1, \ldots, \Gamma_s$ (each necessarily with image in $\T_g\big(\ol\Sigma_0; [s_-, s_M)\big)$ or $\T_g\big(\ol\Sigma_0; (s_M, s_m]\big)$) at the negative punctures of $\Gamma'$. 
 Now, when viewed as a broken tree for $(\Lambda_+, g_+)$, $\Gamma'$ must have non-negative formal dimension, so
\[
|a_i| \geq \sum_{k=1}^s |a_{j_k}| +1 -\mu(\Gamma') \geq \sum_{k=1}^r |d_k| +1-\mu(\Gamma_0).
\]  
At the second inequality we used the non-existence of trees of negative formal dimension in $\mathcal{T}_g\big(\overline{\Sigma};[s_-, s_M)\big)$ and $\mathcal{T}_g\big(\overline{\Sigma}_0;(s_M,s_m]\big)$,
 to estimate that if $\Gamma_k$ has endpoints at $x_1, \ldots, x_p$ then $\displaystyle{|a_{j_k}| = |\widehat{a}_{j_k}| - 1 \geq \sum_{i=1}^p |x_i|- \mu(\Gamma_k)}$, and we used that $\displaystyle{\mu(\Gamma_0)=\mu(\Gamma')+ \sum_{k=1}^s \mu(\Gamma_k)}$.
This implies 
$$
|\widehat{a}_i|  \geq \sum_{k=1}^r |d_k| -\mu(\Gamma_0)+2,
$$
i.e. $\dim(\Gamma_0) \geq 1$.
\end{proof}

With the form of the DGA of  $\ol\Sigma$ established we can now show that $f_{\Sigma}$ is a DGA map.

\begin{proof}[Proof of Theorem \ref{DGAmap}]
This follows from Proposition \ref{prop:sigmabar} together with Proposition \ref{prop:DGAmap}.  
\end{proof}

\subsection{Independence of choices and invariance under compactly supported Legendrian isotopy}
 Recall from Lemma \ref{lem:diffconandcom} that when $\Sigma_1$ and $\Sigma_2$ are conical Legendrian isotopic, they become Legendrian isotopic via a {\it compactly supported} Legendrian isotopy after concatenating $\Sigma_1$ with a certain standard form cobordism from $\Lambda_+$ to $\Lambda_+$.  In this section, we establish invariance of the immersed DGA map $M_\Sigma$ under compactly supported Legendrian isotopy of $\Sigma$.

\begin{proposition}\label{thm:inva}
Suppose that $\Sigma_1$ and $\Sigma_2$, are two conical Legendrian cobordisms from $(\Lambda_-,g_-)$ to $(\Lambda_+, g_+)$ connected by a compactly supported Legendrian isotopy, and let $M_{\Sigma_1}$ and $M_{\Sigma_2}$ be associated immersed DGA maps computed with respect to a choice of metrics $g_1$ and $g_2$ and end data $h^1_{\pm}$ and $h^2_{\pm}$.  Then, $M_{\Sigma_1}$ and $M_{\Sigma_2}$ are immersed homotopic.
\end{proposition}

\begin{proof}  Let $\Sigma_t$, $1 \leq t \leq 2$, be a compactly supported Legendrian isotopy from $\Sigma_1$ to $\Sigma_2$.  Fix $K \gg 0$ so that outside of $(e^{-K},e^{K}) \times M$ the isotopy  $\Sigma_t$ remains constant and both metrics $g_1$ and $g_2$ agree with the products $g_\R \times g_{\pm}$.  We begin by adjusting $\Sigma_1$ and $\Sigma_2$ to have common end data.   

\begin{lemma}\label{invhs}
There exists a choice of end data $h_{\pm}$ such that 
\begin{enumerate}
\item $h_{\pm}(s) =s$ in $[e^{-K}, e^K]$;
\item there exists $B>0$ such that for any $1 \leq t \leq 2$ the action of any Reeb chords $c \in R(\widetilde{\Sigma}_t; [s_-,s_M))$ and $\widehat{a} \in R(\widetilde{\Sigma}_t; \{s_M\})$ satisfies
\[
\mathfrak{a}(c) < B-1 < B+1< \mathfrak{a}(\widehat{a});
\]  
\item both $(\Sigma_1,g_1)$ and $(\Sigma_2, g_2)$ are almost regular with respect to $h_{\pm}$, i.e. the metrics $g_i$ are almost regular when $\widetilde{\Sigma}_i$ is constructed using the end data $h_\pm$; and
\item the immersed DGA maps $M_{\Sigma_1}$ and $M_{\Sigma_2}$ are unchanged when the end data $h^1_{\pm}$ and $h^2_{\pm}$ are replaced with $h_{\pm}$.
\end{enumerate}
\end{lemma}

\begin{proof}[Proof of Lemma \ref{invhs}] 
By the construction in Section \ref{sec:wtSigma}, one can construct $h_+(s)$ so that its maximum $s_M$ is large enough (bigger than the maxima of $h^1_+$ and $h^2_+$) so that $(1)$ and $(2)$ are satisfied.

To prove the last two statements, we consider first the effect of changing the positive end data.
For $\Sigma=\Sigma_k$ with $k\in\{1,2\}$ fixed, let $\wt\Sigma^i$, for $i=1,2$ be two Morse cobordisms constructed from $\Sigma$ using different positive end data $h_{i}$ having their maxima at $s^i_M$ and satisfying $h_1(s) = h_2(s)$ for $s \leq e^{T}$ where the metric $g_k$ has $g_k = g_\R\times g_+$ on $[e^T, +\infty)\times M$   with the maximum at $s^i_M$. 
We are left to prove that the flow lines of $\wt\Sigma^i$ on $(e^T,s^i_M)\times  M$ for $i=1,2$ are in bijection.

Observe from the construction of $\wt\Sigma^i$ that the gradient vector of a local difference function $F_{i,j}(s,x)=h_i(s)f_{i,j}(x)$  of $\wt\Sigma^i$ on $ [e^T,s^i_M]\times M$ is given by
$$\big(\ h'_i(s) f_{i,j}(x), \ h_i(s) \nabla f_{i,j}(x)\ \big),$$
where $f_{i,j}(x)$ is a local difference function of $\Lambda_+$.
These vectors can be directed by
 $$X_i=\left(\ \displaystyle{\frac{h'_i(s)}{h_i(s)} f_{i,j}(x)},\   \nabla f(x)\ \right).$$
To get the bijection between $X_i$, we want to find an invertible map
 $$\begin{array} {rcl}
\Psi: [e^T,s^1_M]\times M & \to & [e^T,s^2_M]\times M\\
 (s, x) & \to & (t, x),
 \end{array}
 $$
 where $t=\psi(s)$ such that $\Psi_* X_1=X_2$. 
 In other words, we need to solve the differential equation $$\displaystyle{\frac{h'_1(s)}{h_1(s)} \frac{dt}{ds}= \frac{h'_2(t)}{h_2(t)}}$$
to get $t=\psi(s)$ with the initial condition $\psi(e^T)=e^T$.
Thus $$\displaystyle{\int_{e^T}^{s} \frac{h_1(p)}{h'_1(p)} \ dp=  \int_{e^T}^{\psi(s)} \frac{h_2(q)}{h'_2(q)} \ dq}.$$ 
It follows from the fact that $h'_i(s)>0$ and $h_i(s)>0$ for $s \in  [e^T,s^i_M)$
that the anti-derivative functions $\displaystyle{\int_{e^T}^{s} \frac{h_i(p)}{h'_i(p)} \ dp}$ are well-defined strictly increasing functions and thus are invertible.
Therefore the above differential equation can be solved and is invertible.
Moreover, using the local formula of $h_i(s)$ near $s_M^i$, one can get $\psi(s^1_M)=s^2_M$.

The effect of change of $h_-(s)$ can be studied following  a similar argument as above.
\end{proof}

Using the common end data $h_\pm$ as in Lemma \ref{invhs}, we can now form for all $ 1 \leq t \leq 2$ the Morse cobordisms $\widetilde{\Sigma}_t$.   Moreover, we can extend $h_+$ to $\overline{h}_+$ to form $(\overline{\Sigma}_t)_0$ as in Section \ref{sec:Sigmabar} with the property that the new Reeb chords at $s_m$ all have action larger than $B$.  After perturbing $(\overline{\Sigma}_t)_0$  near $t=1$ and $t=2$ and modifying the metrics $g_i$ to $\overline{g}_i$ for $i=1,2$, we can produce a new isotopy $\overline{\Sigma}_t$ with $(\overline{\Sigma}_i, \overline{g}_i)$ as in Proposition \ref{prop:sigmabar}.   Moreover, (after possibly applying a fiber rescaling) we can apply Proposition \ref{prop:EK} to arrive at almost complex structures $J_1$ and $J_2$, regular for $\overline{\Sigma}_1$ and $\overline{\Sigma}_2$ and standard at the ends such that 
\[
\big(\mathcal{A}(\overline{\Sigma}_i, J_i), \partial\big) = \big(\mathcal{A}(\overline{\Sigma}_i, \overline{g}_i), \partial\big), \quad i=1,2.
\]
Next, observe that, for all $1 \leq t \leq 2$, the action of Reeb chords satisfies $\mathfrak{a}(b) < B-1< \mathfrak{a}(c)$  when $b$ is any Reeb chord of $\overline{\Sigma}_t$ arising from a Reeb chord in $R(\widetilde{\Sigma}_t, [s_-,s_M))$, and $c$ is any Reeb chord of $\overline{\Sigma}_t$ arising  from the critical points of $h_+$ at $s=s_M$ or $s=s_m$.  
Thus, the Proposition \ref{prop:inv2} implies that the stable tame isomorphism $\overline{\varphi}$ associated to a suitable choice of $J_t$ connecting $J_1$ and $J_2$ restricts to a stable tame isomorphism between the sub-algebras $\mathcal{A}(\Sigma_i) \subset \mathcal{A}(\overline{\Sigma}_i)$. Moreover, Lemma \ref{lem:identity2} shows that $\overline{\varphi}$ restricts to the identity on $\mathcal{A}(\Lambda_\pm)$.  
 In summary, there exists a stable tame isomorphism between two mapping cylinder DGAs for $f_{\Sigma_1}$ and $f_{\Sigma_2}$, 
\[
\overline{\varphi}: \mathcal{A}(\Lambda_+) * \widehat{\mathcal{A}(\Lambda_+)} *\mathcal{A}(\Sigma_1)*S \rightarrow  \mathcal{A}(\Lambda_+) * \widehat{\mathcal{A}(\Lambda_+)} *\mathcal{A}(\Sigma_2)*S', 
\]
 that restricts to a stable tame isomorphism $\varphi:\mathcal{A}(\Sigma_1)*S \rightarrow \mathcal{A}(\Sigma_2)*S'$, and has $\overline{\varphi}|_{\mathcal{A}(\Lambda_\pm)} = \mathit{id}_{\mathcal{A}(\Lambda_\pm)}$.  
 Then, Proposition \ref{prop:DGAhmtp} applies to produce a DGA homotopy $\varphi \circ f_{\Sigma_1} \simeq f_{\Sigma_2}$ so that the required immersed homotopy between $M_{\Sigma_1}$ and $M_{\Sigma_2}$ is given by the diagram
\[
\xymatrix{ & &  \mathcal{A}(\Sigma_1)*S \ar[dd]^{\varphi} & &  \\ \mathcal{A}(\Lambda_+) \ar[rru]^{f_{\Sigma_1}} \ar[rrd]_{f_{\Sigma_2}} & & & & \alg(\Lambda_-). \ar[llu]_{i_1} \ar[lld]^{i_2}\\  & &  \mathcal{A}(\Sigma_2)*S'  & & } 
\]  
\end{proof}

\subsection{Concatenation and invariance under conical Legendrian isotopy}\label{sec:concate}
The following proposition establishes the behavior of the immersed DGA map under concatenation of conical Legendrian cobordisms.
Since we have not yet proven the invariance of the induced immersed DGA map under conical Legendrian isotopy (only compactly supported isotopy), we consider a fixed concatenation parameter $\tau$ as introduced in Section \ref{sec:Concatenate}.  
\begin{proposition} \label{prop:concatenate} For $k=0, 1, 2$, let $\Lambda_k \subset J^1M$ be a Legendrian equipped with a regular metric $g_k$ on $M$.  
For conical Legendrian cobordisms $\Sigma_1 : \Lambda_0 \rightarrow \Lambda_1$ and $\Sigma_2:\Lambda_1 \rightarrow \Lambda_2$, 
let 
\[ 
M_1 = \big(\alg(\Lambda_1) \stackrel{f_{\Sigma_1}}{\rightarrow} \alg(\Sigma_1) \stackrel{i_1}{\hookleftarrow} \alg(\Lambda_0) \big) \quad \mbox{and} \quad
M_2 = \big(\alg(\Lambda_2) \stackrel{f_{\Sigma_2}}{\rightarrow} \alg(\Sigma_2) \stackrel{i_2}{\hookleftarrow} \alg(\Lambda_1) \big)
\]
be the immersed DGA maps associated to $\Sigma_1$ and $\Sigma_2$ (with respect to some choice of metrics). 
Then, the immersed DGA map associated to $\Sigma_2 \circ_{\tau} \Sigma_1$ is immersed homotopic to the composition $M_1 \circ M_2$ (see Definition \ref{def:Comp}).
\end{proposition}

\begin{proof}  Note that from Theorems \ref{thm:inva} and Proposition \ref{prop:Comp}, when proving the result we are free to modify any of $\Sigma_1$, $\Sigma_2$ or $\Sigma_2\circ_{\tau} \Sigma_1$ by a compactly supported Legendrian isotopy and to use any choice of regular metrics (with the correct form near the positive and negative ends of the Morse cobordisms).

Construct the Legendrian $\Sigma_2 \circ_{\tau} \Sigma_1$ as in Section \ref{sec:Concatenate}.
Then, there is an interval $(C-1,C+1) \subset \R_{>0}$ that separates the non-conical parts of $\Sigma_1$ and $\tau \odot\Sigma_2$ from one another in $J^1(\R_{>0}\times M)$, and so that, in $J^1((C-1, C+1) \times M) \cap \Sigma_2 \circ_{\tau} \Sigma_1$ agrees with $j^1(s \cdot \Lambda_2 - A_2)$ for some $A_2 \in \R$.
  Modify $\Sigma_2 \circ_{\tau} \Sigma_1$ to $\Sigma'$ by a Legendrian isotopy supported in $J^1((C-1, C+1) \times M)$ so that $\Sigma'$ appears there as $j^1(h(s) \cdot \Lambda_2 - A_2)$ where $h(s) = s$ near $C\pm 1$ and has a single non-degenerate maximum (resp. minimum) at $s= s_a$ (resp. $s= s_b$) with $C-1 < s_a < s_b < C+1$.  Then, the underlying algebra for the DGA of $\Sigma'$ has the form
\[
\mathcal{A}(\Sigma') = \mathcal{A}(\Sigma_1) * \widehat{\mathcal{A}(\Lambda_2)} * \mathcal{A}(\Sigma_2).
\]
For computing the immersed DGA map $M_{\Sigma'}$  induced 
from the associated Morse cobordism $\widetilde{\Sigma}'$ on $J^1([s_-,s_+] \times M)$, we consider a metric $g$ that agrees with $g_\R \times g_0$ near $\{s_-\}\times M$; with $g_\R\times g_1$ on $(C-1,C+1)\times M$; and with $g_\R \times g_2$ near $\{s_M\}\times M$.  By the construction of $g$, GFTs in $[s_-, s_b] \times M$ are as in $(\overline{\Sigma}_1)_0$ (from Section \ref{sec:Sigmabar}), while GFTs in $[s_b, s_M]\times M$ are as in $\widetilde{\Sigma}_2$ (from Section \ref{sec:wtSigma}).  Therefore, making a small perturbation near $s=s_a$ to arrange regularity and arguing as in the proof of Proposition \ref{prop:sigmabar} we see that the 
resulting immersed map $\Sigma'$ is  related to $M_1$ and $M_2$ 
as in (1) of Proposition \ref{lem:CompositionAlt}.  Thus, the result follows from applying (2) of Proposition \ref{lem:CompositionAlt}.  
\end{proof}

With the formula for concatenation we can now complete the remaining proofs of the theorems stated in Section \ref{sec:functorI}.
\begin{proof}[Proof of Theorems \ref{inva} and \ref{thm:con}]
Theorem \ref{thm:con} follows from Proposition \ref{prop:wd}, 
Theorem \ref{inva}, and Proposition \ref{prop:concatenate}.
  To prove Theorem \ref{inva}, by Lemma \ref{lem:diffconandcom} and Theorem \ref{thm:inva}, we are left to check the invariance of the induced immersed DGA map under global shift and concatenation with a cobordism of the form $\Sigma_0=j^1(s\cdot \Lambda_+ +h(s))$  as in the statement of Lemma \ref{lem:diffconandcom}. 
The first part is trivial.
As to the second part, observe from the construction of $\Sigma_0$ that the immersed DGA map induced by $\Sigma_0$ is   
$$\alg(\Lambda_+) \stackrel{id}\rightarrow \alg(\Lambda_+) \stackrel{id}\hookleftarrow \alg(\Lambda_+).$$
This is because the local difference functions $F_{i,j}$ (and so also all GFTs) for $\Sigma_0$ are exactly the same as for the trivial conical cobordism $j^1(s\cdot \Lambda_+)$ for which $f_\Sigma$ was computed in the proof of Proposition \ref{prop:imfunctor}.  [Since the $h(s)$ term is simultaneously added to all sheets it cancels when we subtract.]    
By the concatenation formula given in the Proposition \ref{prop:concatenate}, we have that the immersed DGA map induced by the concatenation $\Sigma'=\Sigma_0\circ \Sigma$ is immersed homotopic to the immersed DGA map induced by $\Sigma$.
\end{proof}

\section{Augmentations from immersed Lagrangian fillings} \label{sec:aug}

 In this section we verify Theorem \ref{thm:induceaug} stated in the introduction about the invariance (up to DGA homotopy) of the set of augmentations induced by a good immersed Lagrangian filling. See Proposition \ref{prop:aug}.  We then give an example of a family of Legendrian knots with augmentations that can be induced by immersed fillings but not by any  oriented embedded filling.  Along the way, we construct good Lagrangian cobordisms with a single double point associated to a clasp move in the front projection of a Legendrian knot.

\subsection{Induced augmentation sets}
Recall that when $(\mathcal{A}, \partial)$ is a $\Z/n$-graded DGA and $\rho$ is a divisor of $n$, a {\bf $\rho$-graded augmentation} is an unital algebra homomorphism $\epsilon:\mathcal{A} \rightarrow \Z/2$ satisfying $\epsilon(1) = 1$, $\epsilon \circ \partial = 0$, and $\epsilon(a) = 0$ when $|a| \neq 0 \, \mbox{mod}\, \rho$.  Equivalently, $\epsilon$ is a unital homomorphism of $\Z/\rho$-graded DGAs from $(\mathcal{A}^\rho, \partial)$ to $(\Z/2,0)$ where $\mathcal{A}^\rho$ denotes $\mathcal{A}$ with the grading collapsed mod $\rho$.    
In the literature, $0$-graded augmentations (where $n=\rho =0$) are often referred to simply as {\it graded augmentations} or as {\it $\Z$-graded augmentations}.  We write $\mathit{Aug}^\rho(\mathcal{A})$ for the set of all $\rho$-graded augmentations of $\mathcal{A}$. 

Now, let $(\Lambda_+,\nu)$ be a Legendrian knot equipped with a $\Z/n$-graded Maslov potential, and let $\Sigma: \emptyset\rightarrow \Lambda_+$ be a conical Legendrian filling of $\Lambda_+$ corresponding to a good immersed Lagrangian filling $L$.  Moreover, suppose that $\Sigma$ is equipped with a $\Z/\rho$-valued Maslov potential, $\eta$, with the property that $\nu \equiv \eta|_{\Lambda_+} \, \mbox{mod}\, \rho$.  In particular, $\rho$ must divide the Maslov number $m(\Sigma)$.  Then, for any choice of almost regular metric $g$, the induced map $f_\Sigma: \mathcal{A}(\Lambda_+,g_+) \rightarrow \mathcal{A}(\Sigma,g)$ preserves $\Z/\rho$-gradings, and for any $\rho$-graded augmentation $\alpha:\mathcal{A}(\Sigma) \rightarrow \Z/2$ we have an {\bf induced augmentation} of $\Lambda_+$,
\[
\epsilon_{(\Sigma, \alpha)} = \alpha \circ f_{\Sigma}
\]
which is also $\rho$-graded.  Letting $\mathit{Aug}^\rho(\Lambda_+)/\!\!\sim$ denote the set of all $\rho$-graded augmentations of $\Lambda_+$ up to DGA homotopy we can then define the {\bf induced augmentation set} 
\[
I^\rho_\Sigma := \{[\epsilon_{(\Sigma, \alpha)}]\,|\, \alpha \in \mathit{Aug}^\rho(\mathcal{A}(\Sigma))\} \subset \mathit{Aug}^\rho(\Lambda_+)/\!\!\sim.
\]

\begin{proposition} \label{prop:aug} For any fixed regular metric $g_+$ on $\Lambda_+$, the induced augmentattion set $I^\rho_\Sigma$ is a conical Legendrian isotopy invariant of $\Sigma$ (equivalently an invariant of $L$ up to good Lagrangian isotopy).
\end{proposition}
\begin{proof}
From Proposition \ref{inva}, when $(\Sigma, g)$ and $(\Sigma',g')$ are conical Legendrian isotopic there is a stable tame isomorphism $\phi:\mathcal{A}(\Sigma) *S \rightarrow \mathcal{A}(\Sigma') *S'$ such that $ \iota' \circ f_{\Sigma'} \simeq \phi \circ \iota \circ f_{\Sigma}$ where $\iota,\iota'$ are the inclusions of each DGA into its stabilization.  (In Definition \ref{def:homotopy}, the $\iota$ and $\iota'$ are suppressed from notation.)  Then, $\pi' \circ\phi\circ \iota$ is a DGA homotopy equivalence (since all three factors are) where $\pi': \mathcal{A}(\Sigma')*S' \rightarrow \mathcal{A}(\Sigma')$ is the projection, so $\pi' \circ\phi\circ \iota$ induces a bijection between DGA homotopy classes of augmentations of $\mathcal{A}(\Sigma)$ and $\mathcal{A}(\Sigma')$.  Thus, for any augmentation $\alpha: \mathcal{A}(\Sigma) \rightarrow \Z/2$ there exists an augmentation $\alpha': \mathcal{A}(\Sigma') \rightarrow \Z/2$ such that $ \alpha \simeq \alpha' \circ \pi' \circ\phi\circ \iota$, so we can compute
\[
\epsilon_{(\Sigma, \alpha)} = \alpha \circ f_\Sigma \simeq \alpha' \circ \pi' \circ\phi\circ \iota \circ f_\Sigma \simeq \alpha' \circ \pi' \circ \iota' \circ f_{\Sigma'} = \alpha' \circ f_{\Sigma'} = \epsilon_{(\Sigma', \alpha')}.
\]
This shows that $I^\rho_\Sigma \subset I^\rho_{\Sigma'}$, and the reverse inclusion follows by symmetry.
\end{proof}

\begin{remark}
Note that when $\mathcal{A}(\Sigma)$ has a single $\rho$-graded augmentation, the set $I_\Sigma^\rho$ consists of only one element so that $\Sigma$ induces a unique augmentation of $\Lambda_+$.  This situation may occur for grading reasons, for instance, when $\rho = 0$ and $\Sigma$ has no generators in degree $0$ or $1$, as in the example below.
\end{remark}

\subsection{Examples via the Clasp Move}  Suppose $\Lambda_-$ and $\Lambda_+$ are related by a Clasp Move as pictured in Figure \ref{fig:ClaspMove}.  Moreover, suppose that $\Lambda_+$ has Maslov number $0$ and is equipped with a $\Z$-valued Maslov potential $\nu$ such that the upper and lower strands, $U$ and $L$, involved in the Clasp Move satisfy $\nu(U) =u$ and $\nu(L) = l$.

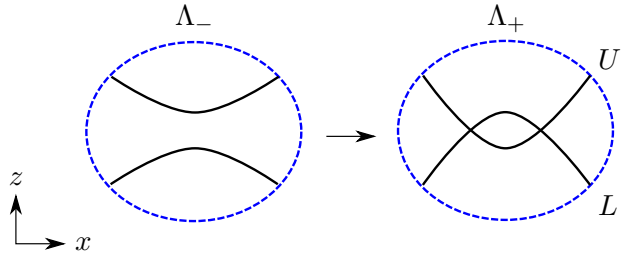
\begin{figure}[!ht]

\quad

\quad

\labellist
\small
\pinlabel $x$ [l] at 36 4
\pinlabel $z$ [b] at 4 36
\pinlabel $U$ [bl] at 366 110
\pinlabel $L$ [tl] at 366 38
\pinlabel $\Lambda_-$ [b] at 116 132
\pinlabel $\Lambda_+$ [b] at 312 132
\endlabellist
\includegraphics[scale=.6]{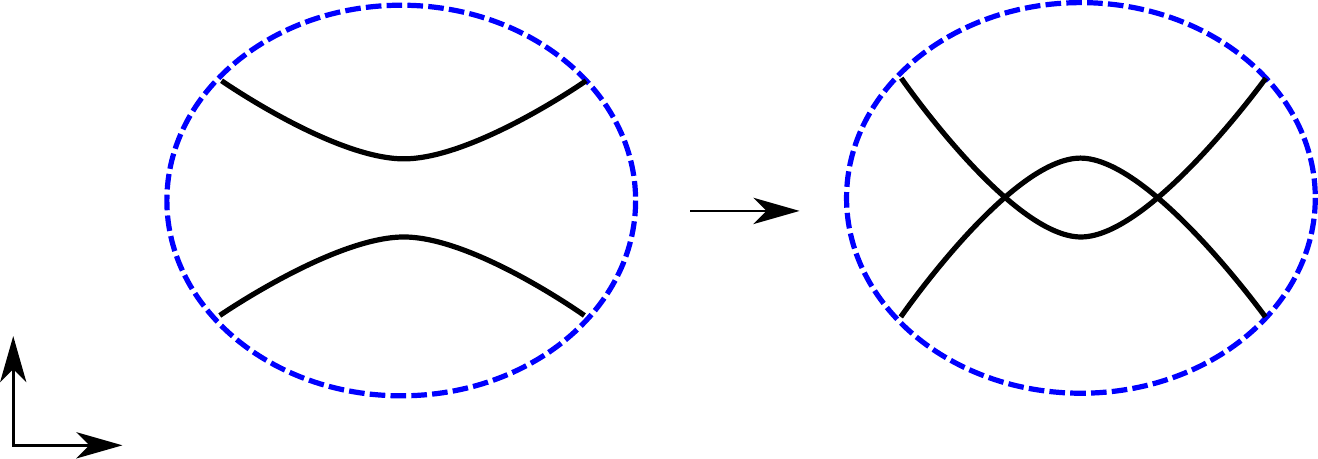}
\caption{Two Legendrians related by a Clasp Move.  The front projections of $\Lambda_-$ and $\Lambda_+$ are assumed to agree outside of the pictured disk.}
\label{fig:ClaspMove}
\end{figure}

\begin{proposition}  \label{prop:clasp}
There exists a conical Legendrian cobordism $\Sigma: \Lambda_- \rightarrow \Lambda_+$ with only one Reeb chord, $r$.  Moreover, the Maslov number of $\Sigma$ is $0$, and with respect to the unique Maslov potential on $\Sigma$ extending $\nu$ the degree of $r$ is $|r| = u-l$.  
\end{proposition} 
\begin{figure}[!ht]

\quad

\quad

\labellist
\small
\pinlabel $\emptyset$ [r] at -12 82
\pinlabel $k$~crossings  at 670 78
\endlabellist
\includegraphics[scale=.6]{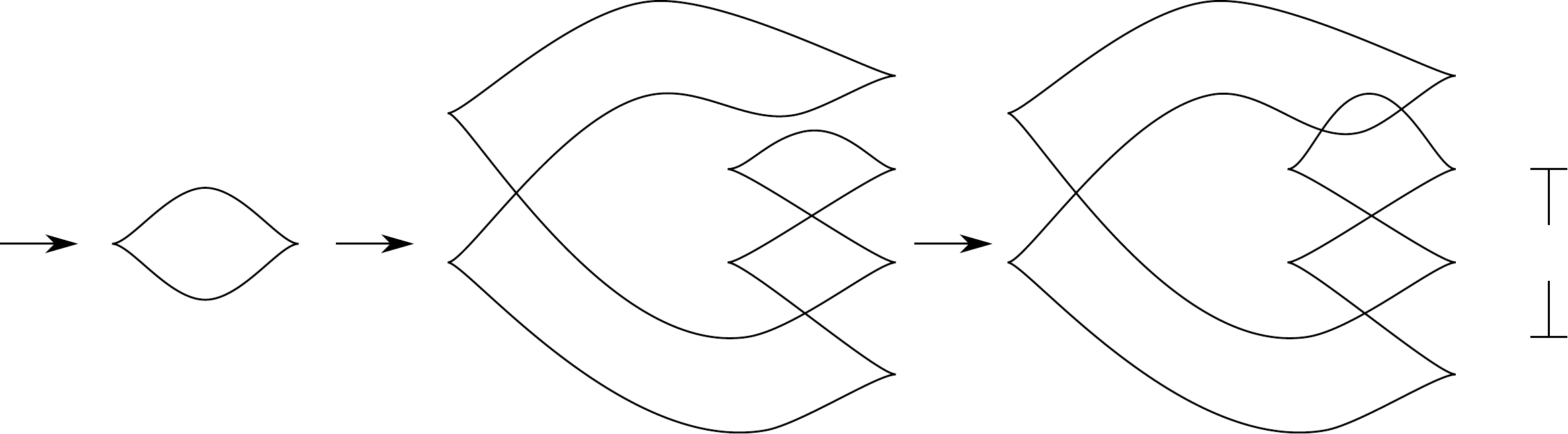}
\caption{An immersed exact Lagrangian filling of the Legendrian knot $\Lambda_k$ with one double point.  From left to the right the arrows are a standard filling of the Legendrian unknot, a Legendrian isotopy, and the clasp move.}
\label{fig:m52}
\end{figure}
\begin{example}
Consider the family of Legendrian twist knots $\Lambda_k$, $k \geq 0$, pictured at the right in Figure \ref{fig:m52};  for $k=0, 1,2$, $\Lambda_k$ is topologically a trefoil, a figure 8 knot, and a $m(5_2)$ knot.  The figure presents an immersed exact Lagrangian filling of $\Lambda_k$ obtained from concatenating a clasp move cobordism from Proposition \ref{prop:clasp} with standard embedded exact Lagrangian cobordisms (i.e. conical Legendrian cobordisms without Reeb chords) realizing a Legendrian isotopy and a filling of the unknot.  (See for instance \cite{EHK}.)  From Proposition \ref{prop:clasp} the corresponding Legendrian filling, $\Sigma_k$, has a single Reeb chord $r$ of degree $|r| = -k$.  Thus, when $k \geq 1$, $\mathcal{A}(\Sigma_k)$ has a unique $0$-graded augmentation $\alpha$ (it has $\alpha(r) =0$), and so the filling induces a unique augmentation $\epsilon = \alpha \circ f_{\Sigma_k}$ of $\mathcal{A}(\Lambda_k)$.  Moreover, a direct computation (for instance, using Ng's resolution procedure \cite{Ng} to compute the DGA of $\Lambda_k$) shows that 
$\mathcal{A}(\Lambda_k)$ 
only has one $0$-graded augmentation and 
the linearized contact homology, $\mathit{LCH}_*(\Lambda_k, \epsilon)$, 
has Poincare polynomial 
\[
P(t) = \sum_{i \in \Z} \mbox{rank}\,LCH_i(\Lambda_k, \epsilon) t^i = t^{-k} + t^1 + t^{k}.
\]
As a result, $\epsilon$ cannot be induced by any embedded exact Lagrangian filling $L$ with Maslov number $0$ since then the Siedel isomorphism would imply $LCH_i(\Lambda_k, \epsilon)  \cong H_{i+1}(L, \Lambda_k; \Z/2)$ so that $P(t) = 2g(L) + t^1$.  

In fact, when $k$ is odd, the assumption that $L$ has Maslov number $0$ is not necessary, as it can be shown that $\epsilon$ cannot be induced by any oriented, embedded  Lagrangian filling $L$.  Indeed, since $L$ is orientable the Maslov number of $L$ is even.  [This follows since a well-defined $\Z/2$-valued Maslov potential for the conical Legendrian lift $\Sigma$ is given by assigning the value $0$ (resp. $1$) to points of $\Sigma$ where the base projection $\Sigma \rightarrow \R_{>0} \times M$ preserves (resp. reverses) orientation.]   Then, a version of the Siedel isomorphism applies with gradings reduced mod $2$ to produce a contradiction.

For $k=0$, $\Lambda_0$ is the right handed Legendrian trefoil with Thurston-Bennequin number $\mathit{tb}=1$.  In this case, $|r| = 0$ so that $\mathcal{A}(\Sigma_0)$ has two $0$-graded augmentations $\alpha_0$ and $\alpha_1$  with $\alpha_j(r) = j$.  Moreover, it can be shown that the induced augmentations $\epsilon_i= \alpha_i \circ f_{\Sigma_0}$ are distinct, so that in this case the induced augmentation set, $I^0_{\Sigma_0}$, has more than one element.  See \cite{PanRu2}.      
\end{example}

\subsubsection{Construction of $\Sigma$}  In this section we construct the cobordism $\Sigma$ from Proposition \ref{prop:clasp} and verify its properties.  Note that in proving the proposition we may replace $\Lambda_-$ and $\Lambda_+$ with Legendrian isotopic links since any Legendrian isotopy can be realized by an embedded exact Lagrangian cobordism (in fact, concordance).  See \cite{EHK}.  

We start with some preliminaries.
\begin{enumerate}
\item  Choose $0< s_- < s_+$ to satisfy
\begin{equation} \label{eq:ineq1}
\frac{s_-}{s_+-s_-} < \frac{1}{4}.
\end{equation}
For instance, take $s_-=1$ and $s_+ =100$.
\item  Choose small positive constants $\alpha, \epsilon >0$ to satisfy $\alpha < \frac{s_+-s_-}{2}$ and
\begin{equation} \label{eq:ineq2}
\mbox{max}\left(\left| \frac{ \alpha}{s_+-s_-}\right|,\left| \frac{ \alpha}{s_+-s_-} + s_- \epsilon + \alpha\left( \frac{-1}{s_+-s_-}-\epsilon\right)\right|\right) < \frac{1}{8}.
\end{equation}
\item Fix a cutoff function $\tau:\R_{>0} \rightarrow [0,1]$ satisfying the following properties
\begin{itemize}
\item $\tau(s) = 1$ for $s \in (0,s_-]$, 
\item $\tau(s) = 0$ for $s \in [s_+,+\infty)$, 
\item $\tau(s)$ is linear in $[s_-+\alpha, s_+-\alpha]$,
\item $\displaystyle 0 \geq \tau'(s) > \frac{-1}{s_+-s_-}-\epsilon$,
\item $\displaystyle \tau(s) \geq \frac{s_+-s}{s_+-s_-}$ for all $s \in [s_-, s_- + \alpha]$,
\item $\displaystyle \tau(s) \leq \frac{s_+-s}{s_+-s_-}$ for all $s \in [s_+ -\alpha,s_+]$.  \vspace{0.02in}
\end{itemize}
See Figure \ref{fig:Interpolate} for the construction of such a function.
\item Fix a plateau function $h:\R \rightarrow [0,1]$ such that  
\begin{align*} 
&h(x) = 1 &  \mbox{for $x \in[-3/4,3/4]$,} \\
& h(x) =0 &  \mbox{for $x \in(-\infty,-1] \cup [1, +\infty)$.} 
\end{align*}
\item Fix small, positive $\delta, k>0$  to satisfy
\begin{equation} \label{eq:ineq4}
\frac{3 \delta}{4\|h'\|_{C^0([-1,1])}} > k.
\end{equation}
\end{enumerate}
Next, isotope $\Lambda_+$ to $\Lambda_+'$ so that in the front projection of $\Lambda_+'$ the crossings that form the clasp appear in the vertical strip  $(x,z) \in [-2, 2]\times \R$ and no other crossings or cusps appear in this strip.  By a further Legendrian isotopy, assume that in this strip the two strands where the crossings occur have local defining functions $f_a,f_b:[-2,2] \rightarrow \R$ given by
\[
f_a(x) = \delta x^2-k,  \quad \quad f_b(x) = -f_a(x) = -\delta x^2+k.
\]
Moreover, arrange that every other strand of $\Lambda_+'$ in $[-2,2]\times\R$ has a linear defining function $f_i$ with slope $m_i$ satisfying
\begin{equation} \label{eq:ineq3}
|m_i| > \sup_{(s,x) \in [s_-,s_+]\times[-2,2]} \left|2 \delta x + \tau(s) \cdot h'(x) 2k \right|.
\end{equation}

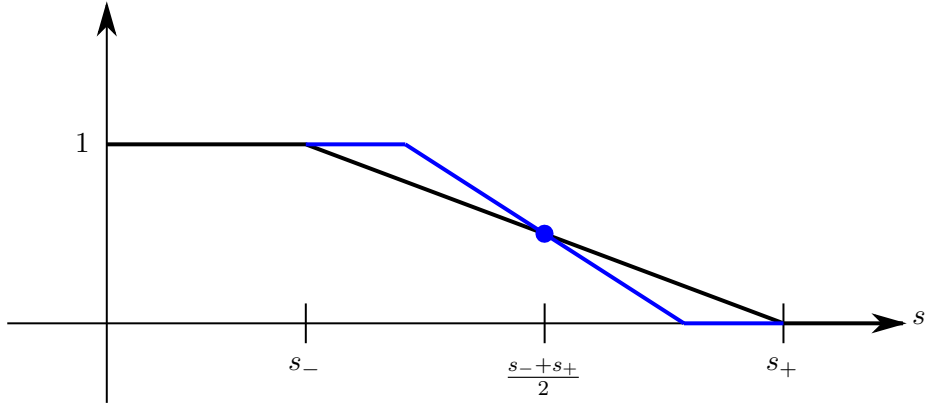
\begin{figure}[!ht]

\quad

\quad

\labellist
\small
\pinlabel $s_-$ [t] at 112 20
\pinlabel $s_+$ [t] at 292 20
\pinlabel $\frac{s_-+s_+}{2}$ [t] at 202 20
\pinlabel $1$ [r] at 34 98
\pinlabel $s$ [l] at 338 32
\endlabellist
\includegraphics[scale=1]{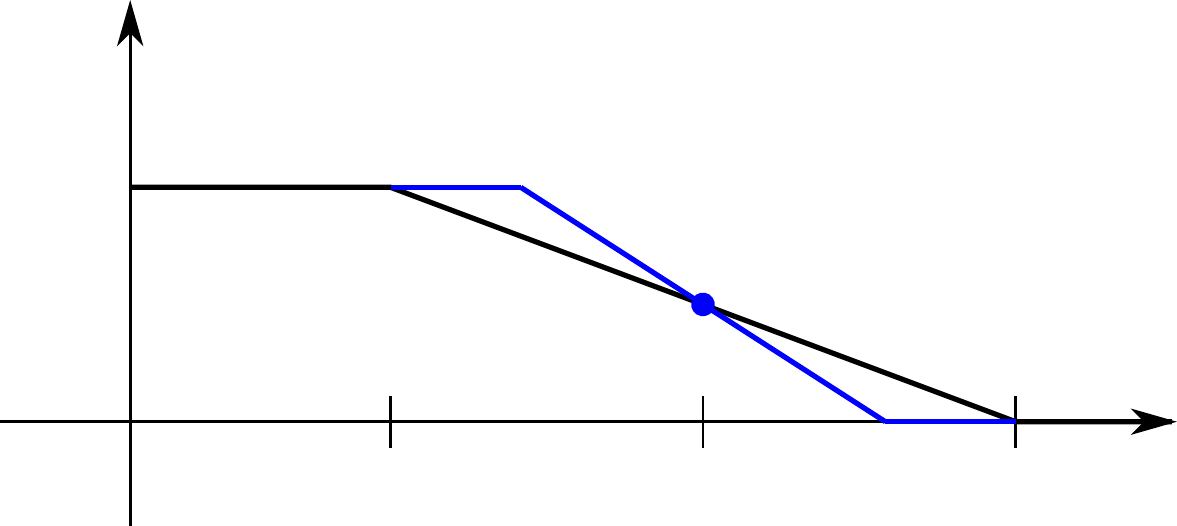}
\caption{To produce $\tau(s)$ with the required properties, start with the piecewise linear function (pictured in black) that is $0$ on $(0, s_-]$, is $1$ on $[s_+,+\infty)$, and is $\displaystyle \frac{s_+-s}{s_+-s_-}$ on $[s_-,s_+]$.  Produce a new piecewise linear function (pictured in blue) by modifying the middle segment by a small positive rotation around its mid-point.  Finally, obtain $\tau$ by rounding the corners in a manner such that the magnitude of the slope is largest in the middle linear section.  Since the rotation can be arbitrarily small, for any $\alpha, \epsilon >0$, the construction can be carried out so that $\tau$ is linear in $[s_-+\alpha, s_+ -\alpha]$ and so that the slope is never larger than $\frac{-1}{s_+-s_-} - \epsilon$ as required.}
\label{fig:Interpolate}
\end{figure}

Now define $\Sigma$ to by starting with the trivial cobordism $j^1(s\cdot \Lambda_+)$,  
removing the part of $j^1(s\cdot \Lambda_+)$ corresponding to $z= s \cdot f_a(x)$ and $z=s \cdot f_b(x)$ with $(s,x) \in \R_{>0} \times [-2,2]$, and replacing it with the Legendrian subset defined by functions   $F_a, F_b:\R_{>0} \times[-2,2] \rightarrow \R$ where
\begin{align*}
F_a(s,x) &= s\cdot \left( \delta x^2-k + \tau(s) \cdot h(x) \cdot 2k \right) \\
F_b(s,x) & = -F_a(s,x).
\end{align*}
Observe that at the ends $\Sigma$ agrees with $j^1(s\cdot \Lambda_-')$ when $s \leq s_-$ and $j^1(s\cdot \Lambda_+')$ when $s \geq s_+$ where $\Lambda_-'$ has defining functions
\[
f^-_a(x) = \delta x^2 -k+h(x)\cdot 2k, \quad f^-_b = -f^-_a. 
\]

\medskip

\begin{lemma}  \label{lem:partialx}
For all $(s,x) \in \R_{>0} \times [-2,2]$ 
\[
\sgn \left(\partial_x(F_a-F_b)\right)  = \sgn (x).
\]
\end{lemma}
\begin{proof}
Compute 
\[
\partial_xF_a(x,s) = s \cdot \left( 2 \delta x + h'(x)\tau(s) 2k\right).
\]
When $x \notin [-1,-3/4] \cup[3/4,1]$, we have $h'(x) =0$ so that we can compute
\[
\sgn \left(\partial_xF_a(x,s) \right) = \sgn( s  2 \delta x ) = \sgn(x).
\] 
When $x \in  [-1,-3/4] \cup[3/4,1]$, estimate
\begin{align*}
|\partial_xF_a(x,s)| & \geq s \cdot \left( |2 \delta x| - |h'(x)\tau(s) 2k|\right)
  \geq s\cdot \left( 2\delta\left(\frac{3}{4}\right)  - \|h'\|_{C^0([-1,1])} 2k \right) > 0
\end{align*}
where the last inequality was from  (\ref{eq:ineq4}).
\end{proof}
As a consequence, for any fixed $s$, the function $F_a(s,\cdot) - F_b(s,\cdot)$ has a unique critical point that is a local minimum at $x = 0$.  In particular, since $f^-_a(0)-f^-_b(0) = 2k > 0$ it follows that $f^-_a(x) > f^-_b(x)$ for all $x \in [-2,2]$ so that the sheets defined by $f^-_a$ and $f^-_b$ do not intersect in the front projection of $\Lambda_-'$.  Thus, $\Lambda_-'$ is indeed Legendrian isotopic to $\Lambda_-$.  
\begin{lemma}
The difference function $F_a-F_b:\R_{>0}\times[-2,2] \rightarrow \R$ only has one critical point.  This unique critical point has the form $p = (s_0,0)$ for some $s_-< s_0< s_+$; the Morse index is $\mbox{ind}(p) = 1$; and $(F_a-F_b)(p) > 0$.   
\end{lemma}
\begin{proof}  From Lemma \ref{lem:partialx}, a critical point of $F_a-F_b$ occurs at $(s,x)$ if and only if $x=0$ and $\partial_s(F_a-F_b)(s,0) = 0$.  Noting that 
\[
\partial_s(F_a-F_b)(s,0)= 2 \partial_sF_a(s,0), \quad \mbox{and} \quad \partial_sF_a(s,x) = \left(\delta x^2-k + \tau(s) \cdot h(x) \cdot 2 k\right) + s \cdot \tau'(s)\cdot h(x) \cdot 2k 
\] 
we are lead to examine the zeros of  
\[
a(s) := \partial_s F_a(s,0) = -k + 2k \left(\tau(s) + s\cdot \tau'(s)\right). 
\]

\medskip

\noindent{\bf Claim 1:}  For $s\leq s_-+\alpha$, $a(s)>0$.

\medskip
To verify, note that when $s\leq s_-$ we have $a(s) = k$ and when $s_- \leq s \leq s_-+\alpha$ we can estimate 
\begin{align*}
a(s) & \geq -k + 2k \left( \frac{s_+-s}{s_+-s_-} + s\cdot\left(\frac{-1}{s_+-s_-} - \epsilon\right) \right) \\
 & \geq -k + 2k \left( \frac{s_+-(s_-+\alpha)}{s_+-s_-} + (s_-+\alpha)\cdot\left(\frac{-1}{s_+-s_-} - \epsilon\right) \right) \\
 & = -k + 2k \left( 1-\frac{s_-}{s_+-s_-} - \frac{\alpha}{s_+-s_-} -s_-\epsilon- \alpha\left(\frac{1}{s_+-s_-} + \epsilon\right) \right) \\
& \geq -k +2k \left( 1 - 1/4 -1/8 \right) = k/4.
\end{align*}
[The first inequality used the properties of $\tau$; the last used (\ref{eq:ineq1}) and (\ref{eq:ineq2})].

\medskip

\noindent{\bf Claim 2:}  For $s\geq s_+-\alpha$, $a(s)<0$.

\medskip
To verify, note that when $s\geq s_+$ we have $a(s) = -k$ and when $s_+-\alpha \leq s \leq s_+$ we can estimate
\begin{align*}
a(s) & \leq -k + 2k \left( \frac{s_+-s}{s_+-s_-} \right) \\
 & \leq -k + 2k \left( \frac{\alpha}{s_+-s_-} \right) \leq -k+2k(1/8) = -k(3/4). \\
\end{align*}

\medskip
Next, since $\tau(s)$ is linear in $[s_-+\alpha, s_+-\alpha]$, note that $a(s)$ will be linear there as well.  Together with the claims, this shows that $a(s)$ has a unique zero $s_0 \in (s_-+ \alpha, s_+-\alpha)$.     The index computation follows since at $(s_0,0)$ we have 
\[
\frac{\partial^2F_a}{\partial x^2}= 2 \delta s_0 >0, \quad \frac{\partial^2F_a}{\partial s^2}=a'(s_0) < 0, \quad \mbox{and} \quad \frac{\partial^2F_a}{\partial x \,\partial s} =0.
\]  
Finally, to check that $(F_a-F_b)(s_0,0) = 2F_a(s_0,0) > 0$, note that $F_a(s,0) = s \cdot(-k + \tau(s) \cdot 2k)$ has $F_a(s,0) \leq 0$  only when $\tau(s) \leq 1/2$. If it were the case that $\tau(s_0) \leq 1/2$, then we would have the contradiction
\[
0=a(s_0) \leq -k+2k(1/2+s_0\tau'(s_0)) = 2ks_0 \tau'(s_0) < 0.
\]
\end{proof}

\begin{proof}[Proof of Proposition \ref{prop:clasp}]
We have already observed that $\Sigma:\Lambda_-' \rightarrow \Lambda_+'$ is a conical Legendrian cobordism and that $\Lambda_-'$ and $\Lambda_+'$ are Legendrian isotopic to $\Lambda_-$ and $\Lambda_+$.  The Reeb chord $r$ corresponding to the critical point of $F_a-F_b$ at $(s_0,0)$ has upper sheet on $z=F_a$ where the Maslov potential $\nu$ takes the value $u$ and lower sheet on $z=F_b$ where the Maslov potential takes the value $l$.  Thus, as in (\ref{eq:Maslov}) the degree of $r$ is
\[
|r| = u-l +\mathit{ind}(s_0,0) -1 = u-l.
\]  Since $\Sigma$ agrees with $j^1(s\cdot \Lambda_+)$ outside of $j^1(\R_{>0} \times [-2,2])$, the only other possibility for Reeb chords would be as critical points of either $s\cdot f_i-F_a$ or $s\cdot f_i-F_b = s\cdot f_i +F_a$ where $f_i$ is one of the linear functions defining a sheet of $\Lambda_+'$ that appears above or below the two sheets involved in the clasp move.  These functions do not have critical points since
\begin{align*}
\left|\partial_x(s \cdot f_i \pm F_a)\right| &= s \cdot \left| m_i \pm \left(2 \delta x + \tau(s) \cdot h'(x) \cdot 2k\right) \right| \\
 & \geq s \cdot \left( |m_i| - \left|2 \delta x + \tau(s) \cdot h'(x) \cdot 2k\right| \right) > 0
\end{align*}
where the last inequality is (\ref{eq:ineq3}).
\end{proof}
\section{The SFT perspective}
\label{sec:SFTpers}

In this section, we give an alternative definition for the DGA of $\Sigma$ and the DGA map $f_{\Sigma}$ by counting holomorphic disks in $\mathit{Symp}(J^1M)$ with boundary on the good Lagrangian cobordism $L$ in Section \ref{sec:alterdis}.
This perspective does not need the construction of $\wt\Sigma$ and fits into the general picture of the symplectic field theory.
We will not prove the well-definedness and invariance theorem for these definitions here, although this should be possible along the lines of \cite{E2,EHK} with the setting for analytic results about holomorphic disks modified appropriately.  Instead, we conclude by observing that with a certain good choice of almost complex structure the correspondence between GFTs and holomorphic disks from \cite{EHK} 
can be applied to relate the SFT perspective with the earlier definitions using GFTs as in Section \ref{sec:DefM}.


\subsection{Almost complex structures}\label{sec:acs}

Let $J_{\pi}^{\pm}$ be compatible almost complex structures in $T^*M$ that are adapted and regular to $\Lambda_{\pm}$ as introduced in Section \ref{sec:diff} and agree with each other outside a compact set.
We first introduce the almost complex structure on $\mathit{Symp}(J^1M)$ that we will use for holomorphic disks with boundary on the good Lagrangian $L: \Lambda_- \to \Lambda_+$ in Section \ref{sec:holo} following  \cite[Section 2.2]{CDGG}.

An almost complex structure $J$ on $\mathit{Symp}(J^1M)=\big(\R_t \times J^1M, \omega=d(e^t \alpha) \big)$ is a {\bf cylindrical almost complex structure}
if
\begin{itemize}
\item $J$ is invariant under the action of $\R_t$;
\item $J(\partial_t)= \partial_z$ and $J(\xi)=\xi$ where $\xi = \ker \alpha$;
\item $J$ is compatible with the symplectic form $\omega$, i.e., $\omega(\cdot, J \cdot)$ is a metric on $\xi$.
\end{itemize}
Given a compatible almost complex structure $J_{\pi}$ in $T^*M$ , there is a unique cylindrical almost complex structure $J$ satisfying $d\pi\circ J=J_{\pi}\circ d\pi$,
where $\pi$ is the projection map $\mathit{Symp}(J^1M)\to T^*M$.
The almost complex structure $J$ is called the {\bf cylindrical lift of $J_{\pi}$}.

Let $J^{\pm}$ be the cylindrical lifts of $J_{\pi}^{\pm}$ in $\mathit{Symp}(J^1 M)$.
When $L \subset \mathit{Symp}(J^1M)$ is a good Lagrangian cobordism from $(\Lambda_-, J^-_{\pi})$ to $(\Lambda_+, J^+_{\pi})$, an almost complex structure $J$ on $\mathit{Symp}(J^1M)$ is {\bf admissible for $L$} if 

 \begin{itemize}
\item for some $N>0$, the restrictions of $J$ to $(N, \infty)\times J^1M$ and $(-\infty,-N)\times J^1M$ agree with $J^+$ and $J^-$ respectively;
\item  the almost complex structures $J$, $J^+$, and $J^-$ all coincide outside of 
$\R\times K$ for some compact $K \subset J^1M$;  
\item  the almost complex structure $J$ is integrable in some neighborhood of  each double point $b$ of $L$, and the two sheets of $L$ around $b$ are real-analytic.
\end{itemize}

Moreover, with a generic choice of almost complex structure $J$ that is admissible to $L$,  one can make $J$  regular to $L$, 
 i.e. all the moduli spaces that we are going to talk about  in Section \ref{sec:holo} are transversely cut out.  (See \cite[Section 3.3]{CDGG} \cite{D}.)

\subsection{Moduli spaces of holomorphic curves}\label{sec:holo}

We will consider two types of holomorphic disks.

\begin{enumerate}
\item The first type of moduli space is the one we count in the differential of the DGA of $\Lambda_{\pm}$.

Let $D_{m+1}$ be the unit disk with $m+1$ points $q, p_1, p_2,\dots, p_m$ removed from the boundary, labeled in a counter-clockwise order,  and let $A$ be an element in $H_1(\Lambda_{\pm})$.
Let $a^{\pm}_i, a^{\pm}_{j_1}, \dots, a^{\pm}_{j_m}$ be $m+1$ Reeb chords of $\Lambda_{\pm}$.
We denote $\M^A_{J^{\pm}}(a^{\pm}_i; a^{\pm}_{j_1}, \dots, a^{\pm}_{j_m})$ by  the moduli space of the $J^{\pm}-$holomorphic disks up to conformal reparametrization, (see Figure \ref{diff} part $(a)$),  
$$u:\  (D_{m+1}, \dd D_{m+1}) \to \big(\mathit{Symp}(J^1M),\R\times \Lambda_{\pm} \big)$$
such that
\begin{itemize}
\item the map $u$ is asymptotic to $(0, \infty) \times a^{\pm}_i$ in a neighborhood of $q$ in $D_{m+1}$;
\item the map $u$ is asymptotic to $(-\infty, 0)\times a^{\pm}_{j_k}$ in a neighborhood of $p_k$ in $D_{m+1}$ for $k=1,\dots, m$;
\item the homology class of the union of boundary segments of $u$ with the capping paths of the corresponding Reeb chords (see Section \ref{sec:diff}), $\ol{u}= \mbox{Image}(u(\dd D_{m+1}))\cup (-\gamma_{a^{\pm}_i}) \cup \gamma_{a^{\pm}_{j_1}}\cup \cdots \cup \gamma_{a^{\pm}_{j_m}}$, is $A$ in $H_1(\Lambda_{\pm})$.
\end{itemize}
One can check \cite{Abbas} for the detailed definition of asymptotic.
Let $\widetilde{\M}_{J^{\pm}}^A(a^{\pm}_i; a^{\pm}_{j_1}, \dots, a^{\pm}_{j_m})$ denote the quotient of $\M^A_{J^{\pm}}(a^{\pm}_i; a^{\pm}_{j_1}, \dots, a^{\pm}_{j_m})$  by the vertical translation of $\R_t$. 
With the almost complex structure $J^{\pm}$ being regular to $\Lambda_{\pm}$, when $\dim \widetilde{\M}^A_{J^{\pm}}(a^{\pm}_i; a^{\pm}_{j_1}, \dots, a^{\pm}_{j_m})=0$, a holomorphic disk $u\in \M^A_{J^{\pm}}(a^{\pm}_i; a^{\pm}_{j_1}, \dots, a^{\pm}_{j_m})$ is called {\bf rigid}.
\item The second type of moduli space has boundary on $L$ instead of $\R\times \Lambda_{\pm}$.

As stated before, the points $q, p_1, \dots, p_m$ are $m+1$ points removed from the boundary of the unit disk.
Let $c_0$ be a double point of $L$ or a Reeb chord of $\Lambda_+$ and for each $1 \leq i \leq m$ let $c_i$ be a double point of $L$ or a  Reeb chord of $\Lambda_-$. 
We denote by $\M^A_J(c_0; c_1, \dots, c_m)$ the moduli space of the $J$--holomorphic disks, up to conformal reparametrization,
$$u:\  (D_{m+1}, \dd D_{m+1}) \to \big(\mathit{Symp}(J^1M), L\big)$$
satisfying the following conditions.  (See Figure \ref{diff} part $(b)$ and $(c)$.)
\begin{itemize}
\item  
If $c_0$ is a Reeb chord of $\Lambda_+$, the map $u$ is asymptotic to $(0,\infty)\times c_0$ in a neighborhood of $q$ in $D_{m+1}$.
If $c_0$ is a double point of $L$, we have that $\displaystyle{\lim_{s\to q} u(s)=c_0}$ around $q$ and $q$ is a positive puncture.
\item For $i=1,\dots, m$, if $c_i$ is a Reeb chord of $\Lambda_-$, the map $u$ is asymptotic to $(-\infty,0)\times c_i$ in a neighborhood of $p_i$ in $D_{m+1}$.
 If $c_i$ is a double point of $L$,  we have $\displaystyle{\lim_{s\to p_i} u(s)=c_i}$ and $p_i$ is a negative puncture.
\item  the homology class of  $\ol{u}= \mbox{Image}(u(\dd D_{m+1}))\cup (-\gamma_{c_0}) \cup \gamma_{c_{1}}\cup \cdots \cup \gamma_{c_{m}}$ is $A$ in $H_1(L)$.
\end{itemize}
When $J$ is regular to $L$ and $\dim \M^A_J(c_0; c_1, \dots, c_m)=0$, a holomorphic disk $u\in \M^A_J(c_0; c_1, \dots, c_m)$ is called {\bf rigid}. 
One can find the detailed description of the formal dimension in \cite[Theorem A1]{CEL} and \cite[Section 3.3]{CDGG}.
\end{enumerate}
 
 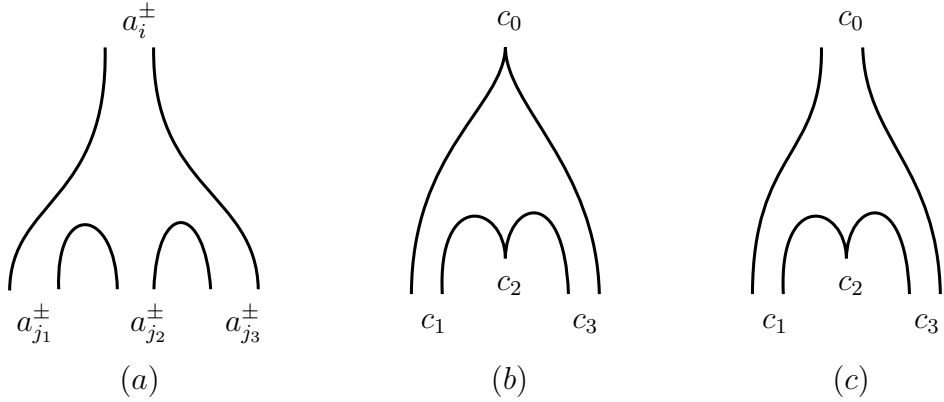
\begin{figure}[!ht]
 \labellist
 \pinlabel $a_i^{\pm}$ at 75 150
 \pinlabel $(a)$ at 75 -40
 \pinlabel $a^{\pm}_{j_1}$ at 20 -10
 \pinlabel $a^{\pm}_{j_2}$ at 80 -10
 \pinlabel $a^{\pm}_{j_3}$ at 130 -10
 \pinlabel $(b)$ at 270 -40
 \pinlabel $c_1$ at 230 -10
 \pinlabel $c_2$ at 270 10
 \pinlabel $c_3$ at 310 -10
 \pinlabel $c_0$ at 270 150
 \pinlabel $(c)$ at 450 -40
 \pinlabel $c_1$ at 410 -10
 \pinlabel $c_2$ at 450 10
 \pinlabel $c_3$ at 490 -10
 \pinlabel $c_0$ at 450 150
 \endlabellist
 \includegraphics[width=5in]{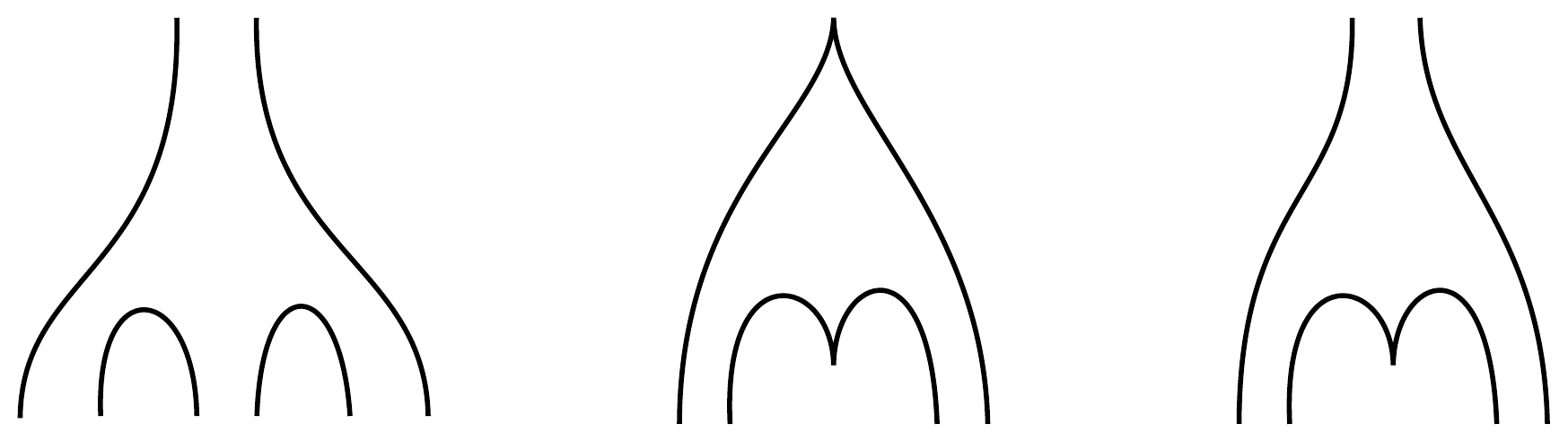}

 \vspace{0.4in}
 
 \caption{Figure $(a)$ describes the holomorphic disk in the first type of moduli space $\M^A_{J^{\pm}}(a^{\pm}_i; a^{\pm}_{j_1},  a^{\pm}_{j_2}, a^{\pm}_{j_3})$.
Figure $(b)$ and $(c)$ describe the holomorphic disk in the second type of moduli space $\M_J^A(c_0; c_1, c_2, c_3)$ for different types of positive punctures.}
 \label{diff}
 \end{figure}

\begin{remark}
The  first type of holomorphic disks gives an equivalent way of  defining the differential on the Legendrian contact homology DGA of $\Lambda_{\pm}$.  See \cite{ENS, DR}.

In the special case where $L$ has no cylindrical ends, i.e. when $L$ is compact,  all the punctures of discs of type (2) are at double points of $L$, and these holomorphic disks are the ones counted by the differential of the Legendrian contact DGA of $\Sigma^{\mathit{Symp}}$, which is the Legendrian lift of $L$ in $\mathit{Symp}(J^1M) \times \R$ as introduced in Section \ref{sec:conical}.
\end{remark}

\subsection{Alternative definitions of DGA of $\Sigma$ and the DGA map $f_{\Sigma}$}\label{sec:alterdis}

Now we can define the DGA of $\Sigma$ and the DGA map $f_{\Sigma}$ as follows.

The underlying graded algebra $\alg(\Sigma)$ is a free product of $\alg(\Lambda_-)$ and $\alg(L)$ over $\Z/2$, i.e
$$\alg(\Sigma)=\alg(\Lambda_-)\ast\alg(L).
$$

Let $L$ be a good Lagrangian cobordism from $(\Lambda_-, J^-)$ to $(\Lambda_+, J^+)$, where $J^{\pm}$ are cylindrical lifts of almost complex structures on $T^*M$ that are adapted and regular to $\Lambda_{\pm}$ as stated in Section \ref{sec:acs}.
Given an almost complex structure $J$ in $\mathit{Symp}(J^1M)$  that is admissible for and is regular to $L$.
The differential $\dd_{\Sigma}$ can be defined by counting rigid holomorphic disks as follows.
For a Reeb chord $a^-_i\in \alg(\Lambda_-)$, the differential is
$$\displaystyle{\dd_{\Sigma}(a^-_i)=\sum_{\dim \wt{\M}^A_{J^-}(a^-_i; a^-_{j_1}, \dots, a^-_{j_m})=0} | \wt{\M}_{J_-}^A(a^-_i; a^-_{j_1}, \dots, a^-_{j_m})|\ a^-_{j_1} \cdots a^-_{j_m} },$$ 
where $a^-_{j_k}$ for $k=1,\dots, m$ are Reeb chords in $\alg(\Lambda_-)$.
For a double point $c_0\in \alg(\Sigma)$, the differential is
$$\displaystyle{\dd_{\Sigma}(c_0)=\sum_{\dim \M^A_{J}(c_0; c_1, \dots, c_m)=0} | \M^A_{J}(c_0; c_1, \dots, c_m)|\ c_1 \cdots c_m },$$
 where $c_1,\dots, c_m$ are Reeb chords of $\Lambda_-$ or double points of $L$. 
The moduli spaces are as described in Section \ref{sec:holo}.
The differential can be extended to $\alg(\Sigma)$ through the Leibniz rule.

Note that the differential of $\alg(\Lambda_+)$ can be defined through counting rigid $J^+$--holomorphic disks with boundary on $\R\times \Lambda_+$ similarly as $\dd_{\Sigma}$ acting on $\alg(\Lambda_-)$.

Finally, we redefine the map $f_{\Sigma}$.
Let $a$ be a Reeb chord of $\Lambda_+$ and $c_i$, for $i=1,\dots,m $ be a double point of $L$ or a Reeb chord of $\Lambda_-$.
The map $f_{\Sigma}:\alg(\Lambda_+)\to \alg(\Sigma)$ can be defined on generators through
$$\displaystyle{f_{\Sigma}(a)=\sum_{\dim\M^A_J(a; c_1,\dots,c_m)=0}} |\M^A_J(a; c_1,\dots,c_m)| \ c_1\cdots c_m$$
and extend to $\alg(\Lambda_+)$ through the Leibniz Rule.

\begin{proposition}
When $L$ is an embedded exact Lagrangian cobordism, the $f_{\Sigma}$ map agrees with the DGA map induced by $L$ in \cite{EHK}.
\end{proposition}
\begin{proof} It follows directly from the definition.
\end{proof}

\begin{remark}\label{sft}
To prove the well-definedness and invariance theorem of the immersed DGA map associated to the above SFT definitions of the DGA of $\Sigma$ and the DGA map $f_{\Sigma}$ one could hope to modify and extend the arguments of \cite{E2,EES}.  This would involve the compactness/gluing results for the corresponding moduli spaces of disks for good Lagrangian cobordisms from Section \ref{sec:holo}.  As the disks involved have punctures both at infinity and at double points of $L$, this would rely on a synthesis of the compactness and gluing results of relative  symplectic field theory \cite{Abbas, BEHWZ} and Legendrian contact homology \cite{EES}.
As the details of these results are lengthy and scattered throughout the literature, we do not attempt this here.  An alternate approach, at least for a restricted class of almost complex structures, is to relate the moduli spaces from Section \ref{sec:holo} with the GFTs used in the definitions from Section \ref{sec:functorI}.  
The correspondence between holomorphic disks and GFTs is first introduced in the case of closed Legendrians in \cite{Ekh} and then in the case of embedded exact Lagrangian cobordisms in \cite{EHK}.
In particular, after an appropriate choice of almost complex structure and a fiber rescaling, the two definitions of the $f_\Sigma$ map  using holomorphic disks and GFTs match with each other.
Furthermore, the exposition in \cite{EHK} indicates that 
the correspondence constructed there applies also to the case of 
 immersed exact Lagrangian cobordisms.
Thus, for certain choices of almost complex structure, the definitions from the SFT perspective presented in this section agree with the definitions from Section \ref{sec:functorI}  using GFTs.


 \end{remark}

\bibliographystyle{abbrv}
\bibliography{PR}

\end{document}